\documentclass[11pt,reqno]{amsart}
\usepackage[a4paper,textwidth=152mm,textheight=235mm]{geometry}
\usepackage[T1]{fontenc}
\usepackage{lmodern,microtype,amsmath,amssymb,amsthm,mathtools,booktabs,array,longtable}
\usepackage{xcolor,listings,fancyhdr,tikz}
\usetikzlibrary{arrows.meta,positioning,calc}
\newif\ifreadingnotes
\readingnotestrue
\ifdefined\HideReadingNotes\readingnotesfalse\fi
\newcommand{\TermNote}[2]{\label{#1}\ifreadingnotes
 {\color{red!80!black}\textup{ (#2)}}\fi}
\newcommand{\ReadingNote}[2]{\ifreadingnotes\par\smallskip
 {\small\upshape\color{red!80!black}\noindent\textbf{Reading note: #1.} #2\par}
 \smallskip\fi}
\usepackage[hidelinks,hypertexnames=false,pdfpagelabels]{hyperref}
\hypersetup{pdftitle={The Linearity Conjecture: Parts 0, 1 and 2},
 pdfauthor={Peter Sziklai},
 pdfsubject={Reading guide, Frobenius methods and planar linearity}}
\newtheorem{theorem}{Theorem}[section]
\newtheorem{lemma}[theorem]{Lemma}
\newtheorem{proposition}[theorem]{Proposition}
\newtheorem{corollary}[theorem]{Corollary}
\theoremstyle{definition}\newtheorem*{definition}{Definition}
\theoremstyle{remark}
\newcommand{\F}{\mathbb F}
\newcommand{\PG}{\operatorname{PG}}
\newcommand{\ord}{\operatorname{ord}}
\newcommand{\adj}{\operatorname{adj}}
\newcommand{\rank}{\operatorname{rank}}
\newcommand{\fr}[2]{{{#1}^{[#2]}}}
\newcommand{\Span}{\operatorname{span}}
\newcommand{\calD}{\mathcal D}
\newcommand{\calN}{\mathcal N}
\newcolumntype{L}[1]{>{\raggedright\arraybackslash}p{#1}}
\newcommand{\Iref}[1]{\textup{I.\ref{p1:#1}}}
\newcommand{\IIref}[1]{\textup{II.\ref{p2:#1}}}
\newcommand{\Ipage}[1]{\pageref{p1:#1}}

\allowdisplaybreaks[1]
\numberwithin{equation}{section}
\newcommand{\CurrentPart}{0}
\fancypagestyle{collection}{\fancyhf{}\fancyhead[C]{\footnotesize\scshape
 The Linearity Conjecture --- Part \CurrentPart}\fancyfoot[C]{\thepage}
 }
\fancypagestyle{plain}{\fancyhf{}\fancyfoot[C]{\thepage}}
\newcommand{\StartPaper}[4]{%
 \clearpage\renewcommand{\CurrentPart}{#1}%
 \setcounter{page}{1}\setcounter{section}{0}\setcounter{subsection}{0}%
 \setcounter{theorem}{0}\setcounter{equation}{0}\setcounter{figure}{0}%
 \setcounter{table}{0}\setcounter{footnote}{0}%
 \pagestyle{collection}\thispagestyle{plain}%
 \pdfbookmark[0]{#2}{collection.part.#1}%
 \phantomsection\label{part#1:first}%
 \begin{center}\vspace*{4mm}{\large\scshape #3\par}
 \vspace{3mm}{\normalsize Peter Sziklai\par}
 \end{center}\vspace{3mm}%
 \begin{quotation}\small\noindent\textsc{Abstract.} #4\end{quotation}\medskip}
\newenvironment{PartContents}{\par\medskip
 \begin{center}\textsc{Contents}\end{center}\small
 \begin{list}{}{\setlength{\leftmargin}{0pt}\setlength{\rightmargin}{0pt}
 \setlength{\itemsep}{1pt}\setlength{\parsep}{0pt}}}{\end{list}\normalsize\medskip}
\newcommand{\ContentsEntry}[3]{\item\noindent
 \parbox[t]{.88\textwidth}{\hyperref[#3]{#1.\ #2}}\hfill\pageref{#3}}
\begin{document}
\StartPaper{0}{Part 0: a guide for reading Parts 1 and 2}
{The Linearity Conjecture, Part 0:\\a guide for reading Parts 1 and 2}
{This manuscript contains the proof of the Linearity Conjecture for $p^e\ne 3,4$. The text contains 3 concatenated parts. In this one, Part 0, we give a guide for reading Parts 1 and 2. We explain the common polynomial foundation and the two routes to
linearity developed in Parts 1 and 2 later. The guide identifies the purpose
of the main steps, the graph joining points on $(p^e+1)$-secants, the
data preserved while extracting coordinate powers.}
\begin{PartContents}
\ContentsEntry{1}{The question, the results, and a route through the papers}{G:route}
\ContentsEntry{2}{The polynomial vector and the common foundation}{G:foundation}
\ContentsEntry{3}{Local geometry, reconstruction, and the cubic extension}{G:local}
\ContentsEntry{4}{The short-secant graph: precisely what it proves}{G:graph}
\ContentsEntry{5}{The binary completion: why all secants enter}{G:binary}
\ContentsEntry{6}{Part 2: extracting coordinate powers and keeping the induction valid}{G:descent}
\ContentsEntry{7}{Why the case five has its own proof}{G:five}
\ContentsEntry{8}{Where each numerical margin is spent}{G:budgets}
\ContentsEntry{9}{What to keep in mind on a second reading}{G:last}
\end{PartContents}
\section{The question, the results, and a route through the papers}\label{G:route}
The object throughout is a small minimal blocking set $B$ in
$\PG(2,q)=\PG(\F_q^3)$, $q=p^h$. Small means $|B|<3(q+1)/2$; minimality implies that
that every point has a tangent. The exact intersection exponent is the
largest $e$ for which every line meets $B$ in $1\pmod{p^e}$ points. It is known that
$1\leq e|h$. 
We put $P=p^e$, write $q=P^n$, and set $|B|=q+Pd+1$.
The goal is to construct an $\F_P$-subspace $U\le\F_q^3$ with
\[
 B=L(U)=\{\langle u\rangle_{\F_q}:u\in U\setminus\{0\}\}.
\]
\label{exp:p0:linear-rank}The rank of this representation means $\dim_{\F_P}U$.
The maximum geometric field of linearity is the largest subfield
over which $B$ admits such a representation.
The existence of $U$ is the conclusion. 

The Linearity Conjecture states that for every small minimal blocking set, such a subspace $U$ exists. In this manuscript we prove it for every case when $p^e\ne 3,4$.

This research was assisted by AI tools, in parts of mathematical exploration, computations and drafting.
The origin of it was in the author's paper in 2008, a section titled "Three new curves" was about a certain curve $g$. Also, "putting together" a large-rank linear set from sublines over a subfield appears in his DSc thesis already. Still, AI assistance was needed to complete this research and it has become rather complicated. As a help, this guide was added; and also some remarks and notes were included along the text, typeset in red.
The guide explains the proofs and their dependencies; it is not an
independent certification of the full manuscripts.
The author retains responsibility for the proofs. In the near future, publication of a more streamlined, nicer version is planned.

References prefixed by I refer to Part~1; those prefixed by II refer
to Part~2. Section and page numbers restart in each part. An unprefixed
section reference in this guide refers to Part~0. All result numbers
and page references below are generated from the same LaTeX source
as the papers, so that they follow subsequent renumbering.

\label{exp:p0:frobenius-endpoint}For $n\ge3$, the \emph{Frobenius endpoint} is the target form of the polynomial matrix $M$ constructed in Part~1:
\[
 \Theta M=\mathsf C(V,V^{[P]},\ldots,V^{[P^{n-3}]}),\qquad\Theta\ne0,
\]
where $V=(X,Y,Z)^t$, $V^{[P^i]}=(X^{P^i},Y^{P^i},Z^{P^i})^t$,
\(\Theta=\Theta(X,Y,Z)\) is a nonzero homogeneous scalar polynomial, multiplying all entries of \(M\)
and $\mathsf C$ has constant coefficients and is linear in each
vector argument separately (\emph{multilinear}). It ends the
coordinate-power extraction; the geometric reconstruction follows.
A short secant means a $(P+1)$-secant, and its dual point is a short vertex.

\begin{longtable}{@{}L{.23\textwidth}L{.40\textwidth}L{.30\textwidth}@{}}
\toprule
Result & What it supplies & Where it is proved\\\midrule\endhead
Theorem~\Iref{I:size-main} & $|B|\ge q+q/P+1$, without assuming
linearity. & I, Proposition~\Iref{I:ordinary-bound},
p.~\Ipage{I:ordinary-bound}.\\
Theorem~\Iref{I:uniform-main} & The Frobenius matrix formulas
and their inverse relations for $g$, in every characteristic. & I, Sections~2--4;
Proposition~\Iref{pp:identities}, p.~\Ipage{pp:identities}.\\
Theorem~\Iref{I:binary-main} & $P=2$: $|B|=3q/2+1$ and a
representation of rank $n+1$. & I, Sections~9--17; the final
identification is in Section~17, p.~\Ipage{B:sec:closure}.\\
Theorem~\Iref{I:endpoint-theorem} & The specified Frobenius endpoint
gives compatible point representatives. & I, Section~8,
p.~\Ipage{I:endpoint-theorem}.\\
Corollary~\Iref{I:cubic-extension} & $q=P^3$, $P\ge4$: a rank-four
$\F_P$-representation, including $P=4,q=64$. & I, Section~8,
p.~\Ipage{I:cubic-extension}.\\
Theorem~\IIref{II:main} & Every exact exponent $P\ge5$: rank
$n+1$ and maximum geometric field $\F_P$. & II, Sections~2--8 for
$P\ge7$; Section~9 for $P=5$.\\\bottomrule
\end{longtable}

There are two completions after the common polynomial foundation.
For $P=2$, equality in a pencil count gives a single polynomial $D$;
its complete factorization leads to relations on every secant and then
to a dimension bound. For $P\ge5$, Part~2 successively extracts
coordinate powers from $M$. The resulting Frobenius endpoint allows Part~1,
Section~8 to construct a common rational function $W$; compatible
representatives, a connected graph, and two counts complete the proof.
\label{exp:p0:compatible}Compatibility means that on each short secant the vectors
$\F_P^*u_a$ together form the nonzero part of one two-dimensional
$\F_P$-space.

\begin{figure}[htbp]
\centering
\begin{tikzpicture}[>=Latex,box/.style={draw,rounded corners=1.5pt,
 align=center,font=\small,text width=54mm,inner sep=5pt},
 node distance=8mm and 11mm]
\node[box,text width=122mm] (g) {Part 1, Sections 2--4\\
 Arrangement $R$, root vector $g$, size bound, Frobenius matrix $M$};
\node[box,below left=8mm and -54mm of g] (binary) {Binary equality\\
 Canonical polynomial $D$};
\node[box,below right=8mm and -54mm of g] (local) {Short-secant regularity\\
 Part 1, Section 5};
\node[box,below=of binary] (factor) {Split $D$ into line factors\\
 Obtain relations on every secant};
\node[box,below=of local] (descent) {Part 2: coefficient descent\\
 Reach the Frobenius endpoint};
\node[box,below=of factor] (code) {Even subsets and triple count\\
 Part 1, Sections 16--17};
\node[box,below=of descent] (normal) {Part 1, Sections 7--8\\
 Common $W$ and compatible vectors};
\node[box,below=of code] (binend) {$P=2$: dimension $n+1$\\
 An even-set count gives $B=L(U)$};
\node[box,below=of normal] (genend) {Graph, dimension, signed pencils\\
 $P\ge5$: $B=L(U)$};
\draw[->] (g.south -| binary.north)--(binary.north);
\draw[->] (g.south -| local.north)--(local.north);
\draw[->] (binary)--(factor);\draw[->] (factor)--(code);
\draw[->] (code)--(binend);\draw[->] (local)--(descent);
\draw[->] (descent)--(normal);\draw[->] (normal)--(genend);
\end{tikzpicture}
\caption{The shared foundation and the two completion arguments.
For $q=P^3$, $P\ge4$, $M$ is already linear: the corollary in
Part~1 goes directly from the local data to its Frobenius endpoint theorem.}
\end{figure}
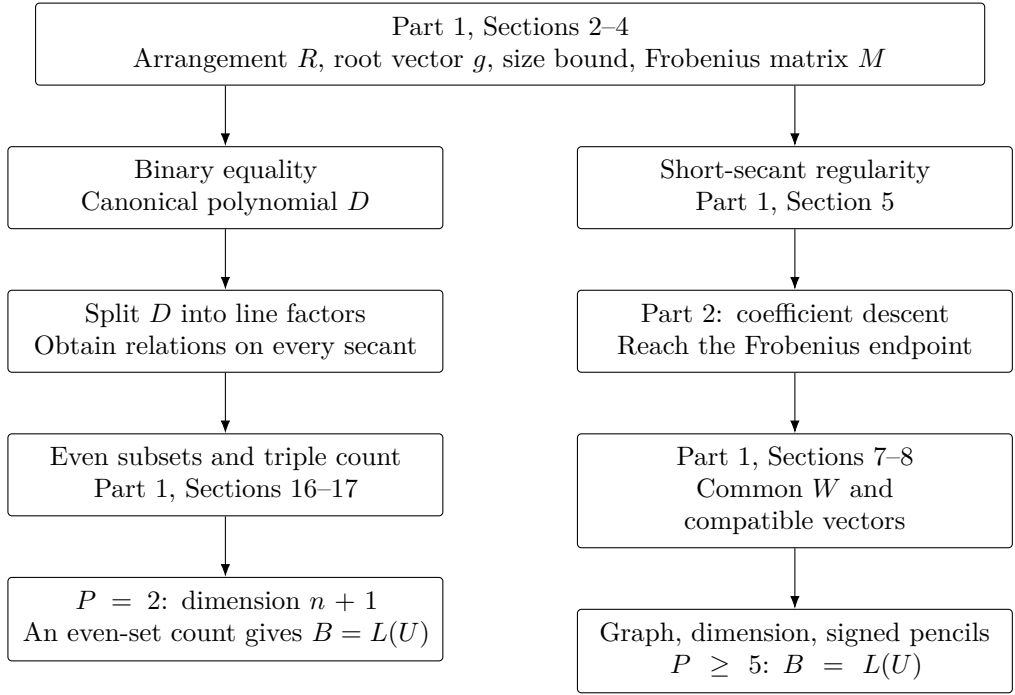

For a first reading, start with the pencil count in I.2 and the
statement of the uniform theorem. Then read I.5, the normalization
lemma I.\ref{p1:I:normalizer}, and the explanation of the graph in
Section~\ref{G:graph} below. This gives the geometric purpose of the
matrices before the longer coefficient calculations. The binary route
can then be read independently of Part~2. For Part~2, keep the definition
of admissible data in II.4 open while reading the successive steps.

\section{The polynomial vector and the common foundation}\label{G:foundation}
Choose a column $a$ for each point of $B$. Its dual line has equation
$L_a=a^tV=0$, where $V=(X,Y,Z)^t$, and $R=\prod_{a\in B}L_a$.
A secant containing $Pr+1$ points corresponds to a dual vertex where
$Pr+1$ arrangement lines meet. Its weight is $r$; a short secant has
$P+1$ points and weight one. This is why a count of secants through
$a$ becomes a count of vertices on $L_a$.

The main vector is $g=(g_1,g_2,g_3)^t$, homogeneous of degree $d$.
Brackets mean coordinatewise powers, so $g^{[P]}=(g_1^P,g_2^P,g_3^P)^t$.
In this notation $g$ is the root vector and $g^{[P]}$ the normalized
tangent vector\TermNote{exp:p0:root-tangent}{raising the coordinates of $g$ to $P$th powers gives the tangent vector; at a rational point $v$ of $R=0$ with $\nabla R(v)\ne0$, the tangent line is spanned by $v$ and $g(v)^{[P]}$}. The construction comes from the author's polynomial arguments
method, cited precisely in Part~1.

\subsection{Why the normalization matters}
Lemma~\Iref{pp:normalization} gives
\[
 R=\det(V,V^{[q]},g^{[P]}),\qquad L_a\mid a^{[1/P]t}g.
\]
The first identity comes from division on the affine grid and the
line-intersection congruences. Those congruences force appropriate
exponents to be divisible by $P$. The relation $Xf_1+Yf_2+Zf_3=0$
then pairs monomials with residues $(1,1,0)$ and produces the vector
$g$. This coefficient comparison is an important step: ordinary
first derivatives alone would detect $p$th powers, not necessarily
$P$th powers.

The second identity says that $g$ belongs to a family of polynomial
vectors with prescribed normal divisibilities\TermNote{exp:p0:normal-divisibility}{the scalar component $a^{[1/P]t}g$ vanishes identically on the dual line $L_a=0$}. Lemma~\Iref{pp:module}
gives an explicit basis of that family. In particular, a vector of
degree below the third basis degree must be a polynomial combination
of $g$ and $V^{[q/P]}$. This degree restriction is used repeatedly to
turn divisibility into an exact identity.

\label{exp:p0:ordinary-special}Write $a^{[1/P]t}g=L_a h_a$. The point $a$ and its dual line
$L_a=0$ are \emph{ordinary} if $h_a|_{L_a}$ is not the zero polynomial;
they are \emph{special} if it is zero, equivalently if
$L_a^2\mid a^{[1/P]t}g$. Thus special means higher normal vanishing:
in coordinates $L_a=Z$, the coefficient of $Z$ in $g_3$ vanishes
identically. The later results show that ordinary points lie on short
secants and special points lie on none; the latter must therefore be
recovered after the short-secant construction. The geometric normality
test uses $g^{[P]}$, since
$(a^{[1/P]t}g)^P=a^tg^{[P]}$; it is not the untwisted scalar product
$a^tg$. The distinction is relevant when coordinates are changed.

\subsection{The size theorem is obtained before the matrix machinery}
On an ordinary line $Z=0$, put $g_3=Zh$. At a vertex of weight $r$,
the coefficient of $Z^{P-1}$ in $R/Z$ has tangential order at least
$Pr-P+1$: at most $P-1$ of the $Pr$ incident factors can contribute
a normal term. The determinant expression then gives
$\ord h|_{Z=0}\ge r-1$.

If $k_a$ counts nontrivial secants through $a$, the pencil identity is
$\sum r=q/P+d$, independently of $k_a$. Comparing with
$\deg h=d-1$ yields
\[
 d-1\ge\sum(r-1)=q/P+d-k_a,
 \qquad k_a\ge q/P+1.
\]
The opposite bound $k_a\le Pd+1$ gives $d\ge q/P^2$ and hence
Theorem~\Iref{I:size-main}. Neither $J\ne0$ nor the Frobenius endpoint is
needed. This is the shortest major argument in the collection.

\subsection{From nonzero derivatives to the Frobenius representation}
The determinant
$J=\det(g,g_X,g_Y)/Z$ has degree $3d-3$. Lemma~\Iref{I:separable}
uses incidence moments to show that there are more than $d^2$ short
vertices. Two coprime combinations of the coordinates of $g$ have at
most $d^2$ common zeros. Therefore some short vertex has $g\ne0$;
its squarefree tangent cone\TermNote{exp:p0:tangent-cone}{the lowest-degree local part, here a product of distinct incident line equations} makes the two transverse first derivatives
independent, so $J\ne0$. The same argument covers $P=2$.

Section I.4 compares Hasse coefficients. A Hasse coefficient of order
$j$ is simply the coefficient of $t^j$ in $g(V+tW)$; no division by
$j!$ is involved. The determinant identities constrain the derivatives
enough to recognize an invertible matrix with coefficients in
$K^p$, where $K=\F_q(x,y)$ and $K^p$ is its subfield of $p$th powers.
Expansion in the basis $x^iy^j$, $0\le i,j<p$, is the concrete tool.

Characteristic two has a direct four-term calculation. In characteristic
three a separate argument obtains the integrating multiplier\TermNote{exp:p0:integrating}{a nonzero scalar function multiplying the vector to obtain the required matrix formula}; the
ordinary odd-characteristic computation cannot simply be divided by
three. Lemma~\Iref{pp:P-integration} then raises the power level from
$p$ to $P=p^e$. Its scalar-matrix argument includes $r=2$ only in
characteristic two. This is the reason that a proof over the prime
field is not yet the full uniform theorem.

Proposition~\Iref{pp:identities} completes Theorem~\Iref{I:uniform-main}:
\[
 Hg=M^{[P]}V,\qquad Mg=Ha_0V^{[q/P^2]},\qquad
 \det M=H^2a_0D,
\]
\[
 Dg=NV^{[q/P^2]},\qquad N=\adj M/H,\qquad
 c=H^{P-1}a_0^P,\qquad
 \Phi=J/c=H^{P-2}D^P.
\]
\label{exp:p0:H-adjugate}Here $\deg\Phi=q/P-(P-2)d-2$ and $\deg M=m$.
The adjugate $\adj M$ is the transpose of the cofactor matrix and
satisfies $M\adj M=(\det M)I$. Differentiating $Hg=M^{[P]}V$
shows that every irreducible factor of $H$ divides every cofactor
of $M$. Since $H$ is squarefree (no repeated irreducible factor),
their product $H$ divides every entry of $\adj M$. Hence
$N=\adj M/H$ has polynomial entries. We call $N$ reciprocal
because $MN=Ha_0D\,I$; dividing out $H$ lowers its degree.

\section{Local geometry, reconstruction, and the cubic extension}\label{G:local}
Section I.5 is the local input to Part~2. Its leading-block lemma
controls the first terms of $M$ at a vertex. At a short vertex,
$g,J,H,a_0,D$ and $\det M$ are nonzero, and every incident component
is ordinary. The short secant is an $\F_P$-subline.

The weight-gap lemma excludes weights $2,\ldots,P-1$. Its key
comparison using $A_r(B_r)_x-B_r(A_r)_x$ has difference exactly $P(2P-1-2r)>0$ in the
range being excluded. Thus, besides tangents, line sizes begin with
$P+1$ and then at least $P^2+1$. In particular every ordinary point
lies on at least
\[
 m_0=q/P+1-\left\lfloor\frac{d-1}{P-1}\right\rfloor
\]
short secants. This count is what supports the interpolation in
Part~2 and its final graph and dimension arguments. There are also
more than $q$ ordinary points, and
$\sigma=|B\setminus O|\le d-1-(P-3)m$ for $P\ge3$.

Section I.6 records additional information about $g$ and explains
the limitations of shortcuts. If a dual point represents a tangent
at $a$, then $g(v)^{[P]}$ represents a line through $a$, which need
not be a tangent. Likewise $g$ may have common zeros. Neither absence
of common zeros nor preservation of tangents is an input to Part~2.
The binary identities $\delta b=b^2$, $\delta c=bc$ do use
$\delta^2=0$ from Proposition~\Iref{I:extra-identities}.

\subsection{What the Frobenius endpoint theorem does}
The Frobenius endpoint is the explicit matrix formula
\[
 \Theta M=\mathsf C(V,V^{[P]},\ldots,V^{[P^{n-3}]}),
 \qquad \Theta\ne0,
\]
where $\mathsf C$ has constant coefficients and is linear in each
independent three-coordinate argument. For example, with two arguments
it is a sum $\sum_{i,j}(X_0)_i(X_1)_jC_{ij}$ of constant matrices.
Before substitution the arguments are independent; after substitution
the entries become ordinary homogeneous polynomials in $X,Y,Z$.

Theorem~\Iref{I:endpoint-theorem} proves that this formula supplies
the common rational function $W$ needed by Lemma~\Iref{I:normalizer}.
The determinant and adjugate split as
$\det\mathsf C=\mathcal F^2\mathcal G$ and
$\adj\mathsf C=\mathcal F\mathsf B$, with
\[
 D=\Theta^{P-2}\mathcal F_\Delta,
 \qquad W=\frac{\mathcal F_\Delta^P}{\Theta^PH},\quad\deg W=d-1.
\]
\label{exp:p0:diagonal}The subscript $\Delta$ means precisely the displayed Frobenius
substitution, also called the \emph{Frobenius diagonal}: replace the
independent vector arguments by these successive powers of $V$. Restriction to two-dimensional normal planes\TermNote{exp:p0:normal-plane}{for a dual vertex $v$, this is the vector plane of line normals $a$ with $a^tv=0$} shows
$h_a|_{L_a}=\beta_aW|_{L_a}$ for every ordinary $a$.
All needed numerator and denominator factors are nonzero at short
vertices, so these restrictions make sense there.

The proof compares coefficients before substitution. Argument degrees
are at most two, below $P$, so uniqueness of base-$P$ digits prevents
different monomials from cancelling on the Frobenius diagonal.
The same argument is now written explicitly for the normal planes.
This theorem is a proved implication with a stated hypothesis. It does
not, by itself, assert that all sets with $P=3$ or $P=4$ reach the
Frobenius endpoint, nor that compatible representatives already have rank $n+1$.

\subsection{Why Section I.8 now has an application inside Part 1}
For $q=P^3$, $P\ge4$, the degrees alone give
\[
 0\le\deg D=P-(P-2)m-1,\qquad
 0\le\deg H=Pm-d+1,\qquad d\ge P.
\]
Thus $m=1$ and $d=P$ or $P+1$. The Frobenius endpoint is already present:
take $\Theta=1$ and $\mathsf C(V)=M(V)$. The remaining counts are
particularly simple:
\[
 m_0=P^2,\quad |B|\le P^3+P^2+P+1<4q/3,
 \quad \sigma\le d-P+2.
\]
The graph is connected, the moment bound gives rank four, and the
signed-pencil bound identifies $B$ with its representation. The full
proof is Corollary~\Iref{I:cubic-extension}.
For $P=4$, its two possible sizes are $81$ and $85$.
For $P=3$, $m=2$ is still allowed; for $P=4$, arbitrary $n$ is still
outside this argument. Keeping the Frobenius endpoint theorem in Part~1 therefore
both preserves the shared foundation and gives it an unconditional
application in the same paper.

\section{The short-secant graph: precisely what it proves}\label{G:graph}
Let $O$ be the ordinary points. Join $a,b\in O$ when their line is
a short secant. All points of such a secant are ordinary.
Suppose the common function $W$ has supplied representatives
$u_a=\lambda_a a$ over a finite extension of $\F_q$, normalized by
$\lambda_a^{P-1}=\beta_a^P$.

Differentiating the normal identity at a short vertex gives an
invertible semilinear fixed-vector equation\TermNote{exp:p0:semilinear}{an equation $Aw^{[P]}=w$ with $A$ invertible; the fixed vectors form an $\F_P$-space} on the two-dimensional
normal plane. Its fixed space has dimension at most two over $\F_P$:
apply Frobenius to a shortest linear dependence and subtract to see
that its coefficients are in $\F_P$. The $P+1$ distinct directions
force dimension exactly two. Thus each short secant comes with all
nonzero vectors of a two-dimensional $\F_P$-space, up to the
$\F_P^*$-multiples of its chosen representatives.

If $a,b$ are adjacent, $u_a+u_b$ has the rational direction of a third
point on their secant. Since $a,b$ are independent over $\F_q$, this
forces $\lambda_a/\lambda_b\in\F_q^*$. Along a path the ratios
multiply. If the graph is connected, fix one vertex $a_0$ and divide
every representative by $\lambda_{a_0}$; all representatives then lie
in $\F_q^3$ at once. This is the role of connectivity.

If every ordinary point is on at least $m_0$ short secants, the graph
has minimum degree at least $Pm_0$. A disconnected graph would have
at least $2(Pm_0+1)$ vertices. Therefore $2Pm_0>|B|$ suffices.
The same inequality also makes the denominator in the moment bound
positive, because $N_0=|O|\le|B|$:
\[
 2+2Pm_0-N_0>0.
\]
These are two uses of one estimate, but they prove different things.

The graph does not bound the dimension of the span and does not
identify the special points. Those are the next two steps:
\[
 P^{r_U}\le
 \frac{PN_0((P-1)m_0+1)}{2+2Pm_0-N_0}
 \quad\hbox{(Proposition~\Iref{I:code-bound})},
\]
followed by the signed-pencil bound of Lemma~\Iref{I:signed}\TermNote{exp:p0:signed-pencil}{a count through a point of two disjoint sets whose line intersections are congruent modulo $P$}.
The moment estimate proves $r_U\le n+1$; the separate count of
projective directions, using $N_0>q$, proves $r_U\ge n+1$.
The written binary proof does not use this graph to recover the field:
its canonical representatives already lie in $\F_q^3$.

\section{The binary completion: why all secants enter}\label{G:binary}
The first binary simplification is $d=q/4$, from the two opposite
bounds on $d$. On each ordinary line the pencil count is then an
equality. It gives exact normal restrictions and
\[
 D|_{L_a}=\sqrt{\kappa_a}\,h_a|_{L_a},
 \qquad \deg D=q/4-1.
\]
This common polynomial replaces the rational function $W$ of the
general Frobenius endpoint argument. The canonical normalization is
\[
 L_a^2\mid u_a^{[1/2]t}g+(u_a^tV)D.
\]
Its scaling factor is $1/\kappa_a\in\F_q$, so no extension-field
descent is required. Every three-secant gives a relation
$u_a+u_b+u_c=0$.

\subsection{The global factorization is the decisive middle step}
Proposition~\Iref{B:support} detects special points by $L_a\mid D$.
The local calculation strengthens the restriction order to the full
plane order $\ord_{\ell^*}D=r-1$ at a vertex with an ordinary
incident line. Proposition~\Iref{B:indices} then counts the common
zeros of $V\times g^{[2]}$, with their local intersection counts\TermNote{exp:p0:intersection}{a common zero is counted with its intersection multiplicity, which can exceed one when the curves are tangent}.
It allows common zeros of $g$; their absence is a later conclusion.

Theorem~\Iref{B:fullsupport} proves
\[
 D\doteq\prod_{a\in\mathcal S}L_a^{\nu_a},\qquad
 \nu_a=2^{e_a}-1,\qquad \sum\nu_a=q/4-1.
\]
Here $\doteq$ means equality up to a nonzero scalar.
For concurrent special lines, the counts force full-degree
multiplicity at the concurrency point. With only one special line,
full-degree multiplicity at two distinct points on it suffices.
For nonconcurrent special lines, a residual factor of degree $h_G>0$
would contribute strictly too little to the global weight count:
B\'ezout on the special lines leaves a deficit $h_G/(\sigma-1)$.
Thus there is no residual factor. This proves both complete splitting
and the existence of an ordinary incident line at every vertex.

The special-point weights now determine all secant lengths:
$r=1+\sum_{a\in\mathcal S\cap\ell}\nu_a$.
The higher normal multiplicities in I.14 make the subsequent
interpolation possible. With
$E(t)=tB(t)^2+A(t)^2$, a squarefree polynomial of degree $2r+1$,
the ordinary representatives have coordinates
$C(t_i)^2(1,t_i)/E'(t_i)$, where $\deg C\le r-1$ and $C$ vanishes
at the special directions. Lagrange interpolation says
\[
 \sum_i\frac{F(t_i)}{E'(t_i)}=0\qquad(\deg F\le2r-1).
\]
Apply it to $C^2$ and $tC^2$. Section I.15 obtains the stronger relation
\[
 \sum_{a\in O\cap\ell}u_a=0\qquad\hbox{on every secant }\ell.
\]
This passage from short triples to all secants is essential to the
binary completion.

\subsection{The final dimension and equality arguments}
Let $\mathcal C$ be the binary vector space of subsets of $O$ meeting
every secant evenly; addition is symmetric difference. \TermNote{exp:p0:code}{A codeword is the binary indicator of a subset; a coordinate word records one $\F_2$-linear coordinate of all $u_a$} The coordinate
words of the canonical representatives belong to $\mathcal C$.
Section I.16 proves that a nonempty word has size at least
$3q/4+1-\sigma$. This improved bound uses the full splitting of $D$;
the earlier bound $q/2+2$ is insufficient for the final estimate.

Choose a generator matrix $G_{\mathcal C}$ with columns $v_a$\TermNote{exp:p0:generator}{its rows are a basis of the code; $\operatorname{wt}$ counts the nonzero coordinates of a word}.
For every binary row vector $z\in\F_2^{k_{\mathcal C}}$, put
$x(z)=N_0-2\operatorname{wt}(zG_{\mathcal C})$, where
$k_{\mathcal C}=\dim_2\mathcal C$. Average over all $2^{k_{\mathcal C}}$
vectors $z$, not merely the rows of the generator matrix.
The second and third averages count pairs and triples of columns:
$\mathbb Ex^2=N_0$, $\mathbb Ex^3=6A_3$, with $A_3$ the number
of distinct zero-sum column triples. The three-secants supply enough
such triples. The minimum word size bounds all nonzero terms in a
cubic average, while the zero word supplies the one exceptional term.
The explicit positive polynomial in I.17 gives
$2^{k_{\mathcal C}}<4q$, hence $k_{\mathcal C}\le n+1$.

The span $U$ of the canonical representatives therefore has dimension
at most $n+1$. It contains more than $q-1$ distinct nonzero vectors,
so its dimension is exactly $n+1$. For $B'=L(U)$, dimension counting
makes $B'$ blocking and all intersections odd. Since $O\subseteq B\cap B'$,
the symmetric difference has size at most $q-4$. A nonempty set
meeting every line evenly has at least $q+2$ points, by a pencil
through one of its points. Therefore $B=B'$, including all special
points. The cases $q=2,4$ are handled separately in I.17.

\subsection{A small example: \texorpdfstring{$q=8$}{q=8}}
Take the R\'edei set of $f(x)=x^2+x^4$ over $\F_8$. Its affine
graph has eight points and it determines five directions, so $|B|=13$.
The line spectrum is
\[
 1^{54},\qquad3^{16},\qquad5^3.
\]
The point $(1:1:0)$ lies on all three five-secants and on no
three-secant. It is the unique special point, with $\nu=1$.
Each of the other twelve points lies on four three-secants and one
five-secant. Thus $d=2$, $\deg D=1$, and
$(k-1)(\nu+1)=(3-1)2=4=q/2$.

The strengthened three-secant count is sharp:
$3b_3=3\cdot16=12(4+1-1)$. The final code estimate gives
$\dim_2\mathcal C\le4=n+1$. This example shows why short triples
alone omit a real part of the geometry: the special point lies on
none of them, but the final all-secant and even-set arguments recover it.

\section{Part 2: extracting coordinate powers and keeping the induction valid}\label{G:descent}
Part~2 imports the uniform representation, the ordinary-point and
short-secant bounds, the normal restrictions, and the reconstruction
results from Part~1. In particular it uses the leading-block lemma,
the weight gap, Lemmas~\Iref{I:signed}, \Iref{I:normalizer},
Proposition~\Iref{I:code-bound}, and
Theorem~\Iref{I:endpoint-theorem}. The geometric results of I.6,
the constant-space criterion and the Moore determinant criterion\TermNote{exp:p0:moore}{a determinant of the successive powers $\ell_j^{P^i}$, used to certify a linear representation} are useful
background but are not additional hypotheses of its descent.

Lines and the Baer case $n=2$ are disposed of first. For $n\ge3$ put
\[
 S=q/P^3,
 \qquad d=PS+\rho,\qquad \deg M=S+t.
\]
These are temporary degree parameters for Part~2; $q/P$ and $q/P^2$
keep their explicit notation. The four useful degrees are
\[
 \begin{array}{c|c}
 H&Pt-\rho+1\\
 a_0&S+2\rho-(P-1)t-1\\
 D&2S-(P-2)t-1\\
 N&2S+\rho-(P-2)t-1
 \end{array}
\]
For $P\ge7$, the moment bound in Part~1 gives
$0\le\rho<(1+3/P)S$ and $t\ge0$. These inequalities begin the
degree comparisons; no upper bound coming from an already assumed
linear representation is used.

\subsection{The first extraction}
The reciprocal formula expresses $Dg$ with the large coordinate
powers $V^{[q/P^2]}$. Expand a suitable common factor by exponent
residues modulo $P$. This supplies one multiplier $G$, simultaneously
for all entries of $M$, such that $GM\in(X^S,Y^S,Z^S)$ entrywise.
Here membership simply means that each monomial is divisible by one
of $X^S,Y^S,Z^S$.

The residue coefficients do not all vanish at any short vertex.
Homogenize them and choose a linear combination outside finitely many
forbidden hyperplanes; this is Lemma~\IIref{II:interpolation}(i).
One finite extension of the coefficient field permits all these choices.
The point representatives are returned to $\F_q$ at the end, as in
Section~\ref{G:graph}.

After cancelling common factors, II.3 obtains
$A=B_0M=\sum V_i^S A_i$, a reciprocal identity $AB_*=FI$, and a
primitive family $A_i$. The Hasse calculation at short vertices
shows that the matrix pencil $\sum u_iA_i$\TermNote{exp:p0:pencil}{a linear combination of fixed matrices; a primitive family has no polynomial factor common to all its entries} has common-row form after
normalization. In elementary terms it is
$\lambda(u)I+u'(u)v^t$, with one fixed row $v^t$.
Its repeated determinant factor then sharpens the residual degree to
\[
 \frac{S-1}{P+1}\le r:=\deg A_i\le\frac{S-1}{P-2}.
\]
This sharp interval starts the iteration; the rough bound $r<S/2$
alone is not the final output of the seed calculation.

\subsection{What one repetition must preserve}
At scale $s$ the matrix is
$C(Z_1,\ldots,Z_j;V)$, linear separately in its formal vector
arguments, with coefficient degree $w<5s/(P+1)$ in $V$.
The definition in II.4 retains primitivity, coefficient dimension
three in every argument, common-row form and nonsingular evaluation
of the last pencil\TermNote{exp:p0:evaluation}{evaluation is $u\mapsto(\sum u_iC_i)w$; nonsingular means its determinant is not identically zero, and coefficient dimension is $\dim\langle C_i\rangle$}, and the two identities
\[
 C_\Delta=\Xi A,\qquad
 U\mathcal D=\phi I,\qquad \phi_\Delta=\Xi^{P+1}F.
\]
The multipliers are nonzero at short vertices. The partial matrix
$U$ also has coefficient dimension three in each older argument.
Table~\IIref{tab:retained-properties} explains why each item is
needed. This is more than keeping only the equation $U\mathcal D=\phi I$.

\label{exp:p0:coefficient-frobenius}Two uses of Frobenius must be distinguished. On constant coefficients,
it leaves all polynomial variables fixed. On a whole coefficient
polynomial, $(aX+Y)^P=a^PX^P+Y^P$ also raises variable exponents.
In II.4, $C^{\langle P\rangle}$ means the latter operation on
each coefficient matrix, while the formal $Z_i$ stay unchanged.

The repeated step has four obligations:
\begin{enumerate}
\item Lemma~\IIref{lem:gap} extracts the squarefree factor $H_*$
and repeated factor $L$ from the adjugate. It gives a polynomial
residual factor $E$ with
$e=s-(P-2)w-\ell-1\ge0$ and $(P+1)w+1\ge s$.
The last-argument specialization of $L$ is nonzero because
its square divides $\det U\ne0$, even for $P=5$.
\item Lemma~\IIref{lem:cancel} removes every factor of $E$ involving
an older parameter before the final substitution. This leaves
$E_0(V)$. Lemma~\IIref{lem:extract} then simultaneously writes
$GC_I=\sum V_a^{s/P}C'_{Ia}$, with
$v=w+\deg G-s/P<5(s/P)/(P+1)$.
\item Lemma~\IIref{lem:hasse} checks that the new last pencil retains
common-row form and nonsingular evaluation. The proof takes local
Hasse coefficients, interpolates a determinant identity globally,
then separates the formal arguments. The nonsingularity check is
essential: a degenerate pencil would invalidate the next adjugate step.
\item Proposition~\IIref{II:descent-step} restores all the data,
including $\phi'_\Delta=\Xi'^{P+1}F$. That scalar identity gives
$D=E_\Delta L_\Delta(\Xi'B_0)^{P-2}$ at the next stage, ensuring
that the next multiplier can again be chosen nonzero at short vertices.
\end{enumerate}

At $s=P$ the residual degree is the integer $w=1$. The coefficients
are now linear in $V$, not constant. Regard $V$ as the final independent
argument; the coefficient matrices of this enlarged multilinear
polynomial are constant. Reorder the arguments to obtain
$V,V^{[P]},\ldots,V^{[S]}$, as required by Part~1.

\subsection{Worked extension degrees \texorpdfstring{$n=4$ and $n=5$}{n=4 and n=5}}
For $n=4$, $S=P$ and the strong interval forces $r=1$.
Write $A_i(V)=\sum_j V_j C_{ij}$ with constant matrices $C_{ij}$.
Then
\[
 \mathsf C(X_0,X_1)=\sum_i(X_1)_iA_i(X_0),\qquad
 B_0M=\mathsf C(V,V^{[P]}),\quad \deg(B_0M)=P+1.
\]
Thus $\Theta=B_0$ and no repeated descent step is needed.
For $n=5$, $S=P^2$: one repetition takes scale $P^2$ to $P$,
where the coefficients are linear. The Frobenius endpoint has three arguments
and degree $1+P+P^2$. In general there are $n-4$ repetitions after
the first extraction when $n\ge4$.

\section{Why the case five has its own proof}\label{G:five}
The proof for $P\ge7$ makes several uses of the size of $P$:
the initial separation $6S-2-3z<PS$, the rough bound $r<S/2$,
the seed and iterative interpolation budgets,
$2(P-3)\ge P+1$ in parameter cancellation, and formal determinant
degree six below $P$. It is not enough to check only the last of these.
For $P=5$, Section II.9 replaces the affected steps.

First the quadratic moment bound gives
$\rho<7S/4-1/5$. This supplies enough short vertices,
$m_0>373S/16$, for the five-case interpolation estimates.
In the first extraction a possible error is a relation among
$X^{5S},Y^{5S},Z^{5S}$. The elementary three-coordinate relation
lemma writes it as $E[V^{[5S]}]_\times$. The cross-product matrix
$[u]_\times$ is defined by $[u]_\times v=u\times v$; this is the
explicit content of the relation.

Multiply by $M$ and put $K=ME$. At short vertices the local normal
restrictions force $K+K^t$ to vanish on the transverse plane.
Interpolation and separation make $K$ skew-symmetric, so $K=[w]_\times$.
The reciprocal identity then forces $\bar D$ to divide every coordinate
of $w$, whereas $\deg(w/\bar D)<0$. Thus $w=K=E=0$.
This proves exact cancellation; an error merely vanishing at selected
points would not suffice.

Next Lemma~\IIref{five:cubic-criterion}, stated over an infinite field
of odd characteristic and used over an algebraic closure, recognizes
common-row form from a determinant of degree three. The moving-normal
plane calculation bounds its degree by $3h+15S<75S/4<m_0$.
This replaces the degree-six test.

Parameter cancellation proceeds backwards through the arguments.
The shifted reciprocal transfers nonsingular evaluation from the last
argument to the preceding one. \label{exp:p0:quadratic}Each matrix $T$ in the reciprocal pencil satisfies
$T^2-\alpha T+\beta I=0$, hence is a scalar matrix plus a matrix
of rank at most one. The common-image alternative\TermNote{exp:p0:image}{all these rank-one corrections have image in one fixed one-dimensional space} would give evaluation rank at most two,
so the common-row alternative holds. Its cubic lifts before substitution
because $3w<s$. Repeating this argument establishes the same property
in every older argument. Their repeated determinant factors then
remove parameter factors from the residual $E$.

The logical order matters: the last-argument specialization of $L$
is proved nonzero from the determinant before $E$ is constructed; only then does $e\ge0$ give
$3w<s$. After the next extraction the multiplier $G$ is cancelled
before comparing degrees. This gives
$w+5v+1\le s-2\ell-1<s$, rather than an unavailable larger bound.
Finally the proof explicitly restores $\phi'_\Delta=\Xi'^6F$.
Section II.9.6 then uses exactly the reconstruction sequence described
in Section~\ref{G:graph}.

\section{Where each numerical margin is spent}\label{G:budgets}
The table records the tests whose degree must stay below the relevant
interpolation or separation threshold. All degrees are measured before
the specified substitution. The scale $s$ is the current smallest
formal power, $\tau=s/P$, and $v$ is the new coefficient degree.
The quantity $z$ is the degree of the common factor removed from $N$;
$h_2=\deg H_2$. No table entry is an additional assumption.

\begin{longtable}{@{}L{.24\textwidth}L{.45\textwidth}L{.24\textwidth}@{}}
\toprule
Test & Degree or estimate & Threshold and use\\\midrule\endhead
Seed separation, $P\ge7$ & $6S-2-3z<PS$ & Coordinate-power
blocks in II.3.\\
Seed determinant & $3(P+1)(S+r)<\frac92(P+1)S$ & Below
$m_*=(P^2-2)S+2$; II.3.\\
Parameter cancellation & $2s-2(P-3)w-3\ell-h_*-2<s$ & Last
argument separated; II.5.\\
Iterative determinant after known factors removed &
$d_{\rm Kr}<(5P+6+3/P+3/(P-1)+4/(P-2))S$ & Below $m_*$;
at $P=7$ the coefficient is $2991/70<47$.\\
Evaluation minors, $P\ge7$ &
$\deg<(2P+P/(P-1)+P/(P-2))S+2$ & Below $m_*$;
excludes singular evaluation in II.6.\\
Return to formal arguments & Degree six in each old argument;
$6v<s$ & $6<P$ and distinct $s$-power blocks; II.6.\\
Restore the crossed identity & $\tau+(P+1)v+1<s$ & Separation
in the new old argument; II.7.\\
Five-case error & $\deg([V^{[5S]}]_\times^t(K+K^t)[V^{[5S]}]_\times)
\le12S-2$ & Below $m_0$; II.9.2.\\
Five-case output cubic & $3h+15S<75S/4$ & Below $m_0$;
II.9.3 and II.9.5.\\
Five-case evaluation minors & $5(S+r+h)+2<155S/12+1/3$ &
Below $m_0$; II.9.5.\\
Five-case return to formal arguments & Degree three in old
arguments; $3v<s$ & $3<5$; II.9.5.\\
Five-case crossed identity & $w+5v+1\le s-2\ell-1<s$ &
After cancelling $G$; II.9.5.\\\bottomrule
\end{longtable}

For reconstruction, the corresponding margins are
\[
 \begin{array}{c|c|c}
 &P\ge7&P=5\\\hline
 2Pm_0-|B|&>256q/343&>119q/200\\
 P^{r_U}&<(12/5)Pq<P^2q&<(1270/119)q<25q
 \end{array}
\]
The first row gives both connectivity and a positive denominator in
the moment bound. After rank $n+1$ is established, put $B'=L(U)$.
The comparison is always
\[
 |B\mathbin\triangle B'|
 \le\frac{Pq-1}{P-1}-|B|+2\sigma.
\]
One must not assume in advance that the first difference on the right
is nonnegative. The bounds show that the whole expression is below
$2(P-1)q/P+2$. A signed-pencil count then forces equality $B=B'$.
The same short secant also identifies the maximum geometric field:
any representation over a larger field would give a subline over
that field inside an $\F_P$-subline, which is impossible after
normalizing three common points to $0,1,\infty$.

\section{What to keep in mind on a second reading}\label{G:last}
The three short proofs to isolate first are the ordinary-pencil size
bound, the new $q=P^3$ corollary, and the graph argument. They expose
the arithmetic and geometric goals of the longer calculations.
The difficult middle steps have different jobs: I.4 builds the full
$P$-Frobenius matrix, I.13 completely splits the binary denominator,
and II.4--7 or II.9 preserve enough structure to repeat extraction.

The examples at the end of Part~2 concern the range of possible sizes.
The graph of $x\mapsto x^P$ attains $(Pq-1)/(P-1)$. The known
Polito--Polverino construction over $q=P^4$ has size
$P^4+P^3+P^2-P+1$. It shows why classifying only endpoint sizes
cannot prove the interval theorem. Neither example replaces the
descent argument for an arbitrary blocking set.

The general hypotheses should remain visible while reading:
the uniform theorem covers every exact exponent $P$;
the local results and the Frobenius endpoint theorem in I.5 and I.8 require $P\ge3$;
the Frobenius endpoint also requires $n\ge3$ and its displayed matrix formula.
The cubic-extension corollary requires $P\ge4$.
The binary special-point and code arguments requiring $q\ge8$ are
preceded or followed by the separate $q=4$ treatment.
The unrestricted $P=3$ and $P=4$ cases are not conclusions of these
two papers. All these assertions concern the plane; additional
hypotheses are needed for published higher-dimensional reductions.

\bigskip
\noindent{\bf Acknowledgements.} Project no. 153080 has been implemented with the support provided
by the Ministry of Culture and Innovation of Hungary from the National Research, Development and Innovation Fund, financed under the ADVANCED 25 funding schemes.

\label{part0:last}
\StartPaper{1}{Part 1: Frobenius methods and binary linearity}
{Part 1: Frobenius methods for small minimal blocking sets\\
and binary linearity in exponent one}
{For a nontrivial small minimal blocking set with exact intersection
exponent $P=p^e$, we prove $|B|\ge q+q/P+1$, the lower bound conjectured
by the author. The proof uses the polynomial vector $g$ and a count
of secants through one point. With the same vector we obtain
explicit polynomial expansions, determinant identities, and paired
matrix formulas involving $P$th powers, valid also in characteristics
two and three. We describe the common zeros of
$g$, its restrictions to lines, and the local multiplicities needed to
recover point representatives over $\F_P$. An application proves
$\F_P$-linearity in $\PG(2,P^3)$ for exact exponent $P\ge4$.
Finally, if $q=2^n$ and
the exact exponent is one, complete factorization and a count of
triples give $B=L(U)$ with $\dim_{\F_2}U=n+1$ and $|B|=3q/2+1$.}
\begin{PartContents}
\ContentsEntry{1}{Introduction and main results}{p1:secmap:1}
\ContentsEntry{2}{The arrangement polynomial and its line divisibilities}{p1:secmap:2}
\ContentsEntry{3}{A nonzero derivative determinant in every characteristic}{p1:secmap:3}
\ContentsEntry{4}{Directional Hasse identities and Frobenius matrix formulas}{p1:secmap:4}
\ContentsEntry{5}{Ordinary components, short secants, and the weight gap}{p1:secmap:5}
\ContentsEntry{6}{Common zeros, tangent lines, and factors of the root vector}{p1:secmap:6}
\ContentsEntry{7}{General reconstruction tools}{p1:secmap:7}
\ContentsEntry{8}{From a matrix formula to compatible point representatives}{p1:secmap:8}
\ContentsEntry{9}{The binary case: notation and exact size}{p1:secmap:9}
\ContentsEntry{10}{Sharp binary restrictions from complete splitting}{p1:secmap:10}
\ContentsEntry{11}{The canonical denominator}{p1:secmap:11}
\ContentsEntry{12}{Special factors, canonical triples, and local multiplicity}{p1:secmap:12}
\ContentsEntry{13}{The denominator splits completely}{p1:secmap:13}
\ContentsEntry{14}{Higher normal multiplicities}{p1:secmap:14}
\ContentsEntry{15}{Every secant satisfies the canonical sum relation}{p1:secmap:15}
\ContentsEntry{16}{A general code of secant relations}{p1:secmap:16}
\ContentsEntry{17}{The third moment closes the dimension bound}{p1:secmap:17}
\end{PartContents}
\section{Introduction and main results}\label{p1:secmap:1}
A blocking set in $\PG(2,q)=\PG(\F_q^3)$ meets every line. It is minimal if each of
its points has a tangent, nontrivial if it contains no line, and small if
its cardinality is less than $3(q+1)/2$. For a subspace $U\le\F_q^3$
over a subfield $\F_P$, write
\[
 L(U)=\{\langle u\rangle_{\F_q}:u\in U\setminus\{0\}\}.
\]
The planar Linearity Conjecture asserts that every small minimal blocking
set has such a representation. \label{exp:p1:rank}Its rank means $\dim_{\F_P}U$. The minimality hypothesis is essential and is retained throughout.
Such sets are called $\F_P$-linear; linear blocking sets were introduced by Lunardon \cite{p1:Lunardon}, see also \cite{p1:LunardonK}.

For $q=p^h$, the exact exponent $e$ is the largest integer for which
every line intersection is $1$ modulo $p^e$; put $P=p^e$. We use the
classical $1\bmod p$ theorem of Sz\H{o}nyi~\cite{p1:Szonyi}, the exponent
bounds of Sz\H{o}nyi--Weiner~\cite{p1:SW}, and the polynomial and subline
theory of Sziklai~\cite{p1:Sziklai,p1:DSc}. More precisely, the short-secant
subline theorem and exponent divisibility are recorded in
\cite[Theorem~11.33 and Corollary~11.35]{p1:DSc}; the classical size
intervals and their qualifications are in \cite[Corollary~11.37]{p1:DSc}.
Here the exact intersection congruence is an explicit hypothesis.

The development includes Blokhuis's prime-order bound~\cite{p1:Blokhuis},
the prime-square case in~\cite{p1:Szonyi}, and the cubic-extension results
of Polverino and Polverino--Storme~\cite{p1:Polverino,p1:PolverinoStorme}.
For directions determined by a graph, the foundational results are
\cite{p1:BBBSS,p1:Ball}; the binary case left by those results is settled
in~\cite{p1:Missing}. The present binary theorem also covers sets without
a R\'edei line.

The machinery is built on the polynomial vector $g$, starting from
the polynomial construction introduced and explored in~\cite{p1:Sziklai}.
\label{exp:p1:powers}Brackets denote coordinatewise powers: $g^{[P]}=(g_1^P,g_2^P,g_3^P)^t$.
In our notation $g^{[P]}$ is the normalized tangent vector and $g$ is
its root vector\TermNote{exp:p1:tangent-vector}{$g$ is the coordinatewise $P$th root; at a rational point $v$ of $R=0$ with $\nabla R(v)\ne0$, the tangent line is spanned by $v$ and $g(v)^{[P]}$}. We retain the needed statements and their foundational
proofs in this notation, including the small-characteristic arguments.
The derivatives, restrictions and common factors of $g$ then lead to
the matrix identities and reconstruction developed below.

\begin{theorem}[Binary exponent one]\label{p1:I:binary-main}
Let $q=2^n$ and let $B\subset\PG(2,q)$ be a nontrivial small minimal
blocking set of exact exponent one. Then
\[
 |B|=3q/2+1,\qquad B=L(U),\qquad\dim_{\F_2}U=n+1.
\]
The assertion includes blocking sets without a R\'edei line.
\end{theorem}

Ball et al's direction theorem~\cite[Theorem~1.1]{p1:BBBSS, p1:Ball} gives linearity of
the affine graph under the hypothesis $s>2$. Its $s=2$ clause supplies
direction bounds, not binary linearity. The remaining R\'edei case is
settled in~\cite[Theorem~2]{p1:Missing}. Our argument is independent of that
proof and establishes Theorem~\ref{p1:I:binary-main} also without a
R\'edei-line hypothesis.

\begin{theorem}[Size bound for the exact intersection exponent]\label{p1:I:size-main}
Let $B$ be a nontrivial small minimal blocking set of exact exponent
$e$ in $\PG(2,q)$, and put $P=p^e$. Then
\[
 |B|\ge q+q/P+1.
\]
The proof uses the arrangement polynomial and an explicit basis of
polynomial vectors satisfying its line divisibilities. It does not
assume a linear representation of $B$.
\end{theorem}

The bound is the one conjectured by Sziklai~\cite[p.~1170]{p1:Sziklai}
and restated in~\cite[Conjecture~4.3]{p1:DBVminimum}, in the small minimal
blocking-set setting of this paper. In that formulation ``maximum field
of linearity'' is defined by the exact line-intersection congruence.
The linear-set bound of~\cite[Theorem~4.1 and Remark~10]{p1:DBVminimum}
assumes linearity and an appropriate secant; the present argument
does not. Proposition~\ref{p1:I:ordinary-bound} proves the stronger local
statement from which Theorem~\ref{p1:I:size-main} follows.

We distinguish the intersection exponent from the scalar field of a
particular representing vector space. When we discuss the largest field
over which the point set admits any linear representation, we mean its
maximum \emph{geometric} field of linearity, as in~\cite[Definitions
1.3--1.4 and Remark~1.11]{p1:JenaVDV}. These notions must not be identified
merely by choosing a representation of unnecessarily large rank.

The uniform representation supplies a common foundation for this theorem
and the companion manuscripts. A polynomial vector or matrix is\label{exp:p1:matrix-terms}
\emph{primitive} if its entries have no common nonconstant factor.
A polynomial matrix is nonsingular when its determinant is not the zero
polynomial; a polynomial is squarefree when no irreducible factor is repeated.
We write $\adj M$ for the transpose of the cofactor matrix, so that
$M\adj M=(\det M)I$. These definitions also apply over a finite
extension of $\F_q$ when auxiliary coefficients are required.

\label{exp:p1:adjugate}\ReadingNote{Matrix of cofactors}{The notation $\adj M$ avoids denominators when using the inverse of a matrix: if $\det M\ne0$, then $M^{-1}=\adj M/\det M$. All its entries are ordinary polynomial determinants.}

\paragraph{Primal points and dual lines.}
A representative $a$ of a point of $B$ gives the line $L_a=a^tV=0$
in the dual plane. A primal secant containing $Pr+1$ points gives
an intersection of $Pr+1$ of these dual lines; we call it a vertex
of weight $r$, and call it \emph{short} when $r=1$.
Thus counting vertices on $L_a$ is exactly counting secants through
$a$. A component below always means one of the lines $L_a=0$.
This correspondence lets us pass between line intersections and
polynomial multiplicities without changing the underlying geometry.
\TermNote{exp:p1:normal-plane}{The normal plane at a vertex $v$ consists of all line normals $a$ with $a^tv=0$; it has vector dimension two}

\begin{theorem}[Uniform Frobenius representation]\label{p1:I:uniform-main}
Let $B$ be a nontrivial small minimal blocking set of exact exponent $e$
in $\PG(2,q)$, where $q=p^h$, $P=p^e$. Write
\[
 |B|=q+Pd+1.
\]
Then $e\mid h$, $n=h/e\ge2$, and there are a
primitive homogeneous vector $g$ of degree $d$, a primitive nonsingular
homogeneous matrix $M$ of degree $m$, a squarefree polynomial $H$, and
nonzero polynomials $a_0,D,c$, satisfying
\begin{align}
 R:=\prod_{a\in B}L_a&=\det(V,V^{[q]},g^{[P]}),
 &\delta g&=\beta g+cV^{[q/P]},\label{p1:I:main-R}\\
 Hg&=M^{[P]}V,& Mg&=Ha_0V^{[q/P^2]},\label{p1:I:main-pair}\\
 \det M&=H^2a_0D,&Dg&=NV^{[q/P^2]},\label{p1:I:main-reciprocal}\\
 N&=\adj(M)/H,& c&=H^{P-1}a_0^P.\nonumber
\end{align}
Here $V=(X,Y,Z)^t$, $L_a=a^tV$, and
\[
 \delta=g^{[P]}\cdot\nabla
 =g_1^P\frac{\partial}{\partial X}
  +g_2^P\frac{\partial}{\partial Y}
  +g_3^P\frac{\partial}{\partial Z}.
\]
The operator acts componentwise on vectors. Scalar normalizations of
$R$ are harmless.
Moreover
\begin{align}
 J:=\det(g,g_X,g_Y)/Z&=H^{2P-3}(a_0D)^P\ne0,\label{p1:I:main-J}\\
 \Phi:=J/c&=H^{P-2}D^P,\qquad
 \deg\Phi=\frac{q}{P}-(P-2)d-2.\label{p1:I:main-Phi}
\end{align}
The degree formulas are
\begin{align*}
 \deg H&=Pm-d+1,&\deg a_0&=2d-(P-1)m-\frac{q}{P^2}-1,\\
 \deg D&=\frac{q}{P^2}-(P-2)m-1,&\deg N&=d-(P-2)m-1.
\end{align*}
No odd-characteristic assumption is made.
\end{theorem}
\label{exp:p1:reciprocal}The matrix $N$ is called reciprocal because $MN=Ha_0D\,I$.
Its entries are polynomials: Proposition~\ref{p1:pp:identities} proves
that $H$ divides every cofactor of $M$, so the displayed quotient has
no denominators.

\paragraph{Proof of Theorem~\ref{p1:I:uniform-main}: organization.}
Lemmas~\ref{p1:pp:normalization}--\ref{p1:pp:module} construct $g$ and its
polynomial basis; Lemma~\ref{p1:I:separable} proves $J\ne0$.
Section~\ref{p1:I:integration} obtains a matrix formula first with
$p$th-power entries, then with $P$th-power entries in
Lemma~\ref{p1:pp:P-integration}.
Proposition~\ref{p1:pp:identities} completes the proof by homogenization
and the reciprocal identity. The distinction between $p$ and $P$ is essential.

\paragraph{Scope of the series.}
The companion paper~\cite{p1:PUB2} treats $P\ge5$. The case $P=3$ requires a separate treatment and is not claimed here.
Corollary~\ref{p1:I:cubic-extension} below treats $q=P^3$ for every
$P\ge4$, including $P=4$, $q=64$. The unrestricted case $P=4$
remains outside the completed arguments.
All statements in this series concern planar blocking sets; higher
dimensional $k$-blocking sets require the additional hypotheses in the
published reduction theorems~\cite{p1:VDV}.

\section{The arrangement polynomial and its line divisibilities}\label{p1:secmap:2}\label{p1:I:normalization}
Throughout the general construction, coefficients lie in $\F_q$.
The exact exponent $e$ is a hypothesis, not inferred from a size interval.
The exponent theorem of Sziklai gives $e\mid h$
\cite[Corollary~11.35]{p1:DSc}. Nontriviality
excludes $q=P$: a nontrivial intersection would have at least $P+1=q+1$
points, a whole line. Thus $n=h/e\ge2$.
A pencil through a point of $B$ gives $|B|\equiv1\pmod P$. An external
point on any line gives
\begin{equation}\label{p1:pp:linebound}
 |B\cap\ell|\le |B|-q=Pd+1<q,\qquad
 1\le d\le \frac{q}{2P}.
\end{equation}
The strict inequality follows directly from smallness. The arguments
below also recover exponent divisibility from the existence and local
structure of a short secant when $P\ge3$.
\begin{lemma}[Normalized dual arrangement]\label{p1:pp:normalization}
There is a primitive homogeneous vector $g\in \F_q[V]^3$ of degree $d$
such that
\begin{equation}\label{p1:pp:R}
 R=\det(V,V^{[q]},g^{[P]}).
\end{equation}
The vector $E=V\times g^{[P]}$ is primitive, and
$L_A\mid A^{[1/P]t}g$ for every $A\in B$\TermNote{exp:p1:inverse-frobenius}{$[1/P]$ takes the unique $P$th root of each coordinate in $\F_q$}.
\end{lemma}
\begin{proof}
The arrangement polynomial, following the polynomial viewpoint
of~\cite{p1:SziklaiSzonyi,p1:Sziklai}, vanishes at every rational dual point.
Division on the affine grid gives
\[
 R(x,y,1)=(x^q-x)A_0+(y^q-y)B_0,
 \qquad \deg A_0,\deg B_0\le Pd+1.
\]
The remainder has degree less than $q$ in each variable and vanishes
on $\F_q^2$, so it is zero.  Homogenize the two quotients to degree
$Pd+1$\TermNote{exp:p1:homogenize}{insert powers of $Z$ so that every monomial has that total degree}.  The restriction of $X(A_0)_h+Y(B_0)_h$ to $Z=0$ has degree
$Pd+2<q+1$ and vanishes at all its $q+1$ rational points, so it is
zero.  This yields a homogeneous vector $f$ of degree $Pd+1$ with
\[
 R=V^{[q]t}f,\qquad V^tf=0.
\]

Minimality supplies at least $|B|>q+1$ distinct primal tangents.
They cannot all be concurrent, so three nonconcurrent tangents may
be taken as coordinate lines.  For each of the $q$ external
directions $(0,y,z)$ on the first tangent, the Euler relation gives
\[
 R(X,y,z)=(X^q-X)f_1(X,y,z).
\]
All roots of the left side are rational and have multiplicity
$1\pmod P$.  The quotient is therefore a $P$th power in $X$.
For $1\le j<P$, expand $\partial_X^{[j]}f_1$ as a polynomial in $X$\TermNote{exp:p1:hasse}{$\partial_X^{[j]}F$ is the coefficient of $t^j$ in $F(X+t,Y,Z)$}.
Its coefficients are bivariate polynomials (homogeneous in $Y$ and $Z$)
of degree at most $Pd+1-j<q$. They vanish at the $q$ external
directions, and hence are zero. Thus these Hasse derivatives vanish
identically.  Every $X$ exponent in $f_1$ is divisible by
$P$: otherwise its lowest nonzero base-$p$ digit below position
$e$ supplies a nonzero Hasse derivative of order $p^a<P$.
The analogous assertions hold for the $Y$ exponents of $f_2$ and
the $Z$ exponents of $f_3$.

In $Xf_1+Yf_2+Zf_3=0$, a monomial in the $i$th summand has its
$i$th exponent $1\pmod P$.  It cannot occur in all three summands,
because its total degree is $2\pmod P$.  Every cancellation pairs
two summands and has exponent residues $1,1,0$ in a suitable order.
For example, a monomial of $Xf_1$ must pair with one in $Yf_2$
or $Zf_3$. In the first case its residues are $(1,1,0)$, and after
removing $XY$ all exponents are divisible by $P$; the other case is
the same with $Y,Z$ interchanged. Collect these paired coefficients,
with their opposite signs, and take their unique $P$th roots in
$\F_q$. This gives $f_1=Yh_3^P-Zh_2^P$ and, from the same pairs,
$f_2=Zh_1^P-Xh_3^P$, $f_3=Xh_2^P-Yh_1^P$.
Thus $f=V\times h^{[P]}$; changing the sign of $h$
gives~\eqref{p1:pp:R}.

A common factor of $g$ would have its $P$th power divide the
squarefree polynomial $R$.  A common factor of $E$ also divides
$R$, so must be an arrangement line.  Put such a line at $Z=0$.
Then $g_3=0$ and $Xg_2^P=Yg_1^P$ there.  Distinct monomial residues
force $g_1=g_2=0$ on that line, contradicting primitivity.
Finally, restriction of~\eqref{p1:pp:R} to $L_A=0$ gives the normal
divisibility, because there $V\times V^{[q]}$ is a nonzero scalar
bivariate polynomial times $A$ (homogeneous in $X$ and $Y$ after
choosing coordinates with $L_A=Z$).
\end{proof}
\label{exp:p1:ordinary-special}
\paragraph{Ordinary and special points.}
For $a\in B$, write $a^{[1/P]t}g=L_a h_a$.
The point $a$ and its dual line $L_a=0$ are \emph{ordinary} when
$h_a|_{L_a}$ is not the zero polynomial, and \emph{special} when it is zero.
Equivalently, special means $L_a^2\mid a^{[1/P]t}g$.
In coordinates $L_a=Z$, this says that the coefficient of $Z$ in
$g_3=Zh_a$ vanishes identically: its normal vanishing has order at least two.
Thus an ordinary or special ``line'' always means the dual line of
the correspondingly named point of $B$. Sections~5 and~12--13 show
that ordinary points lie on short secants and special points lie on none.

\paragraph{Changes of coordinates.}\label{p1:I:coordinates}
For $V_{\rm old}=CV$, $C\in\operatorname{GL}_3(\F_q)$, choose
\[
 g_{\rm new}=(C^{-1})^{[1/P]}g_{\rm old}(CV).
\]
Then $R_{\rm new}=\det(C)^{-1}R_{\rm old}(CV)$, an irrelevant
nonzero scalar normalization. Normal divisibilities, their orders,
ordinary and special components, and vertex weights are therefore
unchanged. Euler's identity shows that $J\ne0$ is equivalent to
$g,g_X,g_Y,g_Z$ spanning a three-dimensional space over $\F_q(V)$;
the chain rule makes this condition invariant as well. When $M$ has
been constructed, use
\[
 M_{\rm new}=(C^{-1})^{[1/P^2]}M_{\rm old}(CV)C^{[1/P]}.
\]
The scalar factors $H,a_0$ are composed with $CV$. Both Frobenius
identities are preserved, since inverse $P^2$-Frobenius on $\F_q$
is raising to the power $q/P^2$.

\label{exp:p1:local-order}\ReadingNote{Lowest-degree terms at a point}{After moving the point to the origin, the least degree of a nonzero term is the local order. The sum of terms of that degree is the initial form; for a product of lines it is the product of the incident line equations, called the tangent cone.}

At a rational vertex normalized to $(0,0,1)$, the terms of $R$ of
degree below $q$ in affine local coordinates are $yg_1^P-xg_2^P$.
By~\eqref{p1:pp:linebound}, the tangent cone has degree $Pr+1<q$.
Disjoint monomial residues show that $g_1,g_2$ have order at least
$r$ and their degree-$r$ initial forms $A_r,B_r$ satisfy
\begin{equation}\label{p1:pp:cone}
 yA_r^P-xB_r^P=
 \text{a nonzero scalar times the squarefree tangent cone}.
\end{equation}
Thus these initial forms are nonzero and coprime.  At a short
vertex they are independent linear forms.

\begin{lemma}[A basis for the normal divisibilities]\label{p1:pp:module}
The family of polynomial vectors
\[
 \mathcal M=\{h\in \F_q[V]^3:L_A\mid A^{[1/P]t}h
                     \text{ for every }A\in B\}
\]
consists exactly of the unique polynomial combinations of
\[
 g,\quad F_0=V^{[q/P]},\quad
 w=\operatorname{diag}(V)(F_0\times g)^{[P-1]},
\]
of degrees $d,\frac{q}{P},(P-1)(\frac{q}{P}+d)+1$.
For $\delta=g^{[P]}\cdot\nabla$ there are homogeneous polynomials
$\beta,c$ such that
\begin{equation}\label{p1:pp:delta}
 \delta g=\beta g+cV^{[q/P]},\qquad
 \deg\beta=Pd-1,\quad \deg c=(P+1)d-\frac{q}{P}-1,\quad c\ne0.
\end{equation}
The coefficient $\beta$ is allowed to be zero.
\end{lemma}
\begin{proof}
Modulo $L_A$, the cross product $F_0\times g$ is proportional to
$A^{[1/P]}$, so all three vectors belong to $\mathcal M$.  Their
determinant is $-R$.  Every Cramer numerator for a member of
$\mathcal M$ vanishes on each $L_A$, since its three columns lie
in the same normal plane.  Squarefreeness of $R$ makes the Cramer
coefficients polynomial.  Since $d<\frac{q}{P}$, this family has no nonzero
element of degree below $d$.

The equality $\delta L_A=(A^{[1/P]t}g)^P$ shows that $\delta$
preserves $\mathcal M$.  The degree $(P+1)d-1$ of $\delta g$ is
below the third basis degree, proving its decomposition and the
displayed degrees.  If $c=0$, each
$g_i\nabla g_j-g_j\nabla g_i$ annihilates both $V$ and $g^{[P]}$.
It is a polynomial multiple of primitive $E$, but has degree
$2d-1<Pd+1$, so it vanishes. All nonzero coordinate ratios are
$p$th powers. Comparing multiplicities of irreducible factors, using primitivity, gives
$g=u^{[p]}$ for a polynomial vector $u$. At a rational dual vertex,
all terms of degree below $q$ in
$\det(V,V^{[q]},u^{[pP]})$ have degree $1\pmod{pP}$.
The local order is below $q$ by~\eqref{p1:pp:linebound},
so every line meets $B$ in $1\pmod{pP}$ points, contrary to exactness.
Thus $c\ne0$, and its degree also gives
\begin{equation}\label{p1:I:lower-d}
 d\ge\left\lceil\frac{q+P}{P(P+1)}\right\rceil.
\end{equation}
\end{proof}

\begin{proposition}[Ordinary pencils and a uniform size lower bound]
\label{p1:I:ordinary-bound}
There is an ordinary point of $B$. At each ordinary point
the number $k_a$ of nontrivial secants satisfies
\[
 \frac{q}{P}+1\le k_a\le Pd+1.
\]
Consequently $d\ge q/P^2$ and $|B|\ge q+q/P+1$.
\end{proposition}
\begin{proof}
If every point were nonordinary, each partial derivative of $g$
would lie in the family $\mathcal M$ in degree $d-1$, and would therefore
vanish. Primitivity would make $g$ a $p$th-power vector, contrary to
exactness of the exponent as in Lemma~\ref{p1:pp:module}.

Normalize an ordinary component to $Z=0$, and write $g_3=Zh$,
where $h_0=h(X,Y,0)\ne0$. The determinant expansion gives, up to
an overall sign,
\[
 \frac RZ=(X^{q/P}g_2-Y^{q/P}g_1)^P
 +Z^{P-1}h^P(XY^q-YX^q)
 +Z^{q-1}(Yg_1^P-Xg_2^P).
\]
At a rational vertex of weight $r$, the reduced product $R/Z$
contains $Pr$ incident linear factors. Its coefficient of $Z^{P-1}$
has order at least $Pr-P+1$ along $Z=0$. Indeed, each incident
factor is $\ell_b+\lambda_bZ$, with $\ell_b$ vanishing at the vertex.
In the coefficient of $Z^{P-1}$, at most $P-1$ of these $Pr$ factors
can contribute their normal term $\lambda_bZ$.
The first term has no such coefficient, and $q-1>P-1$.
Since $XY^q-YX^q$ has a simple zero at every rational point,
$\ord h_0\ge r-1$. The pencil identity $\sum r=\frac{q}{P}+d$ implies
\[
 d-1\ge\sum(r-1)=\frac{q}{P}+d-k_a,
\]
which proves the lower bound for $k_a$.
Every such vertex is a zero of $(Xg_2^P-Yg_1^P)|_{Z=0}$,
a bivariate polynomial (homogeneous in $X$ and $Y$) of degree $Pd+1$. This form is nonzero: otherwise comparison of
monomial residues would force $g_1=g_2=0$ on $Z=0$, and together
with $g_3=0$ this would contradict primitivity. Thus $k_a\le Pd+1$.
\end{proof}

The bound $d\ge q/P^2$ will also simplify the determinant and degree
estimates below and in the companion paper.

\section{A nonzero derivative determinant in every characteristic}\label{p1:secmap:3}\label{p1:I:separability}
The determinant used throughout the proof is
\[
 J=\det(g,g_X,g_Y)/Z,\qquad\deg J=3d-3.
\]
Euler's identity shows divisibility by $Z$: on $Z=0$, if $p\nmid d$
the vector $g$ is a combination of $g_X,g_Y$, and if $p\mid d$ the
last two vectors are dependent.

\begin{lemma}\label{p1:I:separable}
For every exact exponent $P\ge2$ under consideration, $J\ne0$.
\end{lemma}
\begin{proof}
Let $n_j$ count lines meeting $B$ in $Pj+1$
points. The incidence and ordered-pair counts give
\begin{equation}\label{p1:pp:moments}
 \sum jn_j=qd+\frac{q}{P}+d,\qquad
 \sum j^2n_j=\frac{q^2}{P^2}+\frac{qd}{P}+d^2.
\end{equation}
Since $j(2-j)\le0$ for $j\ge2$,
\[
 n_1\ge\frac{(2P-1)qd}{P}-\frac{q^2}{P^2}-d^2+\frac{2q}{P}+2d.
\]
Subtract $d^2$ and call the result $F(d)$. It is increasing for
$d\le q/(2P)$. Proposition~\ref{p1:I:ordinary-bound} gives $d\ge q/P^2$, and
\[
 F(q/P^2)=\frac{(P-2)(P+1)q^2}{P^4}
                +\frac{2(P+1)q}{P^2}>0.
\]
Thus $n_1>d^2$. Choose two linear combinations of $g_1,g_2,g_3$
with no common factor, over a finite extension if necessary.
Such a choice exists: after the first combination is fixed, avoid
its finitely many irreducible factors when choosing the second.
B\'ezout's theorem bounds their common zeros, and hence those of
$g$, by $d^2$. There is therefore a short vertex with $g\ne0$.
At this vertex $A_1,B_1$ are independent linear forms, and the
remaining coordinate of $g$ is nonzero. Substitution in the
$3\times3$ determinant gives $J\ne0$.

This also covers $P=2$: the size bounds force $d=q/4$, and the
displayed expression gives $F(q/4)=3q/2>0$. The squarefree cubic
tangent cone again makes $A_1,B_1$ independent.
\end{proof}

The short-secant tangent cone also recovers the subfield assertion
directly for $P\ge3$: its $P+1$ rational directions are fixed by an
invertible $P$-semilinear projectivity\TermNote{exp:p1:semilinear}{a map induced by $v\mapsto A v^{[P]}$ with $A$ invertible}. After three of them are placed
at $0,1,\infty$, this map is $P$-Frobenius. It has
$p^{\gcd(e,h)}+1$ rational fixed points, so $e\mid h$ and the directions
form an $\F_P$-subline. This is the local argument behind the classical
subline theorem~\cite[Theorem~11.33]{p1:DSc}.

\begin{lemma}[The quadratic exponent bound]\label{p1:I:quadratic}
For a small blocking set whose line intersections are $1\pmod P$,
\begin{equation}\label{p1:I:quadratic-form}
 f_P(|B|)\ge0,\qquad
 f_P(x)=x^2-\bigl((P+1)(q+1)+1\bigr)x
                  +(P+1)(q^2+q+1).
\end{equation}
For $P\ge5$, smallness selects the interval below the smaller root.
\end{lemma}
\begin{proof}
Let $t_i$ count $i$-secants. The three incidence moments are
\[
 \sum t_i=q^2+q+1,\quad \sum it_i=|B|(q+1),\quad
 \sum i(i-1)t_i=|B|(|B|-1).
\]
Every nonzero $t_i$ has $i=1$ or $i\ge P+1$. Thus
$\sum(i-1)(i-P-1)t_i\ge0$, which expands to
\eqref{p1:I:quadratic-form}. At the smallness threshold,
\[
 4f_P\bigl(3(q+1)/2\bigr)
 =(7-2P)q^2+(4-8P)q+(1-2P)<0
 \quad(P\ge5).
\]
The upward-opening quadratic is negative between its roots, so the
smallness hypothesis excludes the interval above the larger root.
This is the classical moment bound underlying the exponent estimates
in~\cite{p1:SW,p1:DSc,p1:VDV}; its proof is included to specify exactly the
input needed at five.
\end{proof}

\begin{corollary}[A convenient coarse upper bound]\label{p1:I:coarse-upper}
For $P\ge7$ and $q\ge P^2$, a small blocking set with line
intersections $1\pmod P$ satisfies
\[
 |B|<q+q/P+q/P^2+3q/P^3.
\]
\end{corollary}
\begin{proof}
Put $x_0=q+q/P+q/P^2+3q/P^3$. Direct substitution gives
\begin{align*}
 P^6 f_P(x_0)
 &=q^2(-P^4+5P^3+7P^2+6P+9)\\
 &\quad-qP^3(2P^3+3P^2+5P+6)+P^6(P+1)<0.
\end{align*}
Indeed $-P^4+5P^3+7P^2+6P+9\le-7P^2+6P+9<0$ for
$P\ge7$. Since $q\ge P^2$, the last two terms together are at most
$-P^5(2P^3+2P^2+4P+6)<0$.
The smaller-root conclusion of Lemma~\ref{p1:I:quadratic} now gives
$|B|<x_0$. This recovers the familiar estimate
\cite[Lemma~1]{p1:VDV} from the moment calculation above.
\end{proof}

\section{Directional Hasse identities and Frobenius matrix formulas}\label{p1:secmap:4}\label{p1:I:integration}
We first record the only additional degree inequality needed in the
Hasse argument:
\begin{equation}\label{p1:I:flat-gap}
 \frac{q}{P}-(P-2)d+1>0.
\end{equation}
For $P\le4$ it follows from $d\le \frac{q}{2P}$. For $P\ge7$,
Corollary~\ref{p1:I:coarse-upper} gives
\[
 |B|<q+q/P+q/P^2+3q/P^3
\]
implies $d<\frac{q}{P^2}+\frac{q}{P^3}+\frac{3q}{P^4}$ and hence~\eqref{p1:I:flat-gap}.
For $P=5$, write $q=125S$ when $n\ge3$. Substituting
$|B|=q+5d+1$ in Lemma~\ref{p1:I:quadratic} gives
$d<27S/4-1/5<\frac{q}{3P}$. Indeed, for $d=5S+\rho$ the bounding quadratic is
\[
 \rho^2-(90S+1)\rho+150S^2-30S,
\]
whose value at $\rho=7S/4-1/5$ is
$-71S^2/16-289S/20+6/25<0$. Its smaller root bounds the
admissible interval. If $n=2$, the roots in Lemma~\ref{p1:I:quadratic} are
$P^2+P+1$ and $P^3+1$; together with Proposition~\ref{p1:I:ordinary-bound}
they give $d=1$, so the same degree inequality holds.

Define $\mathcal H_W^{[j]}g$ to be the coefficient of $t^j$ in
$g(V+tW)$; when the superscript is omitted, the order is one.
\begin{lemma}[Directional determinant identities]\label{p1:pp:higher-flatness}
For $2\le j\le P-2$ and independent direction variables $W$,
and also for $P=3$, $j=2$,
\begin{equation}\label{p1:pp:higher-flatness-equation}
 \det\bigl(g,\mathcal H_W^{[1]}g,\mathcal H_W^{[j]}g\bigr)=0.
\end{equation}
Here $\mathcal H_W^{[j]}$ is the Hasse directional derivative, defined by
$g(V+tW)=\sum_jt^j\mathcal H_W^{[j]}g(V)$.
\end{lemma}
\begin{proof}
For odd $p$ there is an order-two identity even when $P=3$.
The coefficients of $\delta$ have zero ordinary derivatives, so
\[
 \delta^2g=(\delta\beta+\beta^2)g+(\beta c+\delta c)V^{[q/P]}
          =2\mathcal H_{g^{[P]}}^{[2]}g.
\]
Consequently $\det(g,\mathcal H_Wg,\mathcal H_W^{[2]}g)$ vanishes on
$\operatorname{Span}(V,g^{[P]})$. Its $V$-degree $3d-3$ is below
$Pd+1$, so primitivity of $V\times g^{[P]}$ makes the polynomial
identically zero. This proves the extra case $P=3$, $j=2$.
For the remaining range we use the following uniform argument.

Freeze $g(V)^{[P]}$ when forming the substitution and put
$G_j=\mathcal H_{g(V)^{[P]}}^{[j]}g(V)$.  Every $G_j$ belongs to the normal
family $\mathcal M$: substituting $V+t g(V)^{[P]}$ in the normal identity for
$g$ preserves divisibility by $L_A(V)$, because
\[
 L_A(V+t g(V)^{[P]})
 =L_A(V)+t\bigl(A^{[1/P]t}g(V)\bigr)^P.
\]
Its degree is $d+j(Pd-1)$.  For $1\le j\le P-2$, this is less
than $(P-1)(\frac{q}{P}+d)+1$: at $j=P-2$ the difference is
\[
 (P-1)\bigl[\frac{q}{P}-(P-2)d+1\bigr]
 >0.
\]
The positivity $\frac{q}{P}-(P-2)d+1>0$ was checked above.
 Thus $G_j$ is in the polynomial span of
$g,F_0$ throughout this range.

Write $\mathcal A_j=\det(g,\mathcal H_W^{[1]}g,\mathcal H_W^{[j]}g)$ and induct on
$j$, starting with $\mathcal A_0=\mathcal A_1=0$.  Homogeneity gives
\[
 \mathcal H_{W+tV}^{[j]}g=
 \sum_{i=0}^j\binom{d-i}{j-i}t^{j-i}\mathcal H_W^{[i]}g.
\]
Once the lower $\mathcal A_i$ vanish, this implies
$\mathcal A_j(V,W+tV)=\mathcal A_j(V,W)$.  The basis
calculation gives $\mathcal A_j(V,g^{[P]})=0$.  Homogeneity in $W$
therefore makes $\mathcal A_j$ vanish on
$\operatorname{Span}_{\F_q(V)}(V,g^{[P]})$.  Its linear equation
$E\cdot W$ divides $\mathcal A_j$ over $\F_q(V)[W]$, and by
primitivity also over $\F_q[V,W]$.  But
\[
 \deg_V\mathcal A_j=3d-j-1<Pd+1=\deg E,
\]
so $\mathcal A_j=0$.
\end{proof}

In the next three lemmas, $\mathcal H_W^2$ denotes the composition
of the first directional derivative with itself. In odd characteristic
it equals $2\mathcal H_W^{[2]}$, not the second Hasse derivative itself.
The next three matrix recognition lemmas are stated over $\F_q$.
Their proofs use only that the coefficient field is perfect\TermNote{exp:p1:perfect}{every coefficient has a $p$th root; for $K=\F_q(x,y)$, the power subfield is $K^{p^j}=\{f^{p^j}:f\in K\}$}, so also
apply over any perfect field of the indicated characteristic.

\subsection{Recognizing the matrix in characteristic two}
\begin{lemma}[Rational matrix formula in characteristic two]\label{p1:I:binary-integration}
Let $\F_q$ have characteristic two, $K=\F_q(x,y)$, and $u\in K^3$.
If $\Delta=\det(u,u_x,u_y)\ne0$, there is
$C\in\operatorname{GL}_3(K^2)$ such that
\[
 \Delta u=C(x,y,1)^t,\qquad\det C=\Delta^4.
\]
No second-derivative hypothesis is needed.
\end{lemma}
\begin{proof}
Write $u=A+xB+yC_0+xyE$, with coefficient vectors in $(K^2)^3$.
Expansion gives
\[
 \Delta=d_0+xd_1+yd_2+xyd_3,
\]
where $d_0=\det(A,B,C_0)$, $d_1=\det(A,B,E)$,
$d_2=\det(A,E,C_0)$, and $d_3=\det(E,B,C_0)$.
The $xy$ coefficient of $\Delta u$ is
$d_0E+d_1C_0+d_2B+d_3A=0$, the four-vector alternating relation in
dimension three. The other three coefficients give the claimed matrix.
Taking determinants of the vector and its first derivatives gives
$\det C=\Delta^3\det(u,u_x,u_y)=\Delta^4\ne0$.
\end{proof}
\subsection{Recognizing the matrix in odd characteristic}
\begin{lemma}[Rational matrix formula]\label{p1:app:p-integration}
Let $\F_q$ have odd characteristic $p$ and let $u\in \F_q(x,y)^3$
satisfy
\[
 \Delta=\det(u,u_x,u_y)\ne0,\qquad
 \det(u,\mathcal H_Wu,\mathcal H_W^2u)=0
\]
for every formal direction $W$. Then $hu=\mathsf C(x,y,1)^t$, where
\[
 h\in \F_q(x,y)^*,\qquad \mathsf C\in\operatorname{GL}_3(\F_q(x^p,y^p)).
\]
For $p>3$ one can take $h=\Delta^\kappa$, where
$1\le\kappa<p$ and $3\kappa\equiv-1\pmod p$.
\end{lemma}
\begin{proof}
For $p=3$, see the next subsection. Suppose $p>3$.
Coefficient comparison in the basis $(u,u_x,u_y)$ gives
\[
 u_{xx}=a_0u+2\alpha u_x,\qquad
 u_{xy}=b_0u+\beta_0u_x+\alpha u_y,\qquad
 u_{yy}=c_0u+2\beta_0u_y.
\]
Mixed third derivatives give
\begin{align*}
 a_0&=\alpha_x-\alpha^2,&c_0&=(\beta_0)_y-\beta_0^2,\\
 2\alpha_y&=b_0+(\beta_0)_x+\alpha\beta_0,
 &2(\beta_0)_x&=b_0+\alpha_y+\alpha\beta_0.
\end{align*}
Subtraction gives $\alpha_y=(\beta_0)_x$ and
$b_0=\alpha_y-\alpha\beta_0$. Also
$\Delta_x=3\alpha\Delta$ and $\Delta_y=3\beta_0\Delta$.
Choose $\kappa$ as stated. Then $h=\Delta^\kappa$ satisfies
$h_x=-\alpha h$ and $h_y=-\beta_0h$, and substitution shows that
all second derivatives of $hu$ vanish.
Expansion over $\F_q(x^p,y^p)$ in the basis $x^iy^j$, $0\le i,j<p$,
leaves precisely $1,x,y$. The coefficient matrix is invertible
because the rescaled determinant is $h^3\Delta\ne0$.
\end{proof}

\subsection{The additional argument in characteristic three}\label{p1:app:three}

\begin{lemma}[Rational matrix formula in characteristic three]\label{p1:I:ternary-integration}
Let $\F_q$ have characteristic $3$, put $K=\F_q(x,y)$,
and let $u\in K^3$. Suppose that
\[
 \Delta=\det(u,u_x,u_y)\ne0,
 \qquad \det(u,\mathcal H_Wu,\mathcal H_W^2u)=0
 \quad\text{for every formal direction }W=(s,t).
\]
Then there are $H\in K^*$ and $\mathsf C\in\operatorname{GL}_3(\F_q(x^3,y^3))$
such that $Hu=\mathsf C(x,y,1)^t$.
\end{lemma}
\begin{proof}
The coefficient comparison in the second hypothesis gives
\[
 u_{xx}=a u+2\alpha u_x,\qquad
 u_{xy}=b u+\beta u_x+\alpha u_y,\qquad
 u_{yy}=c u+2\beta u_y.
\]
Here all coefficients belong to $K$, since $(u,u_x,u_y)$ is a basis.
Comparing $u_{xxy}$ with $u_{xyx}$ and $u_{xyy}$ with $u_{yyx}$ gives
\[
 a=\alpha_x-\alpha^2,\qquad c=\beta_y-\beta^2,
\]
\[
 2\alpha_y=b+\beta_x+\alpha\beta,\qquad
 2\beta_x=b+\alpha_y+\alpha\beta.
\]
The operators $\partial_x^3$ and $\partial_y^3$ vanish on $K$.
Differentiating the
first displayed equation and substituting the formula for $a$ gives
\[
 0=u_{xxx}=(\alpha_{xx}+\alpha^3)u.
\]
The characteristic-three operator identity
$(\partial_x+\alpha)^3=\partial_x^3+\alpha_{xx}+\alpha^3$
\TermNote{exp:p1:operator}{$\partial_x+\alpha$ acts by $f\mapsto f_x+\alpha f$; its cube means three successive applications} therefore gives
$(\partial_x+\alpha)^3=0$. The operator is linear on the
three-dimensional vector space $K$ over $K_x=\F_q(x^3,y)$, so it has a
nonzero kernel vector $H_0$. Replacing $u$ by $H_0u$ preserves both
hypotheses: the relevant determinants are multiplied by $H_0^3$.
The coefficient $\alpha$ changes to $\alpha+(H_0)_x/H_0=0$;
the compatibility equation then gives $a=0$.

Thus $u=A+xB$ for $A,B\in K_x^3$. These vectors are independent,
because otherwise $u$ and $u_x=B$ would be dependent. The mixed
equation, with $b=-\beta_x$, becomes
\[
 B_y=-\beta_xA+(\beta-x\beta_x)B.
\]
Solving for its two coefficients using a nonzero minor of $(A,B)$
shows that $\beta_x$ and $\beta-x\beta_x$ belong to $K_x$.
Consequently, in the basis $1,x,x^2$ of $K/K_x$,
\[
 \beta=b_0+b_1x,\qquad b_0,b_1\in K_x,
 \qquad B_y=-b_1A+b_0B.
\]
Now compare the $x^2$ coefficients in
\[
 A_{yy}+xB_{yy}
  =(\beta_y-\beta^2)(A+xB)+2\beta(A_y+xB_y).
\]
On the right the coefficient is
\[
 -b_1^2A+(b_{1,y}-2b_0b_1)B+2b_1B_y=b_{1,y}B.
\]
It follows that $b_{1,y}=0$. Hence
\[
 C:=\beta_x=b_1\in \F_q(x^3,y^3)=K^3.
\]

The analogous third-derivative computation gives
$\beta_{yy}+\beta^3=0$. Therefore $(\partial_y+\beta)^3=0$ on
$K/\F_q(x,y^3)$, and there is a nonzero $K_1\in K$ with
\[
 \beta=-\frac{(K_1)_y}{K_1}.
\]
We show $C=0$ by comparing denominators. Write
$K_1=c\prod f^{m_f}$, where the $f\in\F_q[x,y]$ are distinct
irreducible polynomials and the integers $m_f$ may be negative.
Then $\beta=-\sum m_f f_y/f$. Every factor of the reduced
denominator of $\beta$ occurs at most once, and every factor of
that of $C=\beta_x$ occurs at most twice. But $C\in K^3$, so all
these multiplicities are divisible by three. Hence $C$ is polynomial.
On the other hand, write $K_1=A_*/B_*$ over $\F_q(x)[y]$. In
$\beta=-(A_*)_y/A_*+(B_*)_y/B_*$, each numerator has smaller $y$-degree than
its denominator. This remains true after differentiation in $x$:
$(A/B)_x=(A_xB-AB_x)/B^2$ has numerator degree less than $2\deg_yB$
whenever $\deg_yA<\deg_yB$. Thus $C$ is both polynomial and a
rational function with numerator degree smaller than denominator
degree in $y$. This forces $C=0$.

We now have $\beta\in K_x$. On the degree-three extension
$K_x/\F_q(x^3,y^3)$ the operator $\partial_y+\beta$ again has cube zero.
Choose a nonzero $H_1\in K_x$ in its kernel and replace $u$ by $H_1u$.
This leaves $\alpha=0$ and makes $\beta=0$. The compatibility
equations give $a=b=c=0$, so all second derivatives of the rescaled
vector vanish. Expansion in the basis $x^iy^j$, $0\le i,j<3$, of
$K/K^3$ leaves precisely the terms $1,x,y$. Thus the rescaled vector
is $\mathsf C(x,y,1)^t$ for a matrix $\mathsf C$ over $K^3$.
Its determinant equals the nonzero rescaled $\Delta$, so $\mathsf C$ is
invertible. Taking $H=H_0H_1$ proves the assertion.
\end{proof}

\begin{lemma}[The full $P$-Frobenius matrix formula]\label{p1:pp:P-integration}
For $K=\F_q(x,y)$ and $v=(x,y,1)^t$, there are $h\in K^*$ and
$C\in\operatorname{GL}_3(K^P)$ such that
\[
 h\,g(x,y,1)=Cv.
\]
\end{lemma}
\begin{proof}
For $p=2$, use Lemma~\ref{p1:I:binary-integration}.
For odd $p$, use Lemma~\ref{p1:app:p-integration}, including its
characteristic-three case. If $P=3$, the separate order-two argument
above supplies its hypothesis. For $P>3$, use the order-two Hasse
identity and $\mathcal H_W^2=2\mathcal H_W^{[2]}$.
After multiplication by a nonzero rational function, therefore, $u=Cv$ over $K^p$.

All the higher determinant identities survive a multiplication by a nonzero rational function.  Indeed, the Hasse product rule gives
\[
 \det(hu,\mathcal H_W(hu),\mathcal H_W^{[j]}(hu))
 =h^2\sum_{i=0}^j \mathcal H_W^{[i]}h\,
       \det(u,\mathcal H_Wu,\mathcal H_W^{[j-i]}u),
\]
whose terms vanish when the identities through order $j$ hold.

Suppose that $u=Cv$ with $C\in\operatorname{GL}_3(K^r)$, where
$r=p^a<P$.  Put $\xi=x^r$, $\eta=y^r$, and use independent
direction variables $a,b$ with $w=(a,b,0)^t$.  Since $r>1$,
\[
 \mathcal H_wu=Cw,\qquad \mathcal H_w^{[r]}u=(a^rC_\xi+b^rC_\eta)v.
\]
Set $A=C^{-1}C_\xi$ and $B=C^{-1}C_\eta$.
The determinant identity of order $r$ becomes
\[
 \det\bigl(v,w,(a^rA+b^rB)v\bigr)=0.
\]
The monomials $a^{r+1},a^rb,ab^r,b^{r+1}$ are distinct.
Their coefficients show
\[
 \det(v,e_1,Av)=\det(v,e_2,Av)=0,
 \qquad
 \det(v,e_1,Bv)=\det(v,e_2,Bv)=0.
\]
Thus $Av$ and $Bv$ are multiples of $v$.

If a matrix $D\in M_3(K^r)$ satisfies $Dv\in Kv$, then $D$ is
scalar.  To see this, eliminate the eigenvalue using the third
coordinate.  For $r\ge3$ the remaining equations involve
the independent monomials $1,x,y,x^2,xy,y^2$ over $K^r$ and force
the assertion. For $r=2$, write the two equations explicitly:
\begin{align*}
0&=(d_{11}-d_{33})x+d_{12}y+d_{13}-d_{31}x^2-d_{32}xy,\\
0&=d_{21}x+(d_{22}-d_{33})y+d_{23}-d_{31}xy-d_{32}y^2.
\end{align*}
The $xy$ coefficients first give $d_{32}=d_{31}=0$; independence
of $1,x,y$ then kills the other off-diagonal entries and equates
the diagonal entries. Thus the scalar-matrix assertion also holds
at the first binary lifting step.
It follows that $C_\xi=\alpha C$ and $C_\eta=\beta C$.

Choose a nonzero entry $c$ of $C$.  Every entry of $C/c$ has zero
$\xi$- and $\eta$-derivatives, so lies in
$\F_q(\xi^p,\eta^p)=K^{pr}$.  Replacing $u$ by $u/c$ gives the same
representation with a matrix in $\operatorname{GL}_3(K^{pr})$.
Iterate for $r=p,p^2,\ldots,P/p$.  When $e=1$ no lifting step is needed. Otherwise every required order is at most
$P/p\le P-2$, so Lemma~\ref{p1:pp:higher-flatness} supplies all of
them.  The last matrix lies in $\operatorname{GL}_3(K^P)$.
\end{proof}

\subsection{Homogenization and reciprocity}
\begin{proposition}[Full $P$-level identities]\label{p1:pp:identities}
There are a squarefree homogeneous polynomial $H$, a primitive
nonsingular homogeneous polynomial matrix $M$, and nonzero
homogeneous polynomials $a_0,D$ such that
\begin{align}
 Hg&=M^{[P]}V,& Mg&=Ha_0V^{[q/P^2]},&
 \det M&=H^2a_0D,\label{p1:pp:uniform}\\
 Dg&=NV^{[q/P^2]},&N&=\operatorname{adj}(M)/H,&
 c&=H^{P-1}a_0^P.\label{p1:pp:reciprocal}
\end{align}
Writing $m=\deg M$, one has
\begin{align}
 \deg H&=Pm-d+1,&\deg a_0&=2d-(P-1)m-\frac{q}{P^2}-1,\nonumber\\
 \deg D&=\frac{q}{P^2}-(P-2)m-1,&\deg N&=d-(P-2)m-1.\label{p1:pp:degrees}
\end{align}
Moreover
\begin{equation}\label{p1:pp:Phi}
 J=H^{2P-3}(a_0D)^P,\quad
 \Phi=J/c=H^{P-2}D^P,\quad\deg\Phi=\frac{q}{P}-(P-2)d-2.
\end{equation}
\end{proposition}
\begin{proof}
Take entrywise $P$th roots in Lemma~\ref{p1:pp:P-integration} and
clear all denominators by a common $P$th power.  At least one
affine coordinate of $g$ has degree exactly $d$, since otherwise
$Z$ would divide all three homogeneous coordinates.  If the
polynomial root matrix has entries of degree at most $m$, the
polynomial multiplier has degree at most $Pm+1-d$.
Homogenizing to these degrees gives $Hg=M^{[P]}V$.

Make $M$ primitive: a common factor $f$ of all its entries has
$f^P\mid H$ by primitivity of $g$, so divide $M$ by $f$ and $H$
by $f^P$, and repeat.  The resulting $H$ is squarefree.  Indeed,
if $f^2\mid H$, every partial derivative of every $Hg_i$ is
divisible by $f$, whereas
\[
 \partial_j(Hg_i)=M_{ij}^P.
\]
This would make $f$ a common factor of $M$.

The coordinates of $g$ are independent over $\F_q(V)^p$: a relation
over that field can be differentiated to give a row annihilating
$g,g_X,g_Y$, contrary to Lemma~\ref{p1:I:separable}.  Apply $\delta$ to the
first identity and use~\eqref{p1:pp:delta}:
\[
 (Mg)^{[P]}=(\delta H+H\beta)g+HcV^{[q/P]}.
\]
If $\delta H+H\beta$ were nonzero, the three coordinates of $g$
would lie in the span of two elements over $\F_q(V)^p$, since both
Frobenius vectors have entries in that field.  Therefore
$\delta H=-H\beta$ and $Hc$ is a polynomial $P$th power.
Squarefreeness of $H$ gives $Hc=(Ha_0)^P$ and hence
$Mg=Ha_0V^{[q/P^2]}$.

Differentiating $Hg=M^{[P]}V$ gives\TermNote{exp:p1:jacobian}{$\operatorname{Jac}(g)=(\partial g_i/\partial V_j)_{i,j}$ and $\nabla H=(H_X,H_Y,H_Z)^t$}
\[
 M^{[P]}=H\operatorname{Jac}(g)+g(\nabla H)^t.
\]
\label{exp:p1:H-proof}Let $f$ be an irreducible factor of $H$. Modulo $f$, the displayed
matrix is the product of one column and one row, so every $2\times2$
minor of $M^{[P]}$ vanishes. Such a minor is the $P$th power of the
corresponding minor of $M$; thus $f$ divides every cofactor of $M$.
Since $H$ is squarefree, the product of its distinct irreducible
factors divides each cofactor. Therefore $H\mid\adj M$ entrywise,
and $N=\adj M/H$ is a polynomial matrix of degree $2m-\deg H$.
For completeness, fix an
irreducible factor $f$ of $H$ and choose an entry of $M$ not divisible
by $f$, which exists by primitivity. Use it to eliminate its column
from the other two rows. All entries in those rows are divisible by
$f$, since the rank modulo $f$ is at most one. The denominators
introduced are coprime to $f$, so $f^2\mid\det M$.
Since $H$ is squarefree, $H^2\mid\det M$.  Write
$\operatorname{adj}M=HN$ and $\det M=H^2K$.  Multiplying the
second identity by $N$ gives
$Kg=a_0NV^{[q/P^2]}$.  Primitivity of $g$ implies $a_0\mid K$.
Putting $K=a_0D$ proves~\eqref{p1:pp:uniform}--\eqref{p1:pp:reciprocal}.

Taking the determinant with the two first partial derivatives in
the first identity gives $J=(\det M)^P/H^3$.
This proves~\eqref{p1:pp:Phi}; the degree formulas~\eqref{p1:pp:degrees}
follow by comparing degrees in the displayed identities.
\end{proof}

\begin{proof}[Completion of the proof of Theorem~\ref{p1:I:uniform-main}]
The exponent discussion in Section~\ref{p1:I:normalization},
Lemmas~\ref{p1:pp:normalization}, \ref{p1:pp:module} and \ref{p1:I:separable},
and Proposition~\ref{p1:pp:identities} give all the assertions.
\end{proof}

For $P\ge3$, in particular,
\begin{equation}\label{p1:pp:exceptional-degree}
 \deg(Ha_0D)=d-1-(P-3)\deg M\le d-1.
\end{equation}
The coordinate rules in Section~\ref{p1:I:normalization} apply to these identities.

\section{Ordinary components, short secants, and the weight gap}\label{p1:secmap:5}\label{p1:I:local}

We use the established identities
\[
 R=\det(V,V^{[q]},g^{[P]}),\qquad
 Hg=M^{[P]}V,\qquad Mg=Ha_0V^{[q/P^2]},
\]
\[
 \det M=H^2a_0D=H^2K,\qquad
 J=H^{2P-3}K^P,\qquad \deg(HK)=d-1-(P-3)m,
\]
where $P\ge3$ and $m=\deg M$, $H$ is squarefree, $g$ and
$E=V\times g^{[P]}$ are primitive, and $J\ne0$.
For $a\in B$, the normal polynomial $a^{[1/P]t}g$ is divisible
by $L_a=a^tV$. The component is called ordinary if this
divisibility has order exactly one, and special otherwise.
A vertex of weight $r$ has $Pr+1$ incident components.

At a rational vertex normalized to $(0,0,1)$, the transverse
coordinates $g_1,g_2$ have order $r$, with nonzero coprime
homogeneous initial forms $A_r,B_r$ of degree $r$. Their tangent
cone is a nonzero scalar multiple of
\[
 yA_r^P-xB_r^P.
\]
In particular, $A_1,B_1$ are independent linear forms.
This is the local description~\eqref{p1:pp:cone}.

\begin{lemma}[Ordinary pencils]\label{p1:ppg:ordinary}
On an ordinary component $L_a=Z$, write $g_3=Zh$ and
$h_0=h|_{Z=0}$. At every weight-$r$ vertex on this component,
$\operatorname{ord}h_0\ge r-1$. Its number $k_a$ of vertices satisfies
\[
 \frac{q}{P}+1\le k_a\le Pd+1<q.
\]
Consequently it contains at least $\frac{q}{P}-d+2>0$ short vertices.
\end{lemma}
\begin{proof}
The local order and the two pencil bounds are
Proposition~\ref{p1:I:ordinary-bound}. If $s_a$ vertices are short,
the pencil identity gives $\frac{q}{P}+d\ge2k_a-s_a$; hence
$s_a\ge \frac{q}{P}-d+2$.
\end{proof}

\begin{lemma}[Local leading block]\label{p1:ppg:leading}
Suppose a vertex has transverse initial forms $A_r,B_r$ whose
ratio is not a $P$th power. Then $\operatorname{ord}g_3<r$.
Writing $\nu=\operatorname{ord}H$, one has
$\nu+r\equiv1\pmod P$. There is a nonsingular $2\times2$ matrix $C_0$ of forms of degree
$j=(\nu+r-1)/P$ such that
\[
 H_\nu(A_r,B_r)^t=C_0^{[P]}(x,y)^t,
\]
where $H_\nu$ is the initial form of $H$.
\end{lemma}
\begin{proof}
In $Hg=M^{[P]}(x,y,1)^t$, the possible total-degree residues
are zero and one modulo $P$. If $\nu+r$ had residue zero,
$H_\nu A_r,H_\nu B_r$ would both be $P$th powers, contrary to
the hypothesis. Thus $\nu+r=Pj+1$. The first two columns in
the top two rows of $M$ have order at least $j$ and the third
has order at least $j+1$. Their degree-$j$ top-left block is
$C_0$, with the displayed identity. A singular $C_0$ would again
make $A_r/B_r$ a $P$th power, so $C_0$ is nonsingular.

Suppose now that $g_3$ had order at least $r$. The same order
bounds then hold in every row of $M$. The top pair of $Mg$
has a nonzero term of degree $j+r$, namely
$C_0(A_r,B_r)^t$. Every coordinate of $Mg$ has order at least
$j+r$, so its third coordinate in $Mg=Ha_0(x^{q/P^2},y^{q/P^2},1)^t$
forces $\operatorname{ord}(Ha_0)\ge j+r$. Its top pair would
then have order at least $j+r+\frac{q}{P^2}$, a contradiction.
\end{proof}

\begin{lemma}[Short-vertex regularity]\label{p1:ppg:short}
At every short vertex $g\ne0$ and $J\ne0$, and all incident lines are
ordinary. The polynomials
$H,a_0,D,\det M$ are nonzero at every short vertex.
The $P+1$ points of $B$ on a short secant form an $\F_P$-subline.
\end{lemma}
\begin{proof}
For $r=1$, the independent forms $A_1,B_1$ have ratio not a
$P$th power. Lemma~\ref{p1:ppg:leading} gives
$\operatorname{ord}g_3=0$. Their independent first derivatives
then give $J\ne0$. The identity $J=H^{2P-3}(a_0D)^P$
proves all the asserted nonvanishing statements.
For any incident component, its nonzero normal applied to the
invertible transverse linear map $(A_1,B_1)$ is a nonzero
linear form. Thus its normal divisibility has order exactly one.

Let $A\in\operatorname{GL}_2(\F_q)$ represent $(A_1,B_1)$.
The tangent directions are exactly the fixed points of the
projective semilinear map $[v]\mapsto[A^{[P]}v^{[P]}]$.
They are $P+1$ distinct $\F_q$-rational points. Conjugate three of
them to $0,1,\infty$ by a projectivity over $\F_q$. A
$P$-semilinear projectivity fixing these three points is
$P$-Frobenius itself, whose rational fixed points are
$\PG(1,P)$, since $e\mid h$ and $\F_P\subseteq\F_q$.
Dualizing tangent directions to their component
normals preserves the subline property.
\end{proof}

\begin{lemma}[Special components]\label{p1:ppg:special}
The special components are precisely the arrangement-line
factors of $HK$. Their number $\sigma$ satisfies
\[
 \sigma\le d-1-(P-3)m\le d-1,
 \qquad |B|-\sigma\ge q+(P-1)d+2>q+1.
\]
\end{lemma}
\begin{proof}
On a special component $Z=0$, $Z^2\mid g_3$. The third row
of $Hg=M^{[P]}V$, restricted to $Z=0$, first gives
$Z\mid M_{31},M_{32}$, because the $X$- and $Y$-residue
classes modulo $P$ are disjoint. Comparing its coefficient
of $Z$ then gives $Z\mid M_{33}$. Hence $Z\mid\det M$
and therefore $Z\mid HK$. Conversely, every ordinary
component contains a short vertex by Lemma~\ref{p1:ppg:ordinary},
where $HK$ is nonzero by Lemma~\ref{p1:ppg:short}. Thus an
ordinary component cannot divide $HK$. The degree and
cardinality bounds follow.
\end{proof}

\begin{lemma}[The full $P$-weight gap]\label{p1:ppg:gap}
No vertex has weight $r$ with $2\le r\le P-1$. Every ordinary
component contains at least
\[
 m_0(d)=\frac{q}{P}+1-\left\lfloor\frac{d-1}{P-1}\right\rfloor
\]
short vertices.
\end{lemma}
\begin{proof}
Suppose $2\le r<P$. Coprimality and degree show that
$A_r/B_r$ is not a $P$th power. Lemma~\ref{p1:ppg:leading}
gives $\nu+r\equiv1\pmod P$ and
$s_0:=\operatorname{ord}g_3<r$. The allowed total-degree
residues of $Hg_3$ are zero and one, so
$\nu+s_0\equiv0$ or $1\pmod P$. Since $0<r-s_0<P$,
one has $s_0=r-1$, and $\nu\ge P+1-r$.

There is a characteristic-sensitive point here: a rational
function of degree below $P$ need not have nonzero ordinary
derivative. The leading-block identity supplies exactly what
is needed. Write
\[
 C_0=\begin{pmatrix}c_{11}&c_{12}\\c_{21}&c_{22}\end{pmatrix}.
\]
Differentiating its $P$th-power entries gives
\[
 A_r(B_r)_x-B_r(A_r)_x
 =-\frac{y(\det C_0)^P}{H_\nu^2}\ne0.
\]
Thus the relevant ratio is not even a $p$th power.
If $C_{r-1}$ is the nonzero initial form of $g_3$, Euler's
identity applied to the homogeneous initial forms yields
\[
 yJ_{\rm initial}
 =-C_{r-1}\bigl(A_r(B_r)_x-B_r(A_r)_x\bigr)\ne0.
\]
Consequently $\operatorname{ord}J=3r-3$. But
\[
 \operatorname{ord}J\ge(2P-3)\nu
 \ge(2P-3)(P+1-r)>3r-3,
\]
because the difference in the last comparison is
$P(2P-1-2r)>0$. This contradiction proves the gap.

If an ordinary component has $k_a$ vertices and $s_a$ short
ones, the gap gives $\frac{q}{P}+d\ge s_a+P(k_a-s_a)$.
Using $k_a\ge \frac{q}{P}+1$ gives the stated floor formula.
\end{proof}

\begin{lemma}[Matrix normal restrictions]\label{p1:ppg:normals}
For every ordinary component there are nonzero homogeneous
normal forms $h_a,\mu_a$ with
\[
 a^{[1/P]t}g=L_a h_a,\qquad
 a^{[1/P^2]t}M=\mu_a a^{[1/P]t}\pmod{L_a},
 \qquad (H|_{L_a})(h_a|_{L_a})=\mu_a^P.
\]
Here $\mu_a$ has degree $m$ on $L_a$, and both restricted
normal forms are nonzero at every short vertex on that component.
\end{lemma}
\begin{proof}
Normalize $a$ to $(0,0,1)$. Restriction of the third row of
$Hg=M^{[P]}V$ gives $Z\mid M_{31},M_{32}$; set
$\mu_a=M_{33}|_{Z=0}$. Its coefficient of $Z$ is precisely
$(H|_{Z=0})(h_a|_{Z=0})=\mu_a^P$.
The short-vertex first-derivative calculation in
Lemma~\ref{p1:ppg:short} proves nonvanishing there. Returning
to general rational coordinates gives the stated row identity.
This is the coordinate rule recorded in Section~\ref{p1:I:normalization}.
\end{proof}

\section{Common zeros, tangent lines, and factors of the root vector}\label{p1:secmap:6}\label{p1:I:ggeometry}

We distinguish $g$ from $G=g^{[P]}$, the vector that describes the
tangent lines in the polynomial construction. Earlier work sometimes
uses $g$ for this latter vector. The first derivatives of $G$ vanish;
the nonzero determinant $J$ belongs to the root vector $g$.

\begin{proposition}[Common zeros and tangent lines]\label{p1:I:base-locus}
The common projective zero set of $g_1,g_2,g_3$ is finite, consists of rational dual
vertices, and contains at most $d^2$ points; at most $d$ lie on any
arrangement component. For a rational dual point $v$, the following hold.
If $v$ represents a tangent at $a$, then $G(v)$ is defined, differs from
$v$, and represents another line through $a$. If $v$ represents a secant,
then either $g(v)=0$ or $[G(v)]=[v]$.
Every component $L_a=0$ is invariant under $G$ wherever the map is defined.
\end{proposition}
\begin{proof}
Primitivity makes the common zero set finite. At a common zero of $g$ both $R$
and its gradient vanish, since
\[
 \nabla R=V^{[q]}\times g^{[P]}.
\]
The singular points\TermNote{exp:p1:smooth}{a point is singular when all first partial derivatives of a defining equation vanish there; otherwise it is smooth} of a squarefree arrangement of rational lines are
their intersections, hence are rational. Choose two coprime linear combinations of the coordinates of $g$;
B\'ezout gives the bound $d^2$.
On an arrangement component not all coordinates vanish identically;
restricting one nonzero coordinate gives the bound $d$.
At a rational point, $V^{[q]}=V$. A tangent is a smooth point of $R$,
so $V\times G\ne0$, and its normal is $a$. At a secant the gradient
vanishes. The invariant-line assertion is exactly
$L_a\mid a^{[1/P]t}g$ after taking $P$th powers.
\end{proof}

\begin{proposition}[A pencil test for a common zero]\label{p1:I:pencil-test}
Let $\ell$ be a secant, let $a\in B\cap\ell$, and give the other lines
of the pencil through $a$ an affine coordinate $t\in \F_q$, with $\ell$
at infinity. Put
\[
 k_t=\frac{|B\cap m_t|-1}{P},\qquad
 F_{a,\ell}(T_0)=\prod_{t\in \F_q}(T_0-t)^{k_t}.
\]
Then $\ell^*$ is a common zero of $g$ if and only if
$[T_0^d]F_{a,\ell}(T_0)=0$. One point $a$ suffices for the test.
Also $[T_0^j]F_{a,\ell}=0$ for $d<j<\frac{q}{P}$.
\end{proposition}
\begin{proof}
Choose $\ell^*=(0,0,1)$ and $a=(0,1,0)$. The restriction of the
arrangement identity to the pencil, after removal of its normal factor
and taking the unique $P$th root, gives, up to a nonzero scalar,
\[
 F_{a,\ell}(T_0)=g_3(1,0,T_0)-T_0^{q/P} g_1(1,0,T_0).
\]
Since $d<\frac{q}{P}$, the degree-$d$ coefficient is $g_3(0,0,1)$ and the stated
intermediate coefficients vanish. At a secant the two transverse
coordinates $g_1,g_2$ already vanish. This proves the criterion and its
independence of the chosen permissible pencil coordinate.
\end{proof}

\begin{proposition}[Uniqueness and recognition of the multiplier]\label{p1:I:canonical-H}
The primitive first identity $Hg=M^{[P]}V$ determines $H$ up to a
nonzero scalar. Every other polynomial multiplier admitting such an
identity is a scalar times $H$ times a polynomial $P$th power. In
particular,
\[
 H\doteq\prod_{\operatorname{ord}_f(c)\equiv-1\pmod P}f,
\]
where the product is over irreducible polynomial factors and $\doteq$
denotes equality up to scalar. If $P\ne3$, it also equals
\[
 \prod_{f\mid J,\ \operatorname{ord}_f(J)\equiv-3\pmod P}f.
\]
The latter formula must not be used at $P=3$.
\end{proposition}
\begin{proof}
On $Z=1$, compare two identities and divide their invertible coefficient
matrices. A matrix over $\F_q(x,y)^P$ having $(x,y,1)^t$ as an eigenvector
is scalar, by the scalar-matrix lemma in the integration proof. The
ratio of the multipliers is therefore a $P$th power in the rational
function field. For each prime factor, polynomiality and the fact that
$H$ has multiplicity zero or one force that ratio to have nonnegative
multiplicity in its factorization. It is a polynomial $P$th power up to a constant. This also
proves uniqueness of the squarefree representative. The formulas follow
from $c=H^{P-1}a_0^P$ and $J=H^{2P-3}(a_0D)^P$.
For $P=3$ the second exponent has residue zero, so the determinant $J$ no
longer distinguishes factors of $H$ from the other factors.
\end{proof}

\begin{proposition}[Additional reciprocal identities]\label{p1:I:extra-identities}
In the notation of Theorem~\ref{p1:I:uniform-main},
\[
 MN=Ha_0D I,\quad \adj N=a_0DM,\quad \det N=Ha_0^2D^2,
 \quad N^{[P]}g=H^{P-1}a_0^PD^P V.
\]
Furthermore $\delta H=-\beta H$, and
$\delta(c/H^{P-1})=0$. The homogeneous derivation $\delta$\TermNote{exp:p1:derivation}{a linear operator with $\delta(fh)=f\delta h+h\delta f$; $\delta^p$ means $p$ successive applications} satisfies
$\delta^p=0$.
\end{proposition}
\begin{proof}
The matrix identities follow from the adjugate formulas in dimension
three, and the last vector identity follows by applying $N^{[P]}$ to
$Hg=M^{[P]}V$. The differential identity for $H$ was obtained in the
proof of the reciprocal representation. The quotient is $a_0^P$,
so its derivative is zero. Finally the coefficients $g_i^P$ of
$\delta$ are constants for all ordinary partial derivatives. The
commuting-operator binomial formula in characteristic $p$ gives
$\delta^p=\sum_i g_i^{Pp}\partial_i^p=0$.
\end{proof}

\paragraph{Limits of geometric shortcuts.}
The tangent-map assertion is deliberately weaker than preservation of
tangents. For $q=p^3$, the linear R\'edei set associated with $x\mapsto
x^p$ has, up to scalar,
\[
 g=(-Y^{p+1},X^{p+1},X^pZ-YZ^p).
\]
In odd characteristic the tangent $[1:1:0]$ is sent to the secant
$[1:-1:0]$, and $[0:0:1]$, the dual point of the R\'edei line,
is a common zero of $g$. For $q=27$ this set has $40$ points and
line spectrum $1^{639},4^{117},13^1$. Thus neither tangent
preservation nor absence of common zeros is a permissible general input.
Likewise, multiplying $g$ by $H$ creates a polynomial representative
with extra common zeros; these are not additional common zeros of the
primitive rational map. The first and second matrix identities do not
by themselves construct a constant finite-field subspace. The descent
and compatibility arguments in the companion papers supply that extra
step where it is currently available.

\section{General reconstruction tools}\label{p1:secmap:7}\label{p1:I:reconstruction}

For $P=p$, signed indicators with zero line sums are codewords in
the dual of the point--line incidence code\TermNote{exp:p1:dual-code}{the vector $1_A-1_C$; the dual code consists of vectors whose coordinate sum on each line is zero}; compare~\cite{p1:LSVcode}.
The elementary argument below also works modulo the prime power $P$.
\begin{lemma}[A signed-pencil bound]\label{p1:I:signed}
Let $A,C$ be disjoint point sets in a projective plane of order $q$,
and suppose $|A\cap\ell|\equiv|C\cap\ell|\pmod P$ on every line.
If $A\cup C\ne\varnothing$, then
\[
 |A|+|C|\ge\frac{2(P-1)q}{P}+2.
\]
If only one support is nonempty, its size is at least $(P-1)q+P$.
\end{lemma}
\begin{proof}
For a point of $A$, each line through it has either another $P-1$
points of $A$, or enough $C$-points to make the congruence hold.
More precisely $(|A\cap\ell|-1)+(P-1)|C\cap\ell|\ge P-1$.
Summing gives $|A|+(P-1)|C|\ge(P-1)q+P$. If $C$ is nonempty,
interchange $A,C$ and add. The one-support assertion follows from
the same inequality.
\end{proof}

\begin{lemma}[Recognition by a constant space]\label{p1:I:cone}
The vector cone of $B$ means $\{0\}\cup\{v\ne0:[v]\in B\}$.
For $q=P^n$ and a minimal blocking set $B\subset\PG(2,q)$, linearity over
$\F_P$ is equivalent to the existence of an $(n+1)$-dimensional
$\F_P$-subspace of its vector cone.
\end{lemma}
\begin{proof}
An $(n+1)$-space meets each $2n$-space representing a line in the
ambient $3n$-space. Its projectivization therefore blocks, and minimality
forces equality with $B$. Conversely, a representing space has dimension
at least $n+1$, since a space of dimension at most $n$ has at most
$(P^n-1)/(p-1)<q+1$ projective directions. Take an $(n+1)$-subspace.
\end{proof}

\begin{proposition}[Moore certificate]\label{p1:I:Moore}
For $q=P^n$, a minimal blocking set $B$ has a rank-$(n+1)$ linear
representation over $\F_P$ if and only if there are $\F_P$-independent
linear forms $\ell_0,\ldots,\ell_n$ and a polynomial $W_0$, supported
on the arrangement components, such that
\[
 \det(\ell_j^{P^i})_{0\le i,j\le n}=c_0 R W_0^P,
 \qquad c_0\ne0.
\]
If a point has vector-space weight $w$\TermNote{exp:p1:fiber}{$w=\dim_{\F_P}(U\cap\F_qv)$ for the point $[v]$; its fiber consists of the nonzero vectors in this intersection}, its multiplicity in the Moore
product is $(P^w-1)/(P-1)$, and
\[
 \deg W_0=\frac{(P^{n+1}-1)/(P-1)-|B|}{P}.
\]
\end{proposition}
\begin{proof}
The Moore determinant (see, for example, \cite{p1:LidlNiederreiter})
factors, up to scalar, as the product of one
representative of each $\F_P$-projective linear combination of the
$\ell_j$. Indeed each such linear form divides the determinant, and
the product has the same degree; independence makes the determinant
nonzero. Collect equal $\F_q$-projective factors. A weight-$w$ fiber has
$(P^w-1)/(P-1)$ such factors, a number congruent to one modulo $P$.
This proves the forward assertion and the degree formula. Conversely,
the determinant factorization says that the projective support of the
coefficient space of the forms is exactly $B$.
\end{proof}

\begin{lemma}[Compatible representatives over $\F_q$]\label{p1:I:normalizer}
Assume the uniform Frobenius representation, and let $O$ be the ordinary
points. Suppose a rational homogeneous function $W$ of degree $d-1$ has a representation by a quotient whose numerator and denominator
are nonzero at every short vertex, and for each $a\in O$,
\[
 (a^{[1/P]t}g/L_a)|_{L_a}=\beta_aW|_{L_a},\qquad\beta_a\ne0.
\]
Then there are representatives $u_a$ over a finite extension such that
on every short secant the sets $\F_P^*u_a$ are precisely the nonzero
vectors of a two-dimensional $\F_P$-space. If the graph joining two points of $O$ when they lie on a short secant is connected, one common rescaling puts all $u_a$ in $\F_q^3$.
\end{lemma}
\begin{proof}
Choose $\lambda_a^{P-1}=\beta_a^P$ and put $u_a=\lambda_a a$.
The normal divisibility to order two is
$u_a^{[1/P]t}g-(u_a^tV)W\in(L_a^2)$ locally at a short vertex.
The invertible matrix of first derivatives in the two transverse coordinates gives a common invertible
$P$-semilinear fixed-vector equation on its two-dimensional normal
plane. An $\F_P$-independent family of fixed vectors is independent
over the extension field: normalize a shortest dependence, apply
Frobenius, and subtract. Its coefficients must lie in $\F_P$.
Thus the fixed space has dimension at most two; the $P+1$ distinct
directions force dimension two and give exactly its projective points.
For adjacent $a,b$, the vector $u_a+u_b$ has a rational direction.
Since $a,b$ are independent over $\F_q$, this implies
$\lambda_a/\lambda_b\in \F_q^*$. Propagate along the connected graph.
\end{proof}

\begin{proposition}[A dimension bound from pairs and triples]\label{p1:I:code-bound}
Suppose $N_0$ distinct points of $\PG(r-1,P)$ span that space.
Through each point suppose there are at least $m_0$ specified lines,
each containing all its $P+1$ points in the set. If $2+2Pm_0-N_0>0$, then
\[
 P^r\le\frac{PN_0((P-1)m_0+1)}{2+2Pm_0-N_0}.
\]
\end{proposition}
\begin{proof}
Choose nonzero representatives $v_i\in\F_P^r$. For a linear functional
$a$, let $w(a)$ count the indices with $a(v_i)\ne0$, and put
$y(a)=(P-1)N_0-Pw(a)$. If $a\ne0$, some $a(v_i)\ne0$.
On each specified line through that point, at least $P$ points have
nonzero value. The lines share only that point, so
$w(a)\ge\delta:=(P-1)m_0+1$.

Let $A_3$ count unordered collinear triples of the $N_0$ points.
Counting at each point the pairs on its specified lines gives
$A_3\ge N_0P(P-1)m_0/6$. Average over all $P^r$ functionals, including
zero. An individual summand in $y$ is $P-1$ when $a(v_i)=0$ and
$-1$ otherwise. Counting values of $a$ on one, two, or three independent
vectors gives
\[
 \mathbb Ey^2=(P-1)N_0,\qquad
 \mathbb Ey^3=(P-1)(P-2)N_0+6(P-1)A_3.
\]
Here a single index contributes $(P-1)(P-2)$ to the third moment;
a distinct collinear triple contributes $P-1$ in each of its six
orders. All other mixed terms have average zero.
Put $u=(P-1)N_0-P\delta$. For $a\ne0$, $y(a)^2(y(a)-u)\le0$.
Only $a=0$ can contribute positively, and consequently
\[
 (P-1)^2N_0(2+2Pm_0-N_0)
 \le\mathbb E\bigl(y^2(y-u)\bigr)
 \le\frac{((P-1)N_0)^2P\delta}{P^r}.
\]
Cancellation gives the assertion. This is an integer count and also
applies when the characteristic is two.
\end{proof}

\section{From a matrix formula to compatible point representatives}\label{p1:secmap:8}\label{p1:I:endpoint}
We prove an implication from an explicit matrix formula to compatible
point representatives. It is a general reconstruction theorem; its
matrix hypothesis is supplied by Part~2 for $P\ge5$, and already
follows here when $q=P^3$, $P\ge4$.
Assume $P\ge3$, $q=P^n$, $n\ge3$, put $S=P^{n-3}$, and suppose
a nonzero polynomial $\Theta$ and a $3\times3$ matrix polynomial
$\mathsf C(X_0,\ldots,X_{k_{\rm slot}-1})$ with constant coefficients satisfy
\begin{equation}\label{p1:I:endpoint-hyp}
 \Theta M=\mathsf C(V,V^{[P]},\ldots,V^{[S]}),\qquad
 d_{\rm nat}:=\deg(\Theta M)=\frac{PS-1}{P-1}.
\end{equation}
Here $k_{\rm slot}=n-2$ and each $X_i$ is an independent column of
three variables. Each entry of $\mathsf C$ is homogeneous of degree
one in each $X_i$ separately. This is the meaning of \emph{multilinear}
throughout this section. For example, with two arguments it has the
form $\sum_{a,b}(X_0)_a(X_1)_b C_{ab}$, with constant matrices $C_{ab}$.
\label{exp:p1:frobenius-endpoint}We call this constant-coefficient, multilinear representation
\eqref{p1:I:endpoint-hyp} the \emph{Frobenius endpoint}. It is the final
form sought by the successive extraction of coordinate powers;
compatible representatives and their span still have to be constructed.
The subscript $\Delta$ denotes the explicit substitution
$X_i=V^{[P^i]}$, called the Frobenius diagonal.
The letter $m$ still denotes $\deg M$.

\label{exp:p1:pencil}A matrix pencil is a linear combination
$C(Y)=Y_1C_1+Y_2C_2+Y_3C_3$ of fixed matrices. Its
\emph{coefficient dimension} is the dimension of the span of the
three matrices $C_i$. Its evaluation at a column $W$ is the linear
map $Y\mapsto C(Y)W$, with matrix $[C_1W\ C_2W\ C_3W]$.
When other arguments are present, they are treated as independent
variables in the coefficient field.

\ReadingNote{Independent vector variables}{The three coordinates in each $X_i$ are ordinary polynomial variables. Keeping the groups independent makes coefficient comparison possible; the powers of $V$ are substituted only after the identities have been proved.}

The ordinary normal restrictions\TermNote{exp:p1:normal-restriction}{the equations obtained on $L_a=0$ after removing the known factor $L_a$} are
\begin{equation}\label{p1:I:ordinary-normal-package}
 a^{[1/P]t}g=L_a h_a,\quad
 a^{[1/P^2]t}M=\mu_a a^{[1/P]t}\pmod {L_a},\quad
 (H|_{L_a})(h_a|_{L_a})=\mu_a^P.
\end{equation}
We prove that the Frobenius endpoint produces the common rational function
required by Lemma~\ref{p1:I:normalizer}. All matrix arguments below are
valid in characteristic two as well as in odd characteristic. The
restriction $P\ge3$ permits separation of argument degrees at most two.

\begin{theorem}[The Frobenius endpoint gives compatible representatives]
\label{p1:I:endpoint-theorem}
Under \eqref{p1:I:endpoint-hyp}, there is a rational homogeneous function
$W$ of degree $d-1$, represented by a quotient whose numerator and
denominator are nonzero at every short vertex, such that
$h_a|_{L_a}=\beta_aW|_{L_a}$ with $\beta_a\ne0$ for every ordinary
point $a$. Thus Lemma~\ref{p1:I:normalizer} gives compatible representatives
on all short secants. If their adjacency graph is connected, one common
rescaling puts these representatives in $\F_q^3$.
\end{theorem}
The next two subsections prove the theorem. The span dimension and
the equality $B=L(U)$ still require the counting tools of
Section~\ref{p1:I:reconstruction}; the corollary below carries these out
when $q=P^3$.

\subsection{The two determinant factors}

\label{exp:p1:coefficient-frobenius}The coefficient-Frobenius twist $\varphi$ raises constant
coefficients to $P$th powers and leaves all variables unchanged\TermNote{exp:p1:twist-example}{$(aX+Y)^\varphi=a^PX+Y$, whereas $(aX+Y)^P=a^PX^P+Y^P$}. Put
$Y=\mathsf C(X_0,\ldots,X_{k_{\rm slot}-1})
\mathsf C^\varphi(X_1,\ldots,X_{k_{\rm slot}})X_0$.
The uniform identities show that, after $X_i=V^{[P^i]}$, this
column is proportional to $V^{[PS]}$. Hence each polynomial
$Y_a(X_{k_{\rm slot}})_b-Y_b(X_{k_{\rm slot}})_a$ becomes zero.
Its degree in each vector argument is at most two, below $P$.
Distinct monomials remain distinct after substitution, by uniqueness
of base-$P$ digits in each coordinate exponent. Thus these three
polynomials were zero before substitution. The coordinates of
$X_{k_{\rm slot}}$ have no common factor, so $Y=fX_{k_{\rm slot}}$
with $f$ polynomial. Linearity in the last argument makes $f$
independent of it. Consequently
\begin{equation}\label{p1:eq:fullcross}
 \mathsf C(X_0,\ldots,X_{k_{\rm slot}-1})
 \mathsf C^\varphi(X_1,\ldots,X_{k_{\rm slot}})X_0
 =f(X_0,\ldots,X_{k_{\rm slot}-1})X_{k_{\rm slot}}.
\end{equation}
The column $Y$ has argument degrees $(2,\ldots,2,1)$; the scalar
$f$ has degrees $(2,\ldots,2)$, all below $P$, and
$f_\Delta=\Theta^{P+1}H^2a_0$ has degree $2d_{\rm nat}<PS$.
This argument also proves the asserted degrees of the scalar $f$.
Define the multilinear partner by
\[
 \mathsf B(X_0,\ldots,X_{k_{\rm slot}-1})X_{k_{\rm slot}}
 =\mathsf C^\varphi(X_1,\ldots,X_{k_{\rm slot}})X_0.
\]
Then $\mathsf C\mathsf B=fI$.

\label{exp:p1:quadratic}\ReadingNote{Quadratic matrix relation}{The minimal polynomial is the monic polynomial of least degree annihilating $C$. For $3\times3$ matrices, $C^2-\alpha C+\beta I=0$ implies $C=\lambda I+uv^t$ for some scalar $\lambda$; $C^2$ is ordinary matrix multiplication.}

\begin{lemma}[Balanced factorization]\label{p1:lem:balanced}
Both matrix polynomials are primitive. There are multilinear polynomials
$\mathcal{F},\mathcal{G}$ such that
\begin{equation}\label{p1:eq:balanced}
 f=\mathcal{F}\mathcal{G},\qquad \det\mathsf C=\mathcal{F}^2\mathcal{G},\qquad \adj\mathsf C=\mathcal{F}\mathsf B.
\end{equation}
On the Frobenius diagonal,
\begin{equation}\label{p1:eq:gauge}
 \boxed{D=\Theta^{P-2}\mathcal{F}_\Delta.}
\end{equation}
Thus $\Theta$ and $\mathcal{F}_\Delta$ are nonzero at every short vertex.
\end{lemma}
\begin{proof}
Fix one vector argument over the rational function field of the others and
normalize an invertible value to $C(s_0)=B(s_0)=I$.
Expanding $CB=fI$ around $s_0$ gives
\[
 B=\alpha I-C,\qquad C^2-\alpha C+\beta I=0,
\]
with $\alpha$ linear and $\beta$ quadratic. A three-dimensional
vector space cannot be a vector space over a quadratic field extension,
so the minimal polynomial cannot be an irreducible quadratic.
If $C$ is scalar the conclusion is immediate. Otherwise write the
quadratic as $(T-a)(T-b)$. Since $(C-aI)(C-bI)=0$,
\[
 \rank(C-aI)+\rank(C-bI)\le3.
\]
One rank is at most one; let $\lambda$ be the corresponding root.
Thus $C-\lambda I$ has rank at most one, and $\lambda$ is the
eigenvalue occurring at least twice. The monic
characteristic equation shows that a reduced denominator of $\lambda$
must be constant; homogeneity makes $\lambda$ linear. Subtracting
$\lambda I$ gives a linear space of matrices of rank at most one.
Such a space has a common row or a common image\TermNote{exp:p1:row-image}{all row spaces, respectively all column spaces, lie in one fixed one-dimensional space}: two rank-one matrices
with both independent rows and independent images would have a sum
of rank two.

In the first argument, $s\mapsto B(s)W$ is invertible over the rational function field by
\eqref{p1:eq:fullcross}. Hence the $B$-space has dimension three and
$C$ is not a scalar matrix over the rational function field. If $C(s)=0$, then $B(s)=\alpha(s)I$;
the coefficient of $t$ in $C(z+ts)B(z+ts)=f(z+ts)I$ forces
$\alpha(s)=0$. Thus the $C$-space also has dimension three.
A common image would give evaluation rank at most two for $B$.
In the common-row case write
\[
 C(s)=\lambda(s)I+(Us)v^t,\qquad Us_0=0,\quad\lambda(s_0)=1.
\]
Then $\rank U=2$ and
\[
 \det[s\mapsto C(s)W]=(v^tW)^2\lambda^t\adj(U)W\ne0,
\]
since the image of $\adj U$ is the line generated by $s_0$.
The shifted definition of $\mathsf B$ transfers invertible
$C$-evaluation in one argument to invertible $B$-evaluation in the next.
Thus both matrix polynomials have full coefficient dimension and invertible
evaluation in every argument. A common factor would consume an argument
degree and make that dimension one, so both matrix polynomials are primitive.

From $\mathsf C\mathsf B=fI$ we have
$\adj\mathsf C=(\det\mathsf C/f)\mathsf B$.
In the reduced fraction $\det\mathsf C/f$, the denominator must
therefore divide every entry of $\mathsf B$ and hence is constant.
Thus $\mathcal F=\det\mathsf C/f$ is polynomial; interchanging the
matrices gives the same conclusion for $\mathcal G=\det\mathsf B/f$. Their
degree is one in every argument. Taking determinants gives
\eqref{p1:eq:balanced}. Finally,
\[
 \mathcal{F}_\Delta=\frac{\det(\Theta M)}{\Theta^{P+1}H^2a_0}
 =\frac{D}{\Theta^{P-2}}.
\]
Nonvanishing follows from Lemma~\ref{p1:ppg:short}.
\end{proof}

\subsection{Restrictions to two-dimensional normal planes}

\begin{lemma}[Crossed pencils in dimension two]\label{p1:lem:binary}
For each integer $i$, let $A_i$ be a nonsingular $2\times2$ matrix
polynomial, linear in each of the consecutive two-coordinate vectors
$x_i,\ldots,x_{i+k_{\rm slot}-1}$. Suppose
\[
 A_i(x_i,\ldots,x_{i+k_{\rm slot}-1})
 A_{i+1}(x_{i+1},\ldots,x_{i+k_{\rm slot}})x_i=f_i x_{i+k_{\rm slot}}.
\]
Then each $A_i$ is primitive, and $f_i/\det A_i$ is a nonzero constant.
\end{lemma}
\begin{proof}
Define $B_ix_{i+k_{\rm slot}}=A_{i+1}x_i$, so $A_iB_i=f_iI_2$.
First-argument evaluation of $B_i$ is invertible. The coefficient matrices of $A_i$ in the first argument cannot all
be proportional: otherwise $A_i=\ell(x_i)A_0$, and $A_iB_i=f_iI_2$
would make $B_i$ a scalar linear form times $A_0^{-1}$, whose
evaluation has rank one.
A two-dimensional nonsingular $2\times2$ pencil has invertible
evaluation with an independent column: normalize to $\langle I,M\rangle$ with
$M$ nonscalar and use $\det(W,MW)\not\equiv0$.
The shift transfers this property to the next argument. Induction,
simultaneously for all $i$, gives coefficient dimension two in every
argument. Hence $A_i$ is primitive, and so is $\adj A_i$ in dimension two.
Now $B_i=(f_i/\det A_i)\adj A_i$ makes that scalar polynomial.
Its degree in every argument is zero, so it is a nonzero constant.
\end{proof}

Fix an ordinary normal $a$. Use integer indices, and put
$a_i=a^{[P^i]}$, $E_i=\ker a_i^t$. Choose $T_0\in\mathrm{SL}_3(\F_q)$
with $a^tT_0=e_3^t$ and set $T_i=T_0^{[P^i]}$.
Restrict $X_i$ to $E_i$ and write
$\mathsf C_i=\mathsf C^{\varphi^i}(X_i,\ldots,X_{i+k_{\rm slot}-1})$.
To pass from \eqref{p1:I:ordinary-normal-package} to independent normal
planes, first take $i=0$ and write $V=T_0(y_1,y_2,0)^t$.
Then $V^{[P^j]}=T_j(y_1^{P^j},y_2^{P^j},0)^t$.
For $c=1,2$, consider
\[
 a_{-2}^t\mathsf C(T_0\widetilde x_0,\ldots,
 T_{k_{\rm slot}-1}\widetilde x_{k_{\rm slot}-1})T_{-1}e_c,
 \qquad \widetilde x_j=(x_{j1},x_{j2},0)^t.
\]
It has constant coefficients and degree one in each pair of variables.
Its diagonal substitution is zero by the ordinary normal identity
multiplied by $\Theta$. Distinct monomials give distinct exponents
$\sum_jP^j\beta_j$ in each coordinate, since every digit is below $P$.
Thus the polynomial itself is zero. Apply the coefficient automorphism
$\varphi^i$ to obtain the same conclusion for every integer $i$:
$\mathsf C_i(E_{i-1})\subseteq E_{i-2}$. Consequently
\[
 T_{i-2}^{-1}\mathsf C_iT_{i-1}
 =\begin{pmatrix}A_i&*\\0&\mu'_i\end{pmatrix}.
\]
These matrices and scalars are nonsingular or nonzero over the rational function field:
each ordinary component has a short vertex where $\Theta M$ is
invertible. Only the normal planes are periodic: $E_{i+n}=E_i$, since $a\in \F_q^3$.
The coefficients need not have period $n$. Raising them to the
$P^i$th power for every integer $i$ gives the required family;
the proof of Lemma~\ref{p1:lem:binary} is finite induction on the argument
position and uses no coefficient periodicity.
Since $k_{\rm slot}=n-2$, so that $E_{i+k_{\rm slot}}=E_{i-2}$, \eqref{p1:eq:fullcross} restricts
to the crossed identity of Lemma~\ref{p1:lem:binary}.
For the restricted factors $\mathcal{F}_i,\mathcal{G}_i$, the determinant-one frames give
\[
 \det A_i\,\mu'_i=\mathcal{F}_i^2\mathcal{G}_i,\qquad f_i=\mathcal{F}_i\mathcal{G}_i.
\]
Hence $\mu'_i=c_i\mathcal{F}_i$ with $c_i\in\overline{\F_q}^*$. Returning to the original
ordinary component gives, with no deleted components,
\begin{equation}\label{p1:eq:glue}
 (\Theta|_{L_a})\mu_a=c_a\mathcal{F}_\Delta|_{L_a}.
\end{equation}
Define the homogeneous rational function
\begin{equation}\label{p1:eq:W}
 \boxed{W=\frac{\mathcal{F}_\Delta^P}{\Theta^PH},\qquad\deg W=d-1.}
\end{equation}
Equations~\eqref{p1:I:ordinary-normal-package} and \eqref{p1:eq:glue} give
\begin{equation}\label{p1:eq:normalizer}
 h_a|_{L_a}=\beta_a W|_{L_a},\qquad \beta_a=c_a^P\in\overline{\F_q}^*
\end{equation}
on every ordinary component. By \eqref{p1:eq:gauge}, $W$ is defined
and nonzero at every short vertex; $W$ need not be polynomial.

This proves Theorem~\ref{p1:I:endpoint-theorem}.

\subsection{An unconditional application in a cubic extension}
\begin{corollary}\label{p1:I:cubic-extension}
Let $q=P^3$, where $P=p^e\ge4$, and let $B$ be a nontrivial small
minimal blocking set of exact exponent $e$ in $\PG(2,q)$. Then
\[
 B=L(U),\qquad \dim_{\F_P}U=4,
 \qquad |B|\in\{P^3+P^2+1,\ P^3+P^2+P+1\}.
\]
In particular, a small minimal blocking set in $\PG(2,64)$ of
exact exponent two is $\F_4$-linear.
\end{corollary}
\begin{proof}
Proposition~\ref{p1:I:ordinary-bound} gives $d\ge P$. The nonnegative
degrees in Theorem~\ref{p1:I:uniform-main} give
\[
 0\le\deg D=P-(P-2)m-1,\qquad
 0\le\deg H=Pm-d+1.
\]
The first implies $m\le1$ for $P\ge4$; the second excludes $m=0$.
Hence $m=1$ and $P\le d\le P+1$. The matrix $M$ is linear, so
\eqref{p1:I:endpoint-hyp} holds with $\Theta=1$ and one argument.
Theorem~\ref{p1:I:endpoint-theorem} applies.

By Lemma~\ref{p1:ppg:gap}, each ordinary point is on at least
\[
 m_0=P^2+1-\left\lfloor\frac{d-1}{P-1}\right\rfloor=P^2
\]
short secants. All their points are ordinary. Put $N_0=|O|$ and
$b_{\max}=P^3+P^2+P+1$. The preceding bounds give
$q<N_0\le |B|\le b_{\max}<4q/3$.
The short-secant graph has minimum degree at least $Pm_0=q$;
two different components would contain more than $2q$ vertices,
which is impossible. Lemma~\ref{p1:I:normalizer} therefore puts all
compatible representatives in $\F_q^3$ by one rescaling.

Let $r_U$ be their $\F_P$-span dimension. Proposition~\ref{p1:I:code-bound}
applies, since $2+2Pm_0-N_0>2q/3$. Also
$(P-1)m_0+1=P^3-P^2+1<q$, so
\[
 P^{r_U}\le\frac{PN_0((P-1)m_0+1)}{2+2Pm_0-N_0}
 <2Pq<P^2q=P^5.
\]
Thus $r_U\le4$. Since $N_0>q$, a space of dimension at most three,
with at most $(q-1)/(P-1)$ projective directions, cannot contain
the representatives. Hence $r_U=4$.

For their span $U$, put $B'=L(U)$. A line is a six-dimensional
$\F_P$-space in the ambient nine-dimensional space, so it meets $U$.
Thus $B'$ blocks and $|B'|\le b_{\max}$. Its intersections are
$1\pmod P$: in the vector-space intersection with a line, partition
the nonzero vectors into point fibers, divide their sizes by $P-1$,
and reduce modulo $P$.
The special-point bound is
$\sigma\le d-1-(P-3)m=d-P+2$. Since $O\subseteq B\cap B'$,
\[
 |B\mathbin\triangle B'|\le b_{\max}-|B|+2\sigma
 \le\begin{cases}P+4,&d=P,\\6,&d=P+1.\end{cases}
\]
Both bounds are less than $2(P-1)P^2+2$ for $P\ge4$.
The signed-pencil bound, Lemma~\ref{p1:I:signed}, forces $B=B'$.
The two possible sizes follow from $d\in\{P,P+1\}$.
\end{proof}

The cubic-extension classifications in~\cite{p1:Polverino,p1:PolverinoStorme}
assume characteristic at least seven; see also~\cite{p1:LSV} for the
higher-dimensional result. The proof above includes characteristics
two, three and five when $P\ge4$. We make no priority claim for the
individual small orders. For $P=3$, the degree of $D$ allows $m=2$,
so this argument does not give the Frobenius endpoint. Nor does it settle
$P=4$ for arbitrary $n$.

\section{The binary case: notation and exact size}\label{p1:secmap:9}
Throughout the remaining sections $P=2$. Let $q=2^n\ge4$,
$s=q/2$, and let $B\subset\PG(2,q)$ be a
nontrivial small minimal blocking set of exact exponent one.
Thus every line meets $B$ oddly, some intersection is $3\pmod4$, and
$|B|<3(q+1)/2$. For $n=1$ there is no nontrivial small example.

We begin with the polynomial identities proved in
Theorem~\ref{p1:I:uniform-main}. Put $V=(X,Y,Z)^t$ and
$L_{\mathsf A}={\mathsf A}^tV$. Entrywise powers are denoted by brackets. The general construction gives a primitive homogeneous vector $g$ with
\begin{gather}
 |B|=q+2d+1,\quad 1\le d\le \frac{q}{4},\qquad
 R:=\prod_{{\mathsf A}\in B}L_{\mathsf A}=\det(V,\fr Vq,\fr g2),\quad \deg g=d, \label{p1:B:inputR}\\
 \mathcal M_1=\{h:L_{\mathsf A}\mid \fr {\mathsf A}{1/2}^th\ ({\mathsf A}\in B)\},
 \quad \mathcal M_1\text{ has a polynomial basis of degrees }d,s,s+d+1. \label{p1:B:inputM}
\end{gather}
In particular $\mathcal M_1$ has no nonzero homogeneous element of degree
less than $d$. Further, with $\delta=\fr g2\cdot\nabla$,
\begin{gather}
 \delta g=bg+c\fr Vs,\quad \delta b=b^2,\quad\delta c=bc,
 \quad c\ne0,\quad \deg b=2d-1,\quad\deg c=3d-1-s, \label{p1:B:inputdelta}\\
 J=\det(g,g_X,g_Y)/Z\ne0,\qquad J=cD^2,
 \quad\deg D=\frac{q}{4}-1. \label{p1:B:inputJ}
\end{gather}
The identities $\delta b=b^2$ and $\delta c=bc$ follow by applying
$\delta$ to $\delta g=bg+cV^{[s]}$ and using $\delta^2=0$
from Proposition~\ref{p1:I:extra-identities}; independence of
$g,V^{[s]}$ identifies the two coefficients.
The vector $V\times\fr g2$ is primitive. 

All local coordinate changes below respect these identities. Explicitly,
under $V=AW$, $A\in\operatorname{GL}_3(\F_q)$, put $a=\det A$.
For $R'=R(AW)$ the transformed data are
\begin{gather*}
 g'=a^{1/2}(A^{[1/2]})^{-1}g(AW),\qquad c'=a^{3/2}c(AW),\\
 J'=a^2J(AW),\qquad D'=a^{1/4}D(AW).
\end{gather*}
The canonical representatives defined below transform by the single
binary-linear map $u\mapsto a^{1/2}A^tu$, so their sum relations are
independent of the chosen coordinates.

\begin{corollary}\label{p1:B:size}
For every such $B$, one has
\[
 \boxed{|B|=3q/2+1.}
\]
\end{corollary}
\begin{proof}
Proposition~\ref{p1:I:ordinary-bound} gives $d\ge q/4$, while
\eqref{p1:pp:linebound} gives $d\le q/4$. Thus $d=q/4$.
\end{proof}

\section{Sharp binary restrictions from complete splitting}\label{p1:secmap:10}
Fix ${\mathsf A}\in B$ and use compatible coordinates with $L_{\mathsf A}=Z$. Write
\[
 g_3=Zh,\quad \gamma_1=g_1|_{Z=0},\quad \gamma_2=g_2|_{Z=0},\quad h_0=h|_{Z=0},
\]
and put
\[
 \mathcal T=X^s \gamma_2+Y^s \gamma_1,\quad E_{\mathsf A}=X\gamma_2^2+Y\gamma_1^2,
 \quad \mathcal Q=XY^q+YX^q.
\]
Expanding \eqref{p1:B:inputR} gives
\[
 R/Z=(X^s g_2+Y^s g_1)^2+Zh^2\mathcal Q+
 Z^{q-1}(Xg_2^2+Yg_1^2).
\]
Consequently
\begin{equation}\label{p1:B:normal}
 (R/Z)|_{Z=0}=\mathcal T^2,\qquad
 \partial_Z(R/Z)|_{Z=0}=h_0^2\mathcal Q.
\end{equation}
The roots of $\mathcal T$ correspond exactly to the secants through ${\mathsf A}$. A secant
with $2r+1$ points gives a root of multiplicity $r$, and
\begin{equation}\label{p1:B:sumr}
 \sum r=s+d.
\end{equation}
Suppose $h_0\ne0$, and let $k_{\mathsf A}$ be the number of these roots.
Proposition~\ref{p1:I:ordinary-bound} gives $\ord h_0\ge r-1$ at every
root and $k_{\mathsf A}\ge s+1$.
Every root of $\mathcal T$ is rational and is a root of $E_{\mathsf A}$. The latter is
nonzero: $X\gamma_2^2=Y\gamma_1^2$ would force $\gamma_1=\gamma_2=0$ by parity, contrary to
primitivity of $g$. Therefore
\begin{equation}\label{p1:B:pencilbound}
 s+1\le k_{\mathsf A}\le2d+1\le s+1.
\end{equation}
Call ${\mathsf A}$ \emph{ordinary} if $L_{\mathsf A}^2\nmid\fr {\mathsf A}{1/2}^tg$, and
\emph{special} otherwise. These terms refer to this algebraic definition.
Let $O$ be the ordinary set. Equality throughout \eqref{p1:B:pencilbound}
shows that, on each ordinary line, $E_{\mathsf A}$ has exactly $\frac{q}{2}+1$ distinct
rational roots, $\gamma_1,\gamma_2$ have no common zero, and
\begin{equation}\label{p1:B:sharp}
 \mathcal T=\kappa_{\mathsf A} h_0 E_{\mathsf A},\qquad
 \ord_{\ell^*}h_0=(|B\cap\ell|-3)/2.
\end{equation}
In particular $g$ has no common zero, even over an extension, on an ordinary line.

\section{The canonical denominator}\label{p1:secmap:11}
The bivariate polynomial $\gamma_1(\gamma_2)_X+\gamma_2(\gamma_1)_X$
(homogeneous in $X$ and $Y$) equals $Y\omega^2$ for a form
$\omega$ of degree $\frac{q}{4}-1$. Restricting \eqref{p1:B:inputdelta} gives
\[
 (c|_{Z=0})\mathcal T=E_{\mathsf A}\omega^2,
 \qquad J|_{Z=0}=h_0\omega^2.
\]
Here $\omega\ne0$: otherwise the coprime forms $\gamma_1,\gamma_2$ are squares,
so every root of $\mathcal T$ has multiplicity at least two, contradicting
$\deg \mathcal T=\frac{3q}{4}<2(\frac{q}{2}+1)$. Cancellation using \eqref{p1:B:sharp} and
\eqref{p1:B:inputJ} therefore yields
\begin{equation}\label{p1:B:Drestrict}
 c|_{L_{\mathsf A}}\ne0,\quad D|_{L_{\mathsf A}}=\sqrt{\kappa_{\mathsf A}}\,h_0,
 \quad \ord_{\ell^*}(D|_{L_{\mathsf A}})=(|B\cap\ell|-3)/2.
\end{equation}
A special line divides $J$, hence divides $cD$. Since $\deg(cD)=\frac{q}{2}-2$,
initially $|O|\ge q+3$.

\section{Special factors, canonical triples, and local multiplicity}\label{p1:secmap:12}
Every ordinary point lies on at least $\frac{q}{4}+2$ three-secants: it has $\frac{q}{2}+1$
secants and their parameters $r$ sum to $\frac{3q}{4}$. A three-secant containing
an ordinary point consists entirely of ordinary points. At its dual
vertex ${\mathsf V}$, $g({\mathsf V})\ne0$ and $D({\mathsf V})\ne0$. In local coordinates the leading
transverse terms of $g$ are two linear forms. Their independence follows
from squarefreeness of the cubic tangent cone of $R$. Thus $J({\mathsf V})\ne0$,
whereas a special line divides $J$.

Each ordinary point has a unique representative $u_{\mathsf A}$ normalized by
\begin{equation}\label{p1:B:canonical}
 L_{\mathsf A}^2\mid \fr{u_{\mathsf A}}{1/2}^tg+(u_{\mathsf A}^tV)D.
\end{equation}
Existence and uniqueness follow from \eqref{p1:B:Drestrict}: if the initial
normal quotient is $\alpha D$, scale the representative by $\alpha^2$.
In fact $\alpha=1/\sqrt{\kappa_{\mathsf A}}$, so the scaling factor
is $1/\kappa_{\mathsf A}\in\F_q$. The canonical vectors already
belong to $\F_q^3$.
Differentiating \eqref{p1:B:canonical} at a triple vertex gives an invertible
semilinear fixed-vector equation on its two-dimensional normal plane.
Its fixed space has binary dimension at most two. To see this, a shortest
linear dependence among binary-independent fixed vectors, normalized to
have one coefficient equal to one, is unchanged by squaring; minimality
then makes every coefficient binary. Therefore the three canonical
nonzero vectors satisfy
\begin{equation}\label{p1:B:triples}
 u_{\mathsf A}+u_{\mathsf A'}+u_{\mathsf A''}=0.
\end{equation}

\begin{proposition}\label{p1:B:support}
A point ${\mathsf A}\in B$ is special if and only if $L_{\mathsf A}\mid D$. Consequently
\[
 |B\setminus O|\le \frac{q}{4}-1,\qquad |O|\ge\frac{5q}{4}+2.
\]
\end{proposition}
\begin{proof}
An ordinary line does not divide $D$ by \eqref{p1:B:Drestrict}. Suppose a
special ${\mathsf A}$ has $D|_{L_{\mathsf A}}\ne0$. At a secant through ${\mathsf A}$ with $k>0$
ordinary points and size $2r+1$, let $v$ be the order of $D|_{L_{\mathsf A}}$
at the dual vertex. Triple closure gives $r\ge2$. Write $r_D$ for the
full local multiplicity of $D$, so $r_D\le v$. If $v<r-1$, the degree-$r_D$
initial form vanishes in the $k$ ordinary directions, by
\eqref{p1:B:Drestrict}; hence $k\le r_D\le v$. If $v\ge r-1$, then
$k\le2r\le2v+2\le4v$. Summing over the pencil gives
$|O|\le q-4$, contrary to the already established $|O|\ge q+3$.
\end{proof}

At any vertex incident with an ordinary line, of multiplicity $2r+1$,
one furthermore has
\begin{equation}\label{p1:B:localmult}
 \ord_{\mathsf V} D=r-1.
\end{equation}
Here is the needed local argument. The coprime transverse leading pair
has degree $r$: locally the initial term of $R$ is
$xB_r^2+yA_r^2$, and its squarefreeness makes $A_r,B_r$ coprime.
Choose infinity outside the incident directions and put
$A(t)=A_r(t,1)$, $B(t)=B_r(t,1)$. At an incident direction $t_i$,
$B(t_i)\ne0$ and $A(t_i)=\sqrt{t_i}B(t_i)$. The polynomial
$\Omega=A'B+AB'=(AB)'$ is a square of degree at most $2r-2$. At an ordinary
direction, \eqref{p1:B:sharp} makes
$A'(t_i)+\sqrt{t_i}B'(t_i)$ nonzero, hence $\Omega(t_i)\ne0$.
At a special direction that derivative is zero, because the normal
component is divisible by the square of its defining line. Thus $\Omega$
is nonzero, and each special direction is a root of multiplicity at
least two. There are at most $r-1$ special directions and at least
$r+2$ ordinary directions. A smaller-degree
initial form of $D$ would vanish in all of these directions, impossible.
This proves \eqref{p1:B:localmult}. In particular every irreducible component
of $D$ meets an ordinary line with its minimum local intersection order\TermNote{exp:p1:intersection}{the multiplicity of the zero after restriction to the line; for two curves it is the usual local count in B\'ezout's theorem}.

\begin{proposition}\label{p1:B:indices}
Put $w_{\mathsf V}=r_{\mathsf V}-1$ at every arrangement vertex of multiplicity
$2r_{\mathsf V}+1$, and $t=\frac{q}{4}-1$. There are nonnegative integers $\epsilon_i$,
zero on ordinary components, such that
\[
 \mathcal W:=\sum_{\mathsf V}w_{\mathsf V}=t+\sum_i\epsilon_i,
 \qquad \mathcal W_i:=\sum_{{\mathsf V}\in L_i}w_{\mathsf V}=t+\epsilon_i.
\]
These identities allow common zeros of $g_1,g_2,g_3$.
\end{proposition}
\begin{proof}

The common zeros of the coordinates of $V\times g^{[2]}$ are finite,
by primitivity. Count each with its local intersection multiplicity;
the total is $q^2/4+q/2+1$. To prove this, put $\mathcal G=g^{[2]}$,
$m_{\mathcal G}=q/2$, and choose a line $Z=0$, over a finite extension
if necessary, avoiding these common zeros. The forms $X\mathcal G_3-Z\mathcal G_1$ and $Y\mathcal G_3-Z\mathcal G_2$ have degree
$m_{\mathcal G}+1$ and no common component. At infinity their simultaneous vanishing has the same multiplicity
as $Z=\mathcal G_3=0$: the combination
$X\mathcal G_2-Y\mathcal G_1$ is nonzero there and can be cancelled.
These intersections contribute $m_{\mathcal G}$, so B\'ezout leaves $(m_{\mathcal G}+1)^2-m_{\mathcal G}=m_{\mathcal G}^2+m_{\mathcal G}+1$.
Here $\mathcal G_3|_{Z=0}$ is nonzero: otherwise the induced map of that line
would have a fixed point, contrary to its choice.

This count includes zeros over the algebraic closure. Every such zero
lies on $R$. We now use rational coordinates at a vertex of multiplicity
$h=2r+1$. Put $\mathcal E=(g_1^2+xg_3^2,g_2^2+yg_3^2)$ and let
$F_h$ be the product of the incident line equations through the origin.
For each such line $\ell$, the polynomial
$\ell_x\mathcal E_1+\ell_y\mathcal E_2$ is divisible by $\ell$.
Cramer's rule therefore gives
\[
 \mathcal E=A\cdot(x,y)+B\cdot((F_h)_y,(F_h)_x),\qquad B(0,0)\ne0,
\]
where $A,B$ are rational functions with denominators nonzero at the vertex.
The two displayed column vectors have determinant $F_h$. Each
Cramer numerator is divisible by every incident line, proving the
asserted cancellation of its denominator. Moreover
$R=(y+y^q)\mathcal E_1+(x+x^q)\mathcal E_2$. Thus the determinant
with $(x,y)$ equals $BF_h$, and its initial term is the nonzero
tangent cone of $R$ because $h<q$, proving $B(0,0)\ne0$.
Here and below, $I(F,G)$ denotes the intersection multiplicity of
two curves at the current point. Let $e_i$ be the order at
$\mathsf V$ of the restriction of the vector $V\times g^{[2]}$
to incident line $\ell_i$. With $x=0$ chosen outside all incident
directions, this is $I(\mathcal E_1,\ell_i)$: on $\ell_i$ the two
coordinates of $\mathcal E$ are proportional, with a nonzero constant
of proportionality. The local intersection multiplicity is
\[
 \mu_{\mathsf V}=\sum_i e_i-(h-1)=1+\sum_i(e_i-1).
\]
To verify this, choose $x=0$ outside the incident lines and use
$x\mathcal E_2+y\mathcal E_1=BF_h$:
$I(\mathcal E_1,F_h)=I(\mathcal E_1,x)+I(\mathcal E_1,\mathcal E_2)$,
where the first term on the right is $h-1$.

On each arrangement line put
$\epsilon_i=\sum_{\mathsf V}(e_i-1)+f_i$, where $f_i$ counts all geometric smooth zeros of
$\mathcal E$ on that line. Such smooth zeros are simple: they cannot
be common zeros of $g$, because such a zero makes $\nabla R=0$, and at a
zero with $g\ne0$ the matrix of first derivatives of $\mathcal E$ is a nonzero scalar
matrix. On an ordinary line, \eqref{p1:B:sharp} shows that
$E_{\mathsf A}$ has exactly $q/2+1$ simple roots, all at vertices;
hence $e_i=1$ there and there are no additional smooth zeros.
Thus ordinary lines have $\epsilon_i=0$. If $v$ is the number of
vertices, the local multiplicity formula and the degree-$\frac{q}{2}+1$ count of zeros, with multiplicity,
on each line give, respectively,
\[
 v+\sum_i\epsilon_i=\frac{q^2}{4}+\frac{q}{2}+1,\qquad
 \sum_{\mathsf V}(2r_{\mathsf V}+1)+\sum_i\epsilon_i=(\frac{3q}{2}+1)(\frac{q}{2}+1).
\]
Eliminating $v$, and using $\sum_{{\mathsf V}\in L_i}r_{\mathsf V}=\frac{3q}{4}$ on a single
component, gives
\[
 \mathcal W:=\sum_{\mathsf V}(r_{\mathsf V}-1)=\frac{q}{4}-1+\sum_i\epsilon_i,
 \qquad \mathcal W_i:=\sum_{{\mathsf V}\in L_i}(r_{\mathsf V}-1)=\frac{q}{4}-1+\epsilon_i.
\]
\end{proof}

\begin{proposition}\label{p1:B:specialexists}
For $n\ge3$, there is at least one special point. Thus
$1\le |B\setminus O|\le \frac{q}{4}-1$.
\end{proposition}
\begin{proof}
If all components were ordinary, Proposition~\ref{p1:B:indices} would give
$\mathcal W=\mathcal W_i=\frac{q}{4}-1$ for every $i$. Therefore each vertex of positive weight
would lie on every arrangement component. Such a vertex exists when
$\frac{q}{4}>1$, but more than $q+1$ distinct rational lines cannot be concurrent.
\end{proof}

\section{The denominator splits completely}\label{p1:secmap:13}
Put $\mathcal S=B\setminus O$, $\sigma=|\mathcal S|$, and $t=\frac{q}{4}-1$.
For now $\frac{q}{4}>1$, so $1\le \sigma\le t$.

First, $D$ has no rational line factor outside the arrangement. Indeed,
if $L_A\mid D$ for $A\notin B$, every line through $A$ containing
an ordinary point has at least five $B$-points, by
\eqref{p1:B:Drestrict}. Write its intersection size as $2r_\ell+1$.
The external pencil gives $\sum_{\ell\ni A}r_\ell=\frac{q}{4}$. Hence
\[
 |O|\le\sum_{\substack{\ell\ni A\\\ell\cap O\ne\varnothing}}
 (2r_\ell+1)\le\frac52\sum_{\ell\ni A}r_\ell
 =\frac{5q}{8},
\]
contrary to $|O|\ge\frac{5q}{4}+2$.

\begin{theorem}\label{p1:B:fullsupport}
There are positive integers $e_{\mathsf A}$ such that, up to a nonzero scalar,
\[
 \boxed{D=\prod_{{\mathsf A}\in\mathcal S}L_{\mathsf A}^{\nu_{\mathsf A}},\qquad
        \nu_{\mathsf A}=2^{e_{\mathsf A}}-1,\qquad \sum_{{\mathsf A}\in\mathcal S}\nu_{\mathsf A}=\frac{q}{4}-1.}
\]
Every arrangement vertex has an ordinary incident line; $g$ has
no common projective zero. At a secant of size $2r+1$,
\begin{equation}\label{p1:B:rweights}
 r=1+\sum_{{\mathsf A}\in\ell\cap\mathcal S}\nu_{\mathsf A}.
\end{equation}

\end{theorem}
\begin{proof}
If $q=4$, Proposition~\ref{p1:B:support} gives $\mathcal S=\varnothing$
and $D$ is a nonzero constant; the assertions hold with the empty
product. Assume $q\ge8$, so Proposition~\ref{p1:B:specialexists} gives
$\sigma\ge1$. We first handle the concurrent configurations. If $\sigma=1$, write $L$
for the special component. Proposition~\ref{p1:B:indices} gives $\mathcal W=\mathcal W_L$,
so every positive weight lies on $L$. Each ordinary component has
total weight $t$ and meets $L$ once. Its intersection with $L$ thus
has weight $t$, or arrangement multiplicity $\frac{q}{2}+1$. There are $\frac{3q}{2}$
ordinary components, with $\frac{q}{2}$ through each such vertex, giving
exactly three vertices on $L$. By \eqref{p1:B:localmult}, $D$ has
multiplicity $t=\deg D$ at each. Having this multiplicity at two
distinct points forces $D$ to be a scalar multiple of $L^t$:
in coordinates placing the points at $(1,0,0),(0,1,0)$, the two
multiplicity conditions leave only the monomial $Z^t$.

Suppose next that $\sigma\ge2$ special components are concurrent at ${\mathsf V}$.
Proposition~\ref{p1:B:indices} and inclusion--exclusion make the weight
outside their union equal to $(\sigma-1)(w_{\mathsf V}-t)$. Since $w_{\mathsf V}\le t$,
nonnegativity forces $w_{\mathsf V}=t$. There are $\frac{q}{2}+1$ components through
${\mathsf V}$, of which at most $\sigma\le \frac{q}{4}-1$ are special. Thus
\eqref{p1:B:localmult} gives $\ord_{\mathsf V}D=t=\deg D$. In coordinates
${\mathsf V}=(0,0,1)$, $D$ depends only on $X,Y$, so all its geometric line
factors pass through ${\mathsf V}$. Choose an ordinary component avoiding ${\mathsf V}$;
one exists since there are more than $q+1$ ordinary components.
Its restriction of $D$ splits over $\F_q$ by \eqref{p1:B:Drestrict}.
Projection from ${\mathsf V}$ then makes all the pencil factors rational.
The preceding exclusion of outside rational factors shows that they
are precisely special components. In either concurrent case every
vertex has an ordinary incident component: for $\sigma\ge2$, the only
possible all-special vertex was ${\mathsf V}$ itself.

It remains to consider nonconcurrent special components, so $\sigma\ge3$.
Proposition~\ref{p1:B:support} permits the factorization
$D=(\prod_{i=1}^\sigma L_i^{\nu_i})G$, where $G$ has no special line
factor. Set $h_G=\deg G$, so $\sum\nu_i=t-h_G$. At a vertex ${\mathsf V}$ let
$j_{\mathsf V}$ be the number of special incident lines and $w_{\mathsf V}=r_{\mathsf V}-1$.
The global index identities in Proposition~\ref{p1:B:indices} give
\begin{equation}\label{p1:B:globalweight}
 \sum_{j_{\mathsf V}\ge2}(j_{\mathsf V}-1)w_{\mathsf V}
       =(\sigma-1)t+\sum_{j_{\mathsf V}=0}w_{\mathsf V}\ge(\sigma-1)t.
\end{equation}
Always $w_{\mathsf V}\le\ord_{\mathsf V}D$ at such a vertex. Equality holds when an
ordinary line is incident, by \eqref{p1:B:localmult}. Otherwise all
$m_{\mathsf V}=j_{\mathsf V}$ incident lines are special, and
$w_{\mathsf V}=(j_{\mathsf V}-3)/2<j_{\mathsf V}\le\ord_{\mathsf V}D$.

The contribution of the known line factors is exactly
\[
 \sum_{j_{\mathsf V}\ge2}(j_{\mathsf V}-1)\sum_{i:{\mathsf V}\in L_i}\nu_i
     =(\sigma-1)\sum_i\nu_i=(\sigma-1)(t-h_G).
\]
Indeed each other special line is counted once on any fixed $L_i$;
concurrency is counted by $j_{\mathsf V}-1$. B\'ezout on the $\sigma$ special lines
also gives $\sum j_{\mathsf V}\ord_{\mathsf V}G\le \sigma h_G$. They are not concurrent, so
$j_{\mathsf V}\le \sigma-1$, and therefore
\[
 \sum_{j_{\mathsf V}\ge2}(j_{\mathsf V}-1)\ord_{\mathsf V}G
 \le\frac{\sigma-2}{\sigma-1}\sigma h_G=(\sigma-1)h_G-\frac{h_G}{\sigma-1}.
\]
Combining these inequalities with \eqref{p1:B:globalweight} yields
$(\sigma-1)t\le(\sigma-1)t-h_G/(\sigma-1)$. Hence $h_G=0$.

Equality now holds throughout. There are no all-special vertices,
and $w_{\mathsf V}=0$ whenever $j_{\mathsf V}=0$.

We have therefore proved, in all configurations, that $D$ is supported
on the special components and every vertex has ordinary incidence.
Every common zero of $g$ is a vertex of $R$ and would have only special
incident lines; hence none exists. Formula \eqref{p1:B:localmult} proves
\eqref{p1:B:rweights} at every vertex.

Fix a special line $L_i$ and let $k_i$ be its number of vertices.
Summing \eqref{p1:B:rweights} along it, using $\sum r_{\mathsf V}=\frac{3q}{4}$, gives
\[
 \frac{3q}{4}=k_i(\nu_i+1)+(\frac{q}{4}-1-\nu_i),\qquad
 \boxed{(k_i-1)(\nu_i+1)=\frac{q}{2}.}
\]
Thus $\nu_i+1$ is a power of two, and $\nu_i$ is odd.

\end{proof}

\section{Higher normal multiplicities}\label{p1:secmap:14}
\begin{proposition}\label{p1:B:normalhigher}
For a special ${\mathsf A}$ one has
\[
 L_{\mathsf A}^{\nu_{\mathsf A}+1}\mid {\mathsf A}^{[1/2]t}g.
\]
\end{proposition}
\begin{proof}
Take $L_{\mathsf A}=Z$ and write $g_3=Z^\alpha h$ with $h_0=h|_{Z=0}\ne0$.
Here $\alpha\le \frac{q}{4}$: absence of common zeros prevents $g_3$ from vanishing
identically, since the two remaining positive-degree forms would
have a common projective zero. The coefficient of
$Z^{2\alpha-1}$ in $R/Z$ is $h_0^2\mathcal Q$: the other terms below degree
$q-1$ in $Z$ are squares. At a secant root of parameter $r$, complete
splitting gives order at least $2r-(2\alpha-1)$ for this coefficient.
Thus $\ord h_0\ge r-\alpha$ whenever $r\ge\alpha$.

If $\alpha\le\nu_{\mathsf A}+1$, all secants through ${\mathsf A}$ have
$r\ge\nu_{\mathsf A}+1\ge\alpha$ by \eqref{p1:B:rweights}. Summing gives
$\frac{q}{4}-\alpha\ge\frac{3q}{4}-\alpha k_{\mathsf A}$, or
$\alpha(k_{\mathsf A}-1)\ge\frac{q}{2}$. The exact pencil formula in
Theorem~\ref{p1:B:fullsupport} forces $\alpha\ge\nu_{\mathsf A}+1$. If
$\alpha>\nu_{\mathsf A}+1$ the claim is immediate.

\end{proof}

\section{Every secant satisfies the canonical sum relation}\label{p1:secmap:15}
For every secant $\ell$ one now has
\begin{equation}\label{p1:B:allsecants}
 \boxed{\sum_{{\mathsf A}\in O\cap\ell}u_{\mathsf A}=0.}
\end{equation}
Here is the full local proof. Put the dual vertex at $(0,0,1)$.
Choose the direction at infinity outside the incident lines, and
write the other directions as $x+t_i y=0$. Write $A(t),B(t)$ for
the coprime transverse leading forms of $g$, of degree $r$,
evaluated at $(t,1)$, and $d_0(t)$ for the degree-$(r-1)$ initial
form of $D$ evaluated there. Set
$E(t)=tB(t)^2+A(t)^2$ and $\Omega=A'B+AB'$. Full splitting makes $d_0$
exactly the product of the special direction forms to their weights
$\nu_{\mathsf A}$. By Proposition~\ref{p1:B:normalhigher}, the corresponding normal
combination of $A,B$ is divisible by the direction form to order at
least $\nu_{\mathsf A}+1$. The derivative combination $A'B+AB'$ vanishes to order at least $\nu_{\mathsf A}$;
since $\Omega$ is a square and $\nu_{\mathsf A}$ is odd, this is at least
$\nu_{\mathsf A}+1$. Thus $C=\Omega/d_0$ is polynomial, has degree at most $r-1$,
and vanishes at every special direction.

At an ordinary root $t_i$ the canonical normal, computed from
\eqref{p1:B:canonical}, is $u_i=\lambda_i(1,t_i)$ in the two-dimensional
normal plane. Extract the degree-$r$ initial part of
\eqref{p1:B:canonical}, divide by $x+t_i y$, and evaluate at this
direction. Since $A(t_i)=\sqrt{t_i}B(t_i)$, this gives
\[
 \sqrt{\lambda_i}
 =\frac{A'(t_i)+\sqrt{t_i}B'(t_i)}{d_0(t_i)}
 =\frac{\Omega(t_i)}{B(t_i)d_0(t_i)}.
\]
Consequently
\[
 u_i=\frac{C(t_i)^2}{E'(t_i)}(1,t_i).
\]
The polynomial $E$ is squarefree of degree $2r+1$, and
$E'=B^2$. Lagrange interpolation says the sums of these two
coordinates over all its roots are zero, since the numerators have
degrees at most $2r-2$ and $2r-1$. Special roots contribute zero,
proving \eqref{p1:B:allsecants} for both parities of $r$.

\section{A general code of secant relations}\label{p1:secmap:16}
Let $\mathcal C$ be the binary code of subsets $C\subseteq O$ meeting
every secant of $B$ evenly\TermNote{exp:p1:code}{identify a subset with its binary indicator vector; addition is symmetric difference}. Formula \eqref{p1:B:allsecants} places all
binary coordinate words of the canonical representatives in
$\mathcal C$. A coordinate word records the values of one $\F_2$-linear
coordinate on all canonical representatives. The dimension of $\mathcal C$
is thus an upper bound for their binary span dimension.

\begin{proposition}\label{p1:B:code}
Assume $q\ge8$. Every nonempty $C\in\mathcal C$ has even size $k$ with
\[
 \frac{q}{2}+2\le k\le q.
\]
If $S_C=\prod_{{\mathsf A}\in C}L_{\mathsf A}$, there is a homogeneous polynomial $A_C$
of degree $k/2-\frac{q}{4}-1$ such that
\[
\nabla S_C=A_C^2(V\times g^{[2]}),\qquad \delta S_C=0.
\]
Moreover, if $\sigma=|\mathcal S|$, then
\[
 \prod_{{\mathsf A}\in\mathcal S}L_{\mathsf A}^{(\nu_{\mathsf A}-1)/2}\mid A_C,
 \qquad k\ge\frac{3q}{4}+1-\sigma.
\]
\end{proposition}
\begin{proof}
Fix ${\mathsf A}\in C$. Each of its $\frac{q}{2}+1$ secants has at least one further
$C$-point, proving the lower bound. A secant of size $2r+1$ has at
most $2r$ points of $C$, and hence at most $2r-1$ other $C$-points.
Since $\sum r=\frac{3q}{4}$, this gives
$k-1\le\frac{3q}{2}-(\frac{q}{2}+1)=q-1$.
The pencil through any special point partitions $C$ into even parts,
so $k$ is even.

Parity of monomial exponents in the even-degree form $S_C$ gives
$\nabla S_C=V\times h^{[2]}$ for a tuple $h$ of degree $k/2-1$.
For every ${\mathsf A}\in B$, the restriction of $S_C$ to $L_{\mathsf A}$ is a square
or zero: its nonzero root multiplicities are the even intersections
with $C$. Thus its gradient modulo $L_{\mathsf A}$ is normal to $L_{\mathsf A}$.
In coordinates $L_{\mathsf A}=Z$, this says $Yh_3^2=Xh_3^2=0$ modulo $Z$,
so $Z\mid h_3$. Therefore $h\in\mathcal M_1$.
Since $\deg h\le\frac{q}{2}-1$, the basis degrees in \eqref{p1:B:inputM} force
$h=A_Cg$. The gradient formula follows, and its scalar product with
$g^{[2]}$ is zero. Also $A_C\ne0$, since a nonconstant squarefree
product of lines cannot have zero gradient.

For the extra divisibility, fix a special line $Z=0$, with
$m=\nu_{\mathsf A}+1$, and set $v=\ord_Z A_C$. The first odd normal
coefficient of $S_C$ is that of $Z^{2v+1}$; the gradient formula
identifies it as
$((A_C/Z^v)|_{Z=0})^2E_{\mathsf A}$. The bivariate polynomial $E_{\mathsf A}$ (homogeneous in $X$ and $Y$) has simple
zeros at the $k_{\mathsf A}=1+\frac{q}{2m}$ vertices, since $g$ has no common zero.
If $c_{\mathsf V}=|C\cap\ell_{\mathsf V}|$, the product of the defining lines gives
\[
 \ord_{\mathsf V}((A_C/Z^v)|_{Z=0})\ge c_{\mathsf V}/2-v-1.
\]
Summing, with $\sum c_{\mathsf V}=k$, yields
$\deg A_C-v\ge k/2-(v+1)k_{\mathsf A}$, hence
$(v+1)(k_{\mathsf A}-1)\ge \frac{q}{4}$. Thus $v\ge m/2-1=(\nu_{\mathsf A}-1)/2$.
The product of these factors has degree $(\frac{q}{4}-1-\sigma)/2$;
comparison with $\deg A_C=k/2-\frac{q}{4}-1$ proves the final bound.
\end{proof}

\section{The third moment closes the dimension bound}\label{p1:secmap:17}\label{p1:B:sec:closure}
We now prove Theorem~\ref{p1:I:binary-main}. If $q=2$ there is no
nontrivial small example. For $q=4$, Corollary~\ref{p1:B:size} gives seven
points, all ordinary by Proposition~\ref{p1:B:support}, and every secant
has three points. The resulting seven-point linear space is the Fano
plane. Choose three noncollinear points. Relation~\eqref{p1:B:triples}
makes the seven canonical representatives exactly the nonzero sums
of their three representatives. Their binary span therefore has
dimension three and represents $B$.

Assume $q\ge8$. Proposition~\ref{p1:B:specialexists} excludes an empty
special set, and Theorem~\ref{p1:B:fullsupport} applies to every remaining
configuration. Put
\[
 \sigma=|\mathcal S|,\qquad N_0=|O|=\frac{3q}{2}+1-\sigma,
 \qquad 1\le\sigma\le \frac{q}{4}-1.
\]
Let $k_{\mathcal C}=\dim_2\mathcal C$ and $K=2^{k_{\mathcal C}}$. We will show $K<4q$.

\medskip\noindent\textit{Counting the three-secants.}
Write $b_3$ for the number of three-secants of $B$.
For a line $\ell$ meeting $\mathcal S$, put
$t_\ell=|\ell\cap\mathcal S|$ and $o_\ell=|\ell\cap O|$.
Every ordinary point lies on $\frac{q}{2}+1$ secants. The secants with no
special point are exactly the three-secants, by \eqref{p1:B:rweights}.
Since the pencil through an ordinary point partitions $\mathcal S$,
summing the exact three-secant count over $O$ gives
\begin{equation}\label{p1:B:improvedtriple}
 3b_3=N_0(\frac{q}{2}+1-\sigma)
       +\sum_{\ell:\,t_\ell\ge2}o_\ell(t_\ell-1).
\end{equation}
For each line in the sum, \eqref{p1:B:rweights} gives
\[
 o_\ell=3+2\sum_{{\mathsf A}\in\ell\cap\mathcal S}\nu_{\mathsf A}-t_\ell
       \ge t_\ell+3\ge t_\ell.
\]
Counting ordered pairs of special points therefore gives the essential
improvement
\begin{equation}\label{p1:B:triplelower}
 3b_3\ge N_0(\frac{q}{2}+1-\sigma)+\sigma(\sigma-1).
\end{equation}

\medskip\noindent\textit{The binary moments.}
Choose a generator matrix of $\mathcal C$\TermNote{exp:p1:generator}{its rows form a basis of $\mathcal C$; $\operatorname{wt}$ counts nonzero coordinates, and a moment is an average of powers}, with columns
$v_{\mathsf A}\in\F_2^{k_{\mathcal C}}$ indexed by ${\mathsf A}\in O$. These columns are nonzero and
pairwise distinct. Indeed, every binary coordinate functional on the
canonical representatives $u_{\mathsf A}$ gives a word of $\mathcal C$ by
\eqref{p1:B:allsecants}; such functionals distinguish their distinct
nonzero vectors.

For $a\in\F_2^{k_{\mathcal C}}$, define
\[
 x(a)=\sum_{{\mathsf A}\in O}(-1)^{a\cdot v_{\mathsf A}}
      =N_0-2\operatorname{wt}(aG).
\]
The elementary identity $\sum_{a\in\F_2^{k_{\mathcal C}}}(-1)^{a\cdot v}=0$ for $v\ne0$ gives
\[
 \frac1K\sum_a x(a)^2=N_0,\qquad
 \frac1K\sum_a x(a)^3=6A_3,
\]
where $A_3$ is the number of unordered triples of distinct columns
whose sum is zero. In the second identity, terms with repeated indices
average to zero because every column is nonzero; a distinct zero-sum
triple contributes its six orderings. Every three-secant is a check
equation of the code, so $A_3\ge b_3$.

Proposition~\ref{p1:B:code} says that every nonzero word has weight at least
$\frac{3q}{4}+1-\sigma$. Thus, for $a\ne0$,
\[
 x(a)\le\sigma-1=:u,\qquad x(a)^2(x(a)-u)\le0.
\]
The zero word has $x(0)=N_0$. Averaging this inequality and using
\eqref{p1:B:triplelower} gives
\[
 N_0(q+3-3\sigma)+2\sigma(\sigma-1)
 \le 6A_3-uN_0
 \le \frac{N_0^2(N_0-u)}K.
\]
The left side is positive. Consequently
\begin{equation}\label{p1:B:codebound}
 K\le
 \frac{N_0^2(N_0-u)}
 {N_0(q+3-3\sigma)+2\sigma(\sigma-1)}
 <4q.
\end{equation}
Here is an explicit verification of the strict last inequality.
Set $h=\frac{q}{4}-1-\sigma$, so $0\le h\le \frac{q}{4}-1$. Subtracting the numerator
from $4q$ times the denominator gives
\begin{align*}
 &\frac{3q^3}{16}+\frac{59q^2}{4}+40q-16\\
 &\quad +(\frac{51q^2}{8}+48q-24)h+h^2(14q-12-2h)>0.
\end{align*}
Every displayed contribution is nonnegative, and the constant term
is positive; in particular $14q-12-2h\ge\frac{27q}{2}-10>0$.
Since $K$ is a power of two, \eqref{p1:B:codebound} implies
\[
 \boxed{\dim_2\mathcal C\le n+1.}
\]

\medskip\noindent\textit{Reconstructing the blocking set.}
Let $U=\langle u_{\mathsf A}:{\mathsf A}\in O\rangle_2$. Its binary coordinate words
belong to $\mathcal C$, so $\dim_2U\le n+1$.
There are at least $\frac{5q}{4}+2>q-1$ distinct nonzero vectors $u_{\mathsf A}$,
forcing $\dim_2U=n+1$.

Put $B'=L(U)$. Every projective line is a binary $2n$-space in the
ambient $3n$-space, so it meets $U$ nontrivially. Thus $B'$ blocks.
Its line intersections are odd: the $2^r-1$ nonzero vectors in each
intersection with $U$ partition into point fibers of odd sizes.
Also $|B'|\le2^{n+1}-1=2q-1$, and $O\subseteq B\cap B'$.
Hence $B\mathbin\triangle B'$ is an even set of size at most
\[
 (\frac{3q}{2}+1)+(2q-1)-2(\frac{5q}{4}+2)=q-4.
\]
A nonempty even set in a plane of order $q$ has at least $q+2$
points: through any one of its points, its $q+1$ lines require
distinct further points. Therefore $B\mathbin\triangle B'$ is empty,
and $B=L(U)$. This proves Theorem~\ref{p1:I:binary-main}.

\section*{Acknowledgements}
OpenAI ChatGPT assisted with mathematical exploration, computations and
drafting. The author retains responsibility for verifying the
proofs and the submitted text.

\label{part1:last}
\StartPaper{2}{Part 2: the planar Linearity Conjecture for P at least 5}
{Part 2: The planar Linearity Conjecture\\for intersection exponents $P\ge5$}
{Let $B$ be a nontrivial small minimal blocking set of exact exponent $e$
in $\PG(2,p^h)$, and put $P=p^e$. Using the Frobenius representation
and local geometry developed in the first paper, we give a proof of
linearity for $P\ge5$ throughout the full classical exponent interval.
The representing space has dimension $h/e+1$ over $\F_P$, and
$\F_P$ is the maximum geometric field of linearity\TermNote{exp:p2:field}{the largest subfield allowing a linear representation of the same point set}. The proof successively groups monomials by their exponent residues
and obtains a matrix polynomial linear in each of several independent
vector variables. Its restrictions to lines provide point representatives which
have $\F_P$-multiples filling two-dimensional $\F_P$-spaces on the $(P+1)$-secants. Counting pairs and triples then bounds their span.
For $P=5$ a determinant of degree three supplies the additional step. The argument for $P\ge7$ includes characteristic
two, in particular every $P=2^e\ge8$.}
\begin{PartContents}
\ContentsEntry{1}{Introduction}{p2:secmap:1}
\ContentsEntry{2}{The size interval and the polynomial data from Paper I}{p2:secmap:2}
\ContentsEntry{3}{The first Frobenius extraction}{p2:secmap:3}
\ContentsEntry{4}{The repeated coefficient step and its degree bounds}{p2:secmap:4}
\ContentsEntry{5}{Cancelling parameters and extracting the next scale}{p2:secmap:5}
\ContentsEntry{6}{The Hasse coefficient calculation at each step}{p2:secmap:6}
\ContentsEntry{7}{Completion of the descent}{p2:secmap:7}
\ContentsEntry{8}{Reconstruction and identification for exponents at least seven}{p2:secmap:8}
\ContentsEntry{9}{The entire exponent-one interval in characteristic five}{p2:secmap:9}
\ContentsEntry{10}{The upper endpoint is attained}{p2:secmap:10}
\ContentsEntry{11}{A known strictly intermediate example}{p2:secmap:11}
\end{PartContents}
\section{Introduction}\label{p2:secmap:1}
For $q=p^h$, let $B\subset\PG(2,q)=\PG(\F_q^3)$ be a small minimal blocking set,
where small means $|B|<3(q+1)/2$. Its exact exponent is the largest
$e$ for which all line intersections are $1\pmod{p^e}$.
Put $P=p^e$. The exponent theorem of Sz\H onyi and Sziklai gives $1\leq e\mid h$
\cite[Corollary~11.35]{p2:DSc};
we therefore write $q=P^n$.
\label{exp:p2:linear-rank}We use $L(U)=\{\langle u\rangle_{\F_q}:0\ne u\in U\}$;
the rank of this representation is $\dim_{\F_P}U$.

\begin{theorem}\label{p2:II:main}
If $P\ge5$, there is a subspace $U\le\F_q^3$ over $\F_P$ such that
\[
 B=L(U),\qquad\dim_{\F_P}U=n+1.
\]
Moreover $\F_P$ is the maximum geometric field of linearity of $B$.
\end{theorem}

The theorem covers the full size interval permitted by the exact
exponent. No hypothesis $|B|\le(Pq-1)/(P-1)$ is imposed: this bound
is a consequence of the reconstructed representation. The distinction
is essential near the upper end of the classical interval.

The prime and prime-square cases are covered by
\cite{p2:Blokhuis,p2:Szonyi}. The cubic-extension theorems of
\cite{p2:Polverino,p2:PolverinoStorme,p2:LSV} have characteristic at least seven.
For R\'edei-type sets, see \cite{p2:BBBSS,p2:Ball,p2:Missing}.
The phrase ``maximum geometric field of linearity'' has the meaning
of \cite[Definitions~1.3--1.4]{p2:JenaVDV}: the largest field allowing
some representation of the same point set.

The polynomial approach starts with the established theory of
Sz\H{o}nyi~\cite{p2:Szonyi}, Sz\H{o}nyi--Weiner~\cite{p2:SW} and
Sziklai~\cite{p2:Sziklai}. The first paper~\cite{p2:PUB1} supplies the
Frobenius identities, their integration in every characteristic, local
regularity and the weight gap\TermNote{exp:p2:weight-gap}{the exclusion of line-intersection sizes $Pr+1$ for $2\le r<P$}, and the reconstruction lemmas. We cite
those results here and prove the coefficient descent\TermNote{exp:p2:descent}{successive extraction of coordinate powers, reducing the degree of the remaining coefficients}. References prefixed
by \textup{I.} refer to numbered results or equations in~\cite{p2:PUB1}.

\paragraph{Notation and the geometric picture.}
We use $V=(X,Y,Z)^t$, $L_a=a^tV$, and entrywise powers
$V^{[P^i]}=(X^{P^i},Y^{P^i},Z^{P^i})^t$; $[1/P]$ takes the
unique coordinatewise $P$th roots in $\F_q$.
A point $a\in B$ corresponds to the dual line $L_a=0$.
A short secant has $P+1$ points; its dual point is a short vertex.
\label{exp:p2:ordinary-special}
Write $a^{[1/P]t}g=L_a h_a$. The point $a$ and its dual line are
ordinary if $h_a|_{L_a}$ is not the zero polynomial, and special
if it is zero, equivalently if $L_a^2\mid a^{[1/P]t}g$.
In coordinates $L_a=Z$, special means that the first normal coefficient
of $g_3=Zh_a$ vanishes identically. The local results of Part~1 show
that special points lie on no short secant, whereas ordinary points
lie on many; this is why the final identification must recover the
special points as well.
\label{exp:p2:matrix-terms}
A primitive polynomial vector or matrix has no common nonconstant
factor among its entries. Nonsingular means that the determinant is
not the zero polynomial; squarefree means that no irreducible factor
is repeated. The adjugate is the transposed cofactor matrix, with
$M\adj M=(\det M)I$.

The proof has three stages. First we extract the largest coordinate
powers from the matrix $M$ supplied by Paper I. We then repeat the
same coefficient comparison at successively smaller powers until the
remaining coefficients are linear. Finally Paper I turns this matrix
formula into compatible representatives of the ordinary points;
a dimension count and a pencil count recover the entire set.
\label{exp:p2:frobenius-endpoint}For $n\ge3$, the target formula is the \emph{Frobenius endpoint}:
\[
 \Theta M=\mathsf C(V,V^{[P]},\ldots,V^{[P^{n-3}]}),\qquad\Theta\ne0,
\]
where $\mathsf C$ has constant coefficients and is linear in each
independent vector argument separately (multilinear).
It completes the coefficient descent; the representing space and
its dimension are obtained afterwards.

For $P\ge7$, the repeated step uses a determinant identity whose
degree in each independent vector variable is six, below $P$.
At $P=5$ we use a different determinant, of degree three, and transfer
its identity from the last vector variable to the preceding ones.
Neither argument assumes odd characteristic in the case $P\ge7$.
All characteristic-dependent derivative calculations are in Paper I.

The theorem concerns the planar conjecture. Published reductions for
higher dimensional blocking sets have additional hypotheses~\cite{p2:VDV};
no unrestricted higher dimensional conclusion is asserted here.
Together with the binary exponent-one theorem in~\cite{p2:PUB1}, the
present result leaves the unrestricted cases $P=3$ and $P=4$
outside the completed arguments. Part~1, Corollary~\Iref{I:cubic-extension},
does cover $q=P^3$, $P\ge4$, including $(P,q)=(4,64)$.
The ternary case requires a separate analysis.

\section{The size interval and the polynomial data from Paper I}\label{p2:secmap:2}
If $B$ is a line, its exact exponent has $P=q$ and $n=1$,
and the theorem follows from a two-dimensional space over $\F_q$.
Henceforth assume that $B$ is nontrivial.

The size theorem of Paper I and a convenient classical upper bound give,
for $P\ge7$,
\begin{equation}\label{p2:full:literature}
 q+q/P+1\le |B|<q+q/P+q/P^2+3q/P^3.
\end{equation}
Both estimates follow from Part~1: the lower one is
Proposition~\Iref{I:ordinary-bound}, and the upper one is
Corollary~\Iref{I:coarse-upper}, proved from its quadratic moment bound.
The latter recovers \cite[Lemma~1]{p2:VDV} with $k=1$, $p_0=P$;
neither its proof here nor that bound requires odd characteristic.
For $P=5$ we use Lemma~\Iref{I:quadratic} directly.
Throughout the proof the exponent is exact; a mere cardinality label is not substituted
for the intersection congruence.

For $n=2$, the quadratic in Lemma~\Iref{I:quadratic} has roots
$P^2+P+1$ and $P^3+1$. Smallness selects the smaller-root interval,
and the lower bound forces $|B|=q+P+1$. Bruen's equality
characterization~\cite{p2:Bruen} makes $B$ a Baer subplane.
These reductions apply for every $P\ge5$, so assume $n\ge3$.
Until the end of Section~\ref{p2:sec:reconstruction}, assume $P\ge7$;
the case $P=5$ is treated in Section~\ref{p2:sec:five-entire-interval}.
Put
\[
 S=P^{n-3},\quad
 |B|=q+Pd+1,\quad d=PS+\rho.
\]
The interval gives
\begin{equation}\label{p2:full:rho}
 0\le\rho<\left(1+\frac3P\right)S,
 \qquad PS\le d<(P+2)S.
\end{equation}
The nonnegativity of $\rho$ is useful throughout the proof.

\begin{lemma}[Input from the first paper]\label{p2:lem:input}
There are a primitive vector $g$, a primitive nonsingular homogeneous
matrix $M$, and nonzero polynomials $H,a_0,D$ with $H$ squarefree and
\begin{align}
 Hg&=M^{[P]}V,&Mg&=Ha_0V^{[q/P^2]},&\det M&=H^2a_0D,\label{p2:eq:uniform}\\
 Dg&=NV^{[q/P^2]},&N&=\adj M/H.\nonumber
\end{align}
If $\deg M=S+t$, then
\begin{align}
 \deg H&=Pt-\rho+1,&\deg a_0&=S+2\rho-(P-1)t-1,\nonumber\\
 \deg D&=2S-(P-2)t-1,&\deg N&=2S+\rho-(P-2)t-1.\label{p2:eq:degrees}
\end{align}
Every ordinary component contains at least
\begin{equation}\label{p2:eq:mzero}
 m_*=(P^2-2)S+2
\end{equation}
short vertices. There are at least $q+(P-1)d+2>q$ ordinary components
and at most $\sigma\le d-1-(P-3)\deg M\le d-1$ special components.
At every short vertex $H,a_0,D$ and $\det M$ are nonzero, all incident
components are ordinary, and their directions form an $\F_P$-subline.
On an ordinary component,
\begin{equation}\label{p2:eq:normals}
 a^{[1/P]t}g=L_a h_a,\quad
 a^{[1/P^2]t}M=\mu_a a^{[1/P]t}\pmod{L_a},\quad
 (H|_{L_a})(h_a|_{L_a})=\mu_a^P.
\end{equation}
\end{lemma}
\begin{proof}
The uniform representation, local regularity, special-component and
normal-restriction assertions are Theorem~\Iref{I:uniform-main} and
Lemmas~\Iref{ppg:short}--\Iref{ppg:normals} of~\cite{p2:PUB1}.
Their short-vertex lower bound is
$q/P+1-\lfloor(d-1)/(P-1)\rfloor$. Here
$(d-1)/(P-1)<2S$, by~\eqref{p2:full:rho}, which gives~\eqref{p2:eq:mzero}.
\end{proof}

\label{exp:p2:H-reciprocal}The reciprocal matrix $N=\adj M/H$ is polynomial because
every irreducible factor of $H$ divides every cofactor of $M$ and
$H$ is squarefree. Thus the whole product $H$ cancels entrywise;
also $MN=Ha_0D\,I$. This is the divisibility proved in
Part~1, Proposition~\Iref{pp:identities}.

Nonnegativity of $\deg H$ gives the useful lower bound
\begin{equation}\label{p2:full:t-lower}
 t\ge(\rho-1)/P,\qquad t\ge0.
\end{equation}

The second inequality follows from $\rho\ge0$ and the integrality of $t$.

\paragraph{Coordinate changes used in the local calculations.}
Let $V_{\rm old}=CV$, $C\in\operatorname{GL}_3(\F_q)$, and choose
the compatible normalization of Part~1, Section~\Iref{I:normalization}.
For $A=B_0M$ it gives
\[
 A_{\rm new}=(C^{-1})^{[1/P^2]}A_{\rm old}(CV)C^{[1/P]}.
\]
If $\widehat A(u)=\sum_i u_i\widehat A_i$ comes from an extraction
at scale $\tau$\TermNote{exp:p2:scale}{$\tau$ is the coordinate-power exponent being extracted}, the same substitution gives
\[
 \widehat A_{\rm new}(u;V)=
 (C^{-1})^{[1/P^2]}
 \widehat A_{\rm old}(C^{[\tau]}u;CV)C^{[1/P]}.
\]
Consequently, for $T(u)=\sum_i u_i(A^{[P]})^{-1}\widehat A_i^{[P]}$,
\[
 T_{\rm new}(u;V)=C^{-1}T_{\rm old}(C^{[P\tau]}u;CV)C.
\]
The parameter change is invertible. Thus an identity quantified over
all $u$ at a normalized short vertex transfers back to the original
coordinates. These rules also apply to the first extraction, with
$\tau=S$; harmless nonzero scalar normalizations may be absorbed in
the corresponding multipliers.

\subsection{Algebraic tools used at each descent step}\label{p2:II:toolkit}
All coefficient fields used for taking roots are finite fields or
their algebraic closure, so every element has a unique $P$th root.
Auxiliary choices may be made in a finite extension of $\F_q$.
The final point representatives will be brought back to $\F_q$.
\label{exp:p2:power-ideal}We use the abbreviation
$J_s=(X^s,Y^s,Z^s)$: membership means that every monomial is divisible
by at least one of $X^s,Y^s,Z^s$. Thus $F\in J_s$ is equivalent to
$F=X^sF_1+Y^sF_2+Z^sF_3$. No general theory of ideals is needed.

A matrix polynomial $C(Z_1,\ldots,Z_j;V)$ will be linear and
homogeneous in each three-coordinate vector $Z_i$ separately.
For example, with two arguments it is $\sum_{a,b}(Z_1)_a(Z_2)_bC_{ab}(V)$.
\label{exp:p2:degree-diagonal}Its \emph{coefficient degree} is the degree in $V$ of the matrices
$C_{ab}$, with all the $Z_i$ independent. Substituting prescribed
powers of $V$ for the $Z_i$ will be indicated explicitly, or by a
subscript $\Delta$ (the Frobenius diagonal). Degree comparisons are always made before this
substitution unless stated otherwise.

\label{exp:p2:pencil}For $C(Y)=\sum Y_iC_i$, a \emph{matrix pencil} means just this linear
combination of fixed matrices. Its coefficient dimension is the
dimension of their span. Evaluation at $W$ means the map
$Y\mapsto C(Y)W$, with matrix $[C_1W\ C_2W\ C_3W]$.
When other arguments occur, their coordinates are independent variables
in the coefficient field. We say the evaluation is nonsingular when
this determinant is a nonzero polynomial in $W$. A pencil is in
\emph{common-row form} if, after multiplication by a fixed invertible
matrix, it is $\lambda(Y)I+u(Y)v^t$, where $v^t$ is fixed and
$\lambda,u$ are linear in $Y$. In common-image form, the
column $u$ is fixed and the row $v(Y)^t$ varies linearly instead.

\ReadingNote{The matrix form used below}{After multiplying by a fixed invertible matrix, many pencils have the form $\lambda(Y)I+u(Y)v^t$, with one fixed row $v^t$. This is called common-row form; the extra term has rank at most one and its determinant is easy to factor.}

\begin{lemma}[Power basis and extraction]\label{p2:II:power-basis}
Let $P$ be a power of the characteristic. Every polynomial has a unique
expansion
\[
 F(V)=\sum_{0\le\alpha_i<P}V^\alpha F_\alpha(V)^P.
\]
The same monomials form a basis of $\F_q(V)$ over $\F_q(V)^P$. If $s=P\tau$
and $Ff^P\in (X^s,Y^s,Z^s)$, then
$F_\alpha f\in(X^\tau,Y^\tau,Z^\tau)$ for each $\alpha$.
If $F(W)\in \F_q(V)^P[W]$ has test-vector degree less than $P$ and
$F(V)=0$, then all its coefficients vanish.
\end{lemma}
\begin{proof}
Write each monomial exponent uniquely as $\alpha_i+P\beta_i$ and take
$P$th roots of its coefficient in the perfect field. This proves the
polynomial assertion. Multiplying a rational relation by a common
denominator lying in the power subfield reduces it to that assertion.
Discard the monomials divisible by $X^{P\tau}$, $Y^{P\tau}$, or
$Z^{P\tau}$. Exactly those with all three exponents below $P\tau$ remain. A term with residue
$\alpha$ survives precisely when the corresponding exponent of
$F_\alpha f$ has all coordinates below $\tau$. Distinct residues
cannot cancel. The last assertion is the rational independence just
proved, with the coefficient field kept fixed.
\end{proof}

\begin{lemma}[Distinct coordinate powers do not cancel]\label{p2:II:separation}
Let $s_i=P^{j-i}s$ for $1\le i\le j$. Suppose
$F(Z_1,\ldots,Z_j;V)$ has degree less than $P$ in each vector argument
$Z_i$ and degree in $V$ of the coefficients less than $s$. Then
\[
 F(V^{[s_1]},\ldots,V^{[s_j]};V)=0\quad\Longrightarrow\quad F=0.
\]
The assertion also holds componentwise for vector or matrix polynomials.
\end{lemma}
\begin{proof}
A monomial on the diagonal has exponent vector
$\alpha+s\sum_iP^{j-i}\beta_i$, where each coordinate of $\alpha$
is below $s$ and each coordinate of $\beta_i$ is below $P$.
Reduction modulo $s$, followed by uniqueness of the base-$P$ digits,
recovers $\alpha$ and every $\beta_i$. Thus distinct formal monomials
have distinct diagonal monomials. In particular no cancellation is
created by the specialization. Notice that degree bounds must hold
before specialization; vanishing of ordinary derivatives alone is not
a substitute for these bounds.
\end{proof}

\begin{lemma}[Relations among three coordinate powers]\label{p2:II:koszul}
\label{exp:p2:cross-matrix}Write $[u]_\times$ for the matrix satisfying $[u]_\times v=u\times v$.
If $f_1X^s+f_2Y^s+f_3Z^s=0$, then $(f_1,f_2,f_3)$ is a polynomial
combination of
\[
 (Y^s,-X^s,0),\quad (Z^s,0,-X^s),\quad (0,Z^s,-Y^s).
\]
If the $f_i$ are homogeneous of degree $d_0$, the coefficients have
degree $d_0-s$. Consequently there is no nonzero such polynomial relation for
$d_0<s$. For matrices, $F V^{[s]}=0$ therefore means
$F=E[V^{[s]}]_\times$, with $\deg E=\deg F-s$.
\end{lemma}
\begin{proof}
Every monomial of $f_3Z^s$ is divisible by $X^s$ or $Y^s$, since
the same is true of $-f_1X^s-f_2Y^s$. Multiplication by $Z^s$
does not change those two exponents, so $f_3=AX^s+BY^s$. Substitution gives
$(f_1+AZ^s)X^s+(f_2+BZ^s)Y^s=0$. Coprimality of $X^s,Y^s$
gives $f_1+AZ^s=CY^s$ and $f_2+BZ^s=-CX^s$.
These formulas give the three displayed generators and their degrees.
Apply this row by row for the matrix assertion, adjusting signs to the
convention $[u]_\times v=u\times v$.
\end{proof}

\begin{lemma}[Finite choices and interpolation]\label{p2:II:interpolation}
(i) Given finitely many geometric points\TermNote{exp:p2:geometric}{points over the algebraic closure, not necessarily over the original finite field} at which a polynomial family
does not vanish simultaneously, there is a homogeneous linear
combination, after equalizing degrees, nonzero at all those points.
The choice can be made over a finite extension of $\F_q$.

(ii) Suppose a homogeneous form $F$ of degree $d_0$ vanishes at more
than $d_0$ distinct specified points on each of more than $d_0$ distinct
lines. Then $F=0$.
\end{lemma}
\begin{proof}
For (i), first choose a linear form avoiding the finite set, and use
its powers to equalize degrees. At each point, the bad coefficient
choices form one proper hyperplane. A finite union of such hyperplanes
does not cover the coefficient space over $\overline{\F_q}$. The finitely
many coordinates of a good choice belong to one finite extension of
the finite field $\F_q$.
For (ii), restriction to each line is a bivariate form (homogeneous in its two coordinates) with more roots
than its degree, hence zero. Each line equation divides $F$.
The product of the distinct line equations has degree greater than
$d_0$, forcing $F=0$.
\end{proof}

In every application below, $d_0<m_*$ (or $d_0<m_0$ for $P=5$) verifies
the first interpolation threshold; the number of ordinary components
is larger than $q$, which verifies the second. Auxiliary choices are
not asserted to exist over the original finite field; their eventual
descent is part of the reconstruction argument.

\begin{lemma}[Recognition of a quadratic matrix space]\label{p2:II:quadratic-space}
\label{exp:p2:quadratic}Let $\mathcal A$ be a three-dimensional matrix space\TermNote{exp:p2:matrix-space}{a vector space whose elements are matrices} over an infinite
field, containing $I$. Suppose the polynomial identity
$\det(W,TW,T^2W)=0$ holds for its member with independent parameters $T$ and an independent
test vector $W$. Then subtracting a scalar linear function of the
parameters gives a rank-at-most-one matrix space. This space has a
common image or a common row. Nonsingular evaluation at an independent vector excludes
the common-image alternative.
\end{lemma}
\begin{proof}
Recall that the minimal polynomial is the monic polynomial of least
degree annihilating the matrix. A quadratic relation means
$T^2-\alpha T+\beta I=0$, with ordinary matrix multiplication. If its degree is three, there is a
column $W$ for which $W,TW,T^2W$ are independent. Here is the
three-dimensional verification over an algebraic closure. With three
distinct eigenvalues, take the sum of their eigenvectors. With two
eigenvalues and a quadratic repeated factor, choose vectors $u,v,w$
with $(T-\lambda I)v=u$, $(T-\lambda I)u=0$, $Tw=\mu w$, and
$\mu\ne\lambda$; take $W=v+w$. With one eigenvalue, choose $W$
with $(T-\lambda I)^2W\ne0$. In each case the three columns are
independent. Thus the assumed determinant identity forces degree at
most two, in the original field as well.

Thus the minimal polynomial over the rational function field of the
parameters has degree at most two. An irreducible quadratic is impossible:
it would make the three-dimensional space a vector space over a
quadratic field extension. The eigenvalue $\lambda$ that occurs at
least twice is therefore a rational function of the parameters.
It is polynomial: substitute its reduced numerator and denominator
in the monic characteristic equation; any irreducible denominator
factor would then divide the numerator. Scaling the parameters shows
that $\lambda$ is homogeneous of degree one. If the two eigenvalues
are distinct, $T-\lambda I$ has rank one. If there is just one,
$(T-\lambda I)^2=0$, which in dimension three again implies rank at most
one. The rank minors vanish identically, including at special parameters.

Finally write two nonzero rank-one matrices as $uv^t$ and $xy^t$.
If both column pairs and row pairs are independent, their sum has rank
two. Fixing one nonzero member therefore forces every other member to
share either its image or its row. Members of opposite types would
again have a rank-two sum, so one choice holds for the whole space.
In the common-image case $C(u)W$ lies in the fixed plane
$\Span(W,u_0)$, and its evaluation matrix is singular. This proves
the last assertion in every characteristic.
\end{proof}

\section{The first Frobenius extraction}\label{p2:secmap:3}\label{p2:sec:seed}

Write
$N=C_N\overline N$, $D=C_N\overline D$, with $\overline N$
primitive, and put
\[
 z=\deg C_N,\qquad
 \delta_0=\deg\overline D=2S-(P-2)t-1-z.
\]
The reciprocal identity gives
$\overline D M^{[P]}V=H\overline N V^{[PS]}$.
Differentiate it in coordinate $V_j$, multiply by $\overline D$,
and subtract $(\partial_j\overline D)$ times the undifferentiated
identity. Since the derivatives of $V^{[PS]}$ and $M^{[P]}$ vanish,
\[
 \overline D^2M^{[P]}e_j=
 \bigl(\overline D\,\partial_j(H\overline N)
       -(\partial_j\overline D)H\overline N\bigr)V^{[PS]}\in J_{PS}.
\]
Expand
\[
 \overline D^2
 =\sum_{0\le\alpha_i<P}V^\alpha D_\alpha^P.
\]
Lemma~\ref{p2:II:power-basis} gives
\[
 D_\alpha M_{ij}\in(X^S,Y^S,Z^S)
 \quad\hbox{for every }\alpha,i,j,
 \qquad
 \deg D_\alpha\le\left\lfloor\frac{2\delta_0}{P}\right\rfloor.
\]
Because $\overline D$ is nonzero at every short vertex, its residue
coefficients do not all vanish at any such vertex.  Choose a linear
form nonzero at all of these finitely many vertices, homogenize the
nonzero $D_\alpha$ to their largest degree, and choose a linear
combination $G$ outside the finitely many forbidden hyperplanes.
Over a finite extension of the coefficient
field this gives simultaneously
\[
 GM\in(X^S,Y^S,Z^S)\quad\hbox{entrywise},\qquad
 g_0:=\deg G\le\left\lfloor\frac{2\delta_0}{P}\right\rfloor,
\]
and $G$ is nonzero at every short vertex.  Write
$GM=\sum_iV_i^S\widetilde A_i$.

The coefficient degree needed to separate the $PS$ blocks is at most
\begin{align*}
 \delta_0+P(t+g_0)+1
 &\le6S-2(P-3)t-2-3z\\
 &\le6S-2-3z<PS.
\end{align*}
Here we used $t\ge0$ and $P\ge7$.
Thus
\[
 \overline D\widetilde A_i^{[P]}V
 =G^PH\overline N_i.
\]
Primitivity of $\overline N$ implies $\overline D\mid G^PH$.
Cancel the greatest common factor $C$ of all entries of the
$\widetilde A_i$, and write $\widetilde A_i=CA_i$. Since $M$ is
primitive and $C\mid GM$, we have $C\mid G$; put $B_0=G/C$.
The rational scalar $H_2=B_0^PH/\overline D$ is polynomial because
the identities $A_i^{[P]}V=H_2\overline N_i$ and primitivity of
$\overline N$ exclude a reduced denominator. A repeated irreducible
factor of $H_2$ would divide all derivatives of $H_2\overline N_i$,
and hence every entry of every $A_i^{[P]}$, contrary to collective
primitivity. Thus $H_2$ is squarefree. We have therefore
proved the same seed identities
\begin{equation}\label{p2:full:seed}
 A=B_0M=\sum_iV_i^SA_i,\quad
 A_i^{[P]}V=H_2\overline N_i,\quad
 \overline D=B_0^PH/H_2,\quad
 AB_*=FI,
\end{equation}
with collectively primitive $A_i$, all actual scalar factors nonzero
at short vertices, and
\[
 B_*=[A_1^{[P]}V\ A_2^{[P]}V\ A_3^{[P]}V],
 \qquad F=B_0^{P+1}H^2a_0.
\]

Let $r=\deg A_i=t+\deg B_0$ and $h_2=\deg H_2$.  Then
\[
 h_2=Pr+(P-2)t-2S-\rho+z+2.
\]
Modulo a prime factor of $H_2$, differentiating
$A_i^{[P]}V=H_2\overline N_i$ shows that the $A_i$ have rank at most
one and share a right row. More explicitly, modulo an irreducible
factor $f$ of $H_2$, differentiation gives
$A_i^{[P]}\equiv\overline N_i(\nabla H_2)^t$.
The gradient is not identically zero modulo $f$, since the coefficient
field is perfect and $H_2$ is squarefree. Frobenius is injective on
the fraction field of $\overline{\F_q}[V]/(f)$, so the vanishing
$2\times2$ minors descend to the $A_i$. Their common row descends
in the same way, or after taking unique $P$th roots in an algebraic
closure of that fraction field. Thus $H_2^2$ divides
$\det(\sum u_iA_i)$, giving $2h_2\le3r$.
Combining it with $t\ge(\rho-1)/P$ yields
\begin{align*}
 (2P-3)r
 &\le4S+2\rho-2(P-2)t-2z-4\\
 &\le4S+\frac{4\rho}{P}-2-\frac4P-2z\\
 &<\left(4+\frac4P+\frac{12}{P^2}\right)S.
\end{align*}
For $P\ge7$ this implies $0\le r<S/2$.  \begin{lemma}[Seed identities]\label{p2:lem:seed}
The construction above gives~\eqref{p2:full:seed}, with primitive
coefficient family, $0\le r<S/2$, and all displayed actual scalar
factors nonzero at every short vertex.
\end{lemma}
\begin{proof}
Only the last nonvanishing assertion remains to be made explicit. The chosen
$G$, and $C_N,\bar D$, are nonzero at every short vertex. Every factor
of $B_0$ occurs in $\bar D=B_0^PH/H_2$ with multiplicity at least
$P-1$. Thus $B_0,H_2$ and the crossed scalar are nonzero there.
\end{proof}

\label{exp:p2:hasse}Here $\mathcal H_w^{[j]}F$ means the coefficient of $t^j$ in
$F(V+tw)$; for a vector or matrix it is taken entrywise.
The transverse directions at $(0,0,1)$ are the affine directions
$(w_1,w_2,0)$.

\begin{lemma}[First common-row pencil and strong bound]\label{p2:lem:firstrow}
The pencil $A_u=\sum_i u_iA_i$, after multiplication by the inverse of an invertible member,
has the form $\lambda(u)I+u'(u)v^t$ with one common row $v^t$.
Its coefficient dimension is three and its evaluation at an independent vector is
nonsingular. There is a primitive repeated linear determinant factor
$L(u)$ whose coefficients have degree $\ell_0$ in $V$, such that
\begin{equation}\label{p2:eq:strongseed}
 (P-2)r+\ell_0\le S-1,
 \qquad \frac{S-1}{P+1}\le r\le\frac{S-1}{P-2}.
\end{equation}
\end{lemma}
\begin{proof}
Set $P_*=AA^{[P]}$ and $C_u=A\sum_i u_iA_i^{[P]}$.
Then $P_*V=FV^{[PS]}$ and $C_uV=Fu$.
At a short vertex normalize the point to $e_3$ and its $P+1$
directions to $\PG(1,P)$. Normality makes the transverse block of
$A$ scalar and nonzero. For $\widetilde g=A^{[P]}V$ the exact
Hasse coefficients in a direction $w$ are
\[
 \mathcal H_w^{[PS]}\widetilde g=\sum_i w_i^{PS}A_i^{[P]}e_3,
 \qquad
 \mathcal H_w^{[PS+1]}\widetilde g=\sum_i w_i^{PS}A_i^{[P]}w.
\]
These are the pure coefficient of order $PS$ and the mixed
coefficient of order $PS+1$, evaluated at $e_3$. No other term can
occur, since $Pr+1<PS$.
For $w=(w_1,w_2,0)^t$ over $\F_P$, $w_i^{PS}=w_i$. With
$b(w)=(w_2,-w_1,0)^t$, normality makes the bivariate cubic
\[
 b(w)^t\sum_i w_iA_i^{[P]}(e_3)w
\]
zero on $P+1>3$ directions. It is homogeneous in $w_1,w_2$, so it
vanishes identically. For clarity, set $T_i=(A^{[P]})^{-1}A_i^{[P]}$ at this vertex.
Then $T_3=I$, and the pure Hasse coefficient gives a scalar
transverse block for $[T_1e_3\ T_2e_3\ T_3e_3]$.
Write $u=w+u_3e_3$. The cubic says $T(w)w$ belongs to
$\operatorname{Span}(e_3,w)$, and the scalar block gives
$T(u)e_3=\mu w+\eta e_3$. Thus
$\det(e_3,T(u)e_3,T(u)^2e_3)=0$.
Using $P_*[T_1V\ T_2V\ T_3V]=FI$ and $C_u=P_*T(u)$
transforms this into the coordinate-independent identity
\begin{equation}\label{p2:eq:firstkrylov}
 \det(V^{[PS]},u,C_u\adj(P_*)u)=0.
\end{equation}
Its coefficient degree is
$3(P+1)(S+r)<\tfrac92(P+1)S<(P^2-2)S<m_*$ for every $P\ge7$.
Restriction to each ordinary component, then the number of ordinary
components, makes \eqref{p2:eq:firstkrylov} a global identity.

Put $K_0=\overline{\F_q}(V)^P$ and $T_i=(A^{[P]})^{-1}A_i^{[P]}$.
The matrix $B'=[T_1V\ T_2V\ T_3V]$ is nonsingular by
\eqref{p2:full:seed}, and $P_*B'=FI$. Thus
$P_*^{-1}u=T(u)V/F$, $P_*^{-1}V^{[PS]}=V/F$, and
$C_u=P_*T(u)$, where $T(u)=\sum_i u_iT_i$.
Multiplying the columns of \eqref{p2:eq:firstkrylov} by
$P_*^{-1}$ therefore gives the determinant identity
$\det(V,T(u)V,T(u)^2V)=0$.
Its test-vector degree is three, below $P$, and its matrix entries
belong to $K_0$; $P$-power-basis independence therefore replaces $V$ by
an independent vector. Every member of the normalized matrix space
has minimal polynomial of degree at most two.

Lemma~\ref{p2:II:quadratic-space} now gives common-row form. Taking
$P$th roots of the coefficients gives the assertion for $A_u$.

Choose its primitive repeated factor $L(u)$, so
$L^2\mid\det A_u$, $L\mid\adj A_u$, and $2\ell_0\le3r$.
The argument modulo $H_2$ also gives $H_2\mid\adj A_u$;
$L$ and $H_2$ are coprime. Thus
\[
 \adj A_u=L(u)H_2K_u,\qquad
 K_{V^{[S]}}=E\bar N,\qquad
 E=\frac{B_0^2HC_N}{L(V^{[S]})H_2}.
\]
The diagonal of $L$ is nonzero because $\ell_0<S$.
Primitivity of $\bar N$ makes $E$ polynomial. Comparing degrees gives
\[
 \deg E=S-(P-2)r-\ell_0-1\ge0,
 \qquad D=B_0^{P-2}L(V^{[S]})E.
\]
It remains to prove the lower bound, which replaces $r\ge S/P$.
If $(P+1)r+1<S$, every entry of $A_iB_*$ has degree below $S$.
In $\sum_iV_i^SA_iB_*=FI$, separation of the $S$ blocks in the
off-diagonal and diagonal-difference equations makes each $A_iB_*$
scalar. Invertibility of $B_*$ would then make all $A_i$ proportional,
contrary to coefficient dimension three. Therefore
$(P+1)r+1\ge S$, proving~\eqref{p2:eq:strongseed}.
\end{proof}

\section{The repeated coefficient step and its degree bounds}\label{p2:secmap:4}\label{p2:sec:repeated-step}

\begin{definition}[Admissible data at scale $s$]\label{p2:II:admissible}
At depth $j\ge1$, let $s_i=S/P^{i-1}$, $1\le i\le j$, and $s=s_j$.
At this stage we have the matrix polynomial
\[
 C(Z_1,\ldots,Z_j;V)
 =\sum_{i_1,\ldots,i_j}(Z_1)_{i_1}\cdots(Z_j)_{i_j}
 C_{i_1\ldots i_j}(V),
\]
whose coefficients have degree $w<5s/(P+1)$ in $V$. Its coefficient
family is primitive; every parameter argument has coefficient dimension
three; and its last pencil has the common-row quadratic property
and nonsingular evaluation. On the full Frobenius diagonal,
\begin{equation}\label{p2:eq:actual}
 C(V^{[S]},\ldots,V^{[s]};V)=\Xi(V)A(V),
\end{equation}
with a polynomial $\Xi$ which is nonzero at short vertices.
Writing $\mathbf Z=(Z_1,\ldots,Z_{j-1})$, define
\begin{align}
 U(\mathbf Z;V)&=C(\mathbf Z,V^{[s]};V),\nonumber\\
 \calD(\mathbf Z;V)W&=C^{\langle P\rangle}(W,\mathbf Z;V)V.
 \label{p2:eq:partner}
\end{align}
If $C=\sum Z^I C_I(V)$, where $Z^I$ denotes a monomial in the
formal vector arguments, then
\[
 C^{\langle P\rangle}=\sum Z^I C_I(V)^{[P]}.
\]
Thus each entry of each whole coefficient polynomial matrix $C_I(V)$
is raised to its $P$th power, while the formal arguments remain unchanged.
\label{exp:p2:coefficient-frobenius}\ReadingNote{Frobenius on coefficient polynomials}{For example,
$(Z_1)_1(aX+Y)$ becomes $(Z_1)_1(a^PX^P+Y^P)$.
The degree in $V$ is multiplied by $P$, while the degree in each
formal argument stays the same.}
The remaining invariant is
\begin{equation}\label{p2:eq:partial}
 U\calD=\phi I,\qquad \phi_\Delta=\Xi^{P+1}F.
\end{equation}
The degrees in $V$ of $U,\calD$ are $s+w,Pw+1$; both are multilinear
in the old arguments. Their actual diagonals are invertible at short
vertices, and $U$ has coefficient dimension three in each old argument.
The pair $(C,\Xi)$ is called admissible when all these requirements
hold\TermNote{exp:p2:admissible}{this lists precisely the properties that must survive each repetition of the argument}. For $P=5$ the same definition has $w<5s/6$.
\end{definition}
For $j=1$ the crossed identity is \eqref{p2:full:seed}.

Table~\ref{p2:tab:retained-properties} records why each property is
retained. Section~\ref{p2:sec:completion} restores all of them after
one extraction; Section~\ref{p2:sec:five-entire-interval} uses the same
properties with its separate degree estimates at $P=5$.
\begin{table}[!htb]
\small
\renewcommand{\arraystretch}{1.12}
\caption{Properties retained at each coefficient step.}\label{p2:tab:retained-properties}
\centering
\begin{tabular}{@{}p{0.36\textwidth}p{0.58\textwidth}@{}}
\toprule
Retained property & Use in the next step\\\midrule
The coefficient family is primitive. & A common factor of the new
coefficients can be cancelled from the common multiplier.\\
Every argument has coefficient dimension three. & No argument
collapses to one scalar linear form; parameter-dependent common
factors are excluded.\\
The last pencil has common-row form and nonsingular evaluation. &
Its repeated determinant factor gives the adjugate factor and the
residual degree bounds.\\
$C_\Delta=\Xi A$, with $\Xi$ nonzero at all short vertices. &
The local normal restrictions and invertibility at short vertices
remain available after extraction.\\
$U\calD=\phi I$, with $\phi_\Delta=\Xi^{P+1}F$. & This reciprocal
identity supplies the next simultaneous coefficient extraction.\\
$U$ has coefficient dimension three in every old argument. &
Parameter factors cannot divide every entry of $\calD$; this also
supports the backward argument for $P=5$.\\\bottomrule
\end{tabular}
\end{table}

There is no parameter-dependent common factor of $C$: such a factor
would consume degree one in some argument and make its coefficient
dimension one. Thus primitivity in $V$ and the argument conditions
give primitivity in the full polynomial ring.

\begin{lemma}[Remaining factors and degree bounds]\label{p2:lem:gap}
There are a squarefree polynomial $H_*(V)$, a primitive last-argument
linear form $L(\mathbf Z,Y;V)$, and polynomial matrices $K,\calN$ with
\begin{align}
 \calD&=H_*\calN,& \adj C&=LH_*K,\label{p2:eq:adjrecursive}\\
 K(\mathbf Z,V^{[s]};V)&=E(\mathbf Z;V)\calN(\mathbf Z;V).
 \label{p2:eq:E}
\end{align}
If $\ell=\deg_V L$ and $e=\deg_V E$, then
\begin{equation}\label{p2:eq:interval}
 \boxed{(P+1)w+1\ge s,\qquad
 e=s-(P-2)w-\ell-1\ge0.}
\end{equation}
The degree of $E$ in each old parameter group is zero or one.
\end{lemma}
\begin{proof}
The common-row last pencil gives a primitive repeated factor $L$ with
$L^2\mid\det C$, $L\mid\adj C$, $2\ell\le3w$, and earlier-argument degrees
at most one. The greatest common factor of the entries of $\calD$ has no parameter-dependent factor:
such a factor consumes an argument, and \eqref{p2:eq:partial} would make $U$
projectively constant in that argument. Let its greatest common factor in $V$ be $H_*$.
It is squarefree. A repeated factor would divide all derivatives
with respect to $V$ of $\calD$, hence all entries of the $C_I^{[P]}$,
contradicting coefficient primitivity.

Modulo an irreducible factor $f$ of $H_*$, differentiation of
$\calD=H_*\calN$ gives, for each parameter coefficient,
$C_I^{[P]}\equiv n_I(\nabla H_*)^t$. Here
$\nabla H_*\not\equiv0\pmod f$: the factor is simple and the
coefficient field is perfect. The $2\times2$ minors of
$C_I^{[P]}$ are $P$th powers of those of $C_I$. Injectivity of
Frobenius modulo the prime $(f)$ therefore gives the asserted rank
bound for $C_I$; taking roots also gives their common row.
Consequently $H_*\mid\adj C$. Since $L$ is primitive
as a last-argument form, it is coprime to $H_*$. This proves the first
line. The last-argument specialization $L(\mathbf Z,V^{[s]};V)$
is nonzero: its square divides $\det U$ by $L^2\mid\det C$, whereas
$U\calD=\phi I$ with $\phi\ne0$ gives $\det U\ne0$.
This proof works also for $P=5$ and uses no bound on $\ell$.
Now specialize the last argument and use \eqref{p2:eq:partial}; the
rational scalar in \eqref{p2:eq:E} is polynomial since $\calN$ is
primitive. Comparing degrees in $V$ gives the formula for $e$;
comparing earlier-argument degrees gives
$\deg_{Z_i}E=1-\deg_{Z_i}L\in\{0,1\}$.

If $(P+1)w+1<s$, write $U=\sum_iV_i^s C_i(\mathbf Z;V)$.
Each $C_i\calD$ has degree below $s$. Separation in the off-diagonal
and diagonal-difference equations of \eqref{p2:eq:partial} makes
each $C_i\calD$ scalar. Invertibility of $\calD$ makes all $C_i$
proportional, contradicting coefficient dimension three.
\end{proof}

\section{Cancelling parameters and extracting the next scale}\label{p2:secmap:5}

\begin{lemma}[Parameter cancellation]\label{p2:lem:cancel}
Every parameter-dependent irreducible factor of $E$ divides $K$
before the last-argument substitution. After cancellation,
\begin{equation}\label{p2:eq:Ezero}
 K_0(\mathbf Z,V^{[s]};V)=E_0(V)\calN(\mathbf Z;V),
 \qquad e_0:=\deg E_0\le e,
\end{equation}
where $K_0$ is still linear in the last argument.
\end{lemma}
\begin{proof}
Since $E$ has degree at most one in each old parameter group,
such a factor occurs with multiplicity one and is linear in some
old coordinate $u$, say $f=au+b$. The coefficients $a,b$ may depend
on $V$ and all the other parameter coordinates. They remain in the
coefficient ring; separation below concerns only the final argument. The degree of $K$ in
$u$ is at most two. Substitute $u=-b/a$ and multiply by $a^2$.
The resulting matrix is linear in the last argument, with coefficient
degree in $V$ at most, writing $h_*=\deg H_*$,
\begin{align*}
 2w-\ell-h_*+2\deg_V f
 &\le2w-\ell-h_*+2e\\
 &=2s-2(P-3)w-3\ell-h_*-2<s.
\end{align*}
The strict inequality follows from \eqref{p2:eq:interval} and
\begin{equation}\label{p2:eq:threshold}
 2(P-3)\ge P+1\qquad(P\ge7).
\end{equation}
Its last-argument diagonal is zero by \eqref{p2:eq:E}; separation makes
the unspecialized matrix zero. Hence $f\mid K$ after allowing powers of $a$ as denominators. Clearing these
denominators proves the same divisibility for polynomials, since
$f$ and $a$ are coprime.
Cancel all such factors. This includes factors with zero derivatives
with respect to $V$; no perfect closure of the parameter field is used.
\end{proof}

\begin{lemma}[Uniform coefficient extraction]\label{p2:lem:extract}
Assume $s>1$, so that $P\mid s$, and put $\tau=s/P$.
There is one polynomial $G(V)$, with
$\gamma=\deg G$, such that simultaneously for all coefficient matrices
\begin{equation}\label{p2:eq:extract}
 GC_I=\sum_{a=1}^3V_a^\tau C'_{Ia},\qquad
 v:=\deg C'_{Ia}=w+\gamma-\tau.
\end{equation}
After a common-factor cancellation the new family is primitive,
$G$ is nonzero at every short vertex, and
\begin{align}
 0\le\gamma&\le\frac{2e_0}{P}
 \le\frac{6(s-1)}{P(P+1)}<\tau,\label{p2:eq:gamma}\\
 v&\le\frac{s-(P-4)w-2\ell-2}{P}
 \le\frac{5s-P-6-2(P+1)\ell}{P(P+1)}
 <\frac{5\tau}{P+1}.\label{p2:eq:v}
\end{align}
In particular,
\begin{equation}\label{p2:eq:allgaps}
 v<w,\quad w+v<s,\quad6v<s,\quad\tau+(P+1)v+1<s.
\end{equation}
\end{lemma}
\begin{proof}
Indeed, \eqref{p2:eq:Ezero} gives
\[
 E_0\calD=E_0H_*\calN
 =H_*K_0(\mathbf Z,V^{[s]};V)\in J_s
\]
entrywise, since $K_0$ is linear in its last argument.
Ordinary differentiation and elimination of the undifferentiated
term give $E_0^2\partial_a\calD\in J_s$.
Parameter-coefficient comparison yields
\[
 E_0^2f^P\in J_s=(X^s,Y^s,Z^s)
\]
for every scalar coefficient polynomial $f(V)$ of the old matrix polynomial.
We do not infer that a function with zero ordinary derivatives is
a $P$th power. Instead write the actual expansion by exponent residues modulo $P$
\[
 E_0^2=\sum_{0\le\alpha_1,\alpha_2,\alpha_3<P}
 V^\alpha G_\alpha^P.
\]
Lemma~\ref{p2:II:power-basis} gives, simultaneously,
\[
 G_\alpha f\in J_\tau,
 \qquad \deg G_\alpha\le\left\lfloor\frac{2e_0}{P}\right\rfloor.
\]
On the actual diagonal at a short vertex the adjugate and crossed
identities give
\[
 E_\Delta L_\Delta(\Xi B_0)^{P-2}=D.
\]
To see this, on the full diagonal use $U_\Delta=\Xi B_0M$
and $\phi_\Delta=(\Xi B_0)^{P+1}H^2a_0$ in the two adjugate
identities, and cancel the nonzero matrix $HN$.
Every factor on the left is therefore nonzero at a short vertex.
Since $E_0$ divides $E_\Delta$, it too is nonzero there. Hence the $G_\alpha$ do not all vanish
there. Choose a linear form nonzero at the finitely many short
vertices, homogenize these coefficients to their maximum degree,
and take a linear combination $G$ as in Lemma~\ref{p2:II:interpolation}(i). It remains nonzero at
every short vertex and satisfies $Gf\in J_\tau$ for every $f$.
These coefficient choices can be made over a finite algebraic extension
of $\F_q$. All later constants are allowed to lie in $\overline{\F_q}$;
the compatible point representatives are descended to $\F_q$ at the end.
The extraction and the bound $\gamma\le2e_0/P$ follow.

The remaining degree bounds follow from
$e_0\le s-(P-2)w-\ell-1$ and $(P+1)w+1\ge s$.
For the last bound in \eqref{p2:eq:allgaps}, explicitly,
\[
 \tau+(P+1)v+1
 \le\frac{6s-6-2(P+1)\ell}{P}<s.
\]
Any common factor of the new coefficients divides $G$, by old
primitivity. Cancel it from both; $\gamma,v$ only decrease.
All cancelled factors are nonzero at short vertices because $G$ is.
\end{proof}

\section{The Hasse coefficient calculation at each step}\label{p2:secmap:6}

\begin{lemma}[Preservation of the last coefficient pencil]\label{p2:lem:hasse}
The new matrix polynomial from \eqref{p2:eq:extract} has the quadratic common-row
property and nonsingular evaluation in its last argument. These assertions
hold with all older vector arguments independent.
\end{lemma}
\begin{proof}
Write its actual matrix as
\[
 \mathcal A=\Xi'A=\sum_iV_i^\tau\widehat A_i,\qquad c=\deg \Xi'.
\]
The older Frobenius blocks in $\widehat A_i$ have scales at least
$P\tau$, and the residual coefficient degree is $v<\tau$.
Total degree gives the exact formula
\begin{equation}\label{p2:eq:c}
 c=\frac{S-\tau}{P-1}+v-r.
\end{equation}
No estimate $c\le r$ is used.
Put
\[
 D_1=A^{[P]},\qquad C_i=\widehat A_i^{[P]},\qquad
 T_i=D_1^{-1}C_i,\qquad \widehat B=[C_1V\ C_2V\ C_3V].
\]
Normalize a short vertex to $e_3$ with its directions in $\PG(1,P)$.
By normality, $D_1$ preserves the radial line, has a nonzero scalar
transverse block, and preserves every incident normal plane
$\Span(e_3,w)$. Also $T_3=\Xi'^PI$ at the vertex.
For $\widehat g=\mathcal A^{[P]}V$ the exact Hasse coefficients are
\[
 \mathcal H_w^{[P\tau]}\widehat g=\widehat B w,\qquad
 \mathcal H_w^{[P\tau+1]}\widehat g=\sum_iw_i C_iw.
\]
The residual degree is $Pv+1<P\tau$; older blocks start at
$P^2\tau$. Thus there are no other terms at these orders.

Here $w$ is first chosen in the normalized $\F_P$-direction plane,
so $w_i^{P\tau}=w_i$. In the expansion of
$\sum_i V_i^{P\tau}\widehat A_i^{[P]}V$, order $P\tau$ comes
only from $V_i^{P\tau}$; order $P\tau+1$ additionally uses the
linear final vector. A residual term has degree at most $Pv+1$,
and changing an older block first contributes at order $P^2\tau$.
This proves the two formulas without an unrecorded mixed term.

Normality gives a scalar transverse block for $\widehat B$.
The mixed identity gives a bivariate cubic (homogeneous in the two direction coordinates)
\[
 \det(e_3,w,\sum_iw_iT_iw)=0
\]
on $P+1>3$ directions, hence on the whole transverse plane.
Using $T_3=\Xi'^PI$ extends it to arbitrary parameter vectors.
Write $u=w+u_3e_3$, with $w=(u_1,u_2,0)^t$.
The scalar transverse block gives $T(u)e_3=\mu w+\eta e_3$.
The cubic identity says $T(w)w\in\operatorname{Span}(e_3,w)$;
since $T_3$ is scalar, the same holds for $T(u)w$.
Consequently both $T(u)e_3$ and
$T(u)^2e_3=\mu T(u)w+\eta T(u)e_3$ lie in this plane. Therefore
\[
 \det(e_3,T(u)e_3,T(u)^2e_3)=0,
\]
including when $\mu=0$. The construction is
coordinate covariant: changing $V$ conjugates the $T_i$ and changes
their parameter index by the corresponding $P\tau$-Frobenius twist.

The cleared global polynomial is
\begin{equation}\label{p2:eq:krylov}
 \det(D_1V,C(u)V,C(u)\adj(D_1)C(u)V).
\end{equation}
Its raw degree in $V$ is $6P(S+r)+3P(c-\tau)+3$.
But
\[
 D_1V=B_0^PHg,\qquad
 \adj D_1=B_0^{2P}H^P C_N^P\bar N^{[P]},
\]
so it has the forced content $B_0^{3P}H^{P+1}C_N^P$.
This is nonzero at short vertices. Division gives
\begin{align*}
 d_{\rm Kr}
 &=4PS+\rho+P(P+3)r+3P(c-\tau)+P+2-Ph_2\\
 &=\left(4P+\frac{3P}{P-1}\right)S+\rho+P^2r
   +3Pv-\frac{3P^2}{P-1}\tau+P+2-Ph_2.
\end{align*}
The coefficient of $\tau$ after substituting
$v<5\tau/(P+1)$ is
\[
 3P\left(\frac5{P+1}-\frac P{P-1}\right)
 =-\frac{3P(P^2-4P+5)}{P^2-1}<0.
\]
Using $h_2\ge0$, \eqref{p2:full:rho}, and
$r\le(S-1)/(P-2)$ therefore yields
\begin{equation}\label{p2:full:krylov-budget}
 d_{\rm Kr}
 <\left(5P+6+\frac3P+\frac3{P-1}+\frac4{P-2}\right)S
 <(P^2-2)S<m_*.
\end{equation}
The omitted constant is $-4/(P-2)<0$.  To check the middle inequality
uniformly, the left coefficient is at most
$5P+6+3/7+1/2+4/5<5P+8$ for $P\ge7$, while
$5P+8<P^2-2$ already at $P=7$ and thereafter.  More exactly, at
$P=7$ the coefficient in \eqref{p2:full:krylov-budget} is
$2991/70<47$.

Each coefficient of the divided polynomial vanishes at all short
vertices; the count on each ordinary component makes it vanish on
that component, and the number of components makes it vanish globally.
Since $T_i\in\operatorname{Mat}_3(\overline{\F_q}(V)^P)$, $P$-power-basis independence
promotes the degree-three test vector to an independent vector.
Thus Lemma~\ref{p2:II:quadratic-space} applies to the normalized last space.

It remains to check nonsingular evaluation. If $\det\widehat B=0$,
its transverse scalar is zero at every short vertex, while its radial
eigenvalue is nonzero. All its columns are radial there. Put $a=A^{[P]}V$ and $b_i=C_iV$. The minors
$a_j(b_i)_k-a_k(b_i)_j$, for $1\le j<k\le3$, have degree at most
\[
 2P(S+r)+P(c-\tau)+2
 =\left(2P+\frac P{P-1}\right)S+Pr+Pv
   -\frac{P^2}{P-1}\tau+2.
\]
The same negative $\tau$ coefficient and strong bound on $r$ give
\[
 \deg(\text{evaluation minors})
 <\left(2P+\frac P{P-1}+\frac P{P-2}\right)S+2
 <(P^2-2)S+2=m_*\qquad(P\ge7).
\]
The same interpolation makes them zero
globally. The $P$-power-basis argument makes every $T_i$ scalar, so each
actual $\widehat A_i$ is proportional to the old substituted matrix.

Compare the formal new coefficient $C'_i$ with the \emph{old primitive
matrix polynomial} $C$. For any two matrix positions $a,b$, the
polynomials $(C'_i)_a C_b-(C'_i)_b C_a$ vanish after substitution
of the older powers of $V$. Their coefficient degree is $w+v<s$
and their degree in each older argument is at most two, so
Lemma~\ref{p2:II:separation} makes them zero before substitution as well.
Primitivity gives $C'_i=\eta_i C$ with polynomial $\eta_i$, whereas
\[
 \deg_V\eta_i=v-w=\gamma-\tau<0.
\]
This is impossible for a nonzero coefficient. Hence $\widehat B$
is nonsingular, and the common-image alternative in the quadratic
space classification is excluded.

For the formal promotion, first note that the common-row property
survives a change of the invertible member used for normalization. Indeed, for
$C(y)=\lambda(y)I+u(y)v^t$ and
$S_0=\lambda_0I+u_0v^t$ invertible, one has $\lambda_0\ne0$ and
\[
 S_0^{-1}C(y)=\frac{\lambda(y)}{\lambda_0}I+
 S_0^{-1}\left(u(y)-\frac{\lambda(y)}{\lambda_0}u_0\right)v^t.
\]
Thus use arbitrary last-argument vectors $y_0,y$ and test vector $W$ in
\[
 \det(C'_{y_0}W,C'_yW,C'_y\adj(C'_{y_0})C'_yW)=0.
\]
Every earlier-argument degree is six, below $P$, and the degree in $V$ is
$6v<s$. Separation promotes it to independent old arguments.

For this application of Lemma~\ref{p2:II:separation}, each earlier-argument
degree is six and $6<P$. Also $v<5\tau/(P+1)$ and $s=P\tau$,
so $6v<30s/(P(P+1))<s$ for every $P\ge7$.
 Evaluation in the last argument is nonsingular: otherwise specialize the old
arguments and $W=V^{[1/P]}$ in
$\overline{\F_q}[X^{1/P},Y^{1/P},Z^{1/P}]$, then raise to the
$P$th power, contradicting
$\det\widehat B\ne0$. This again excludes a common image and proves
the formal common-row assertion.
\end{proof}

\section{Completion of the descent}\label{p2:secmap:7}\label{p2:sec:completion}

\begin{proposition}[One complete descent step]\label{p2:II:descent-step}
Suppose $P\ge7$, $s>P$, and $(C,\Xi)$ is admissible at scale $s$.
Then the extraction in Lemma~\ref{p2:lem:extract} gives admissible data
$(C',\Xi')$ at scale $s/P$, with one more vector argument and
$\Xi'=G\Xi$ after the common-factor cancellation described there.
\end{proposition}
\begin{proof}
Lemmas~\ref{p2:lem:gap}--\ref{p2:lem:hasse} provide the extraction and the
new last pencil. Specializing the newest old argument to $V^{[s]}$
makes the new crossed product $G^{P+1}\phi I$. Its off-diagonal
entries and diagonal differences have coefficient degree
$\tau+(P+1)v+1<s$ and old-argument degrees at most two. Separation
therefore gives $U'\calD'=\phi'I$ before specialization.

The scalar identity on the full diagonal must also be restored.
By the definition of the shifted partner,
\[
 \calD'_\Delta V^{[PS]}=(C'_\Delta)^{[P]}V
       =(\Xi'A)^{[P]}V.
\]
Since $C'_\Delta=\Xi'A$ and $AA^{[P]}V=FV^{[PS]}$, this gives
\[
 \phi'_\Delta V^{[PS]}=\Xi'^{P+1}F V^{[PS]},
 \qquad \phi'_\Delta=\Xi'^{P+1}F.
\]
All factors on the right are nonzero at short vertices. At the next
stage the identity $E_\Delta L_\Delta(\Xi'B_0)^{P-2}=D$ therefore
again guarantees the nonvanishing required in the extraction.

Specializing the new last argument recovers $GC$, so all old
coefficient dimensions remain three; the new partial matrix is
exactly $GC$, which gives its old-argument dimension condition too.
Lemma~\ref{p2:lem:hasse} gives nonsingular evaluation and common-row
form in the new last argument. The cancelled coefficients are
primitive, and Lemma~\ref{p2:lem:extract} supplies $v<5\tau/(P+1)$
and the required nonvanishing of $G$. Thus every item in the
definition is restored.
\end{proof}

\begin{proposition}[Full Frobenius endpoint]\label{p2:prop:endpoint}
There are a nonzero polynomial $\Theta(V)$ and a constant-coefficient $3\times3$ matrix polynomial $\mathsf C$,
linear in each of $k_{\rm slot}=n-2$ vector arguments with
\begin{equation}\label{p2:eq:endpoint}
 \Theta M=\mathsf C(V,V^{[P]},\ldots,V^{[P^{n-3}]}),\qquad
 \deg(\Theta M)=d_{\rm nat}:=\frac{PS-1}{P-1}.
\end{equation}
\end{proposition}
\begin{proof}
Lemmas~\ref{p2:lem:seed} and \ref{p2:lem:firstrow} give admissible data at
scale $S$, with $r<5S/(P+1)$. Apply
Proposition~\ref{p2:II:descent-step} at each scale $s>P$.
Eventually $s=P$, where \eqref{p2:eq:interval} gives
\[
 \frac{P-1}{P+1}\le w\le\frac{P-1}{P-2}<2.
\]
Hence the integer $w$ equals one. The residual coefficients in $V$
are linear in $V$, so the complete list of arguments is $1,P,\ldots,S$,
as asserted. If $n=3$, $M$ is already linear by \eqref{p2:eq:degrees};
if $n=4$, the first strong bound already forces $r=1$.
\end{proof}

\section{Reconstruction and identification for exponents at least seven}\label{p2:secmap:8}\label{p2:sec:reconstruction}

The common-function result in
\cite[Theorem~\Iref{I:endpoint-theorem}]{p2:PUB1} gives a common rational
function on every ordinary component. Lemma~\Iref{I:normalizer}
then gives compatible point representatives.
To obtain representatives over $\F_q$ and bound
their span dimension, use \eqref{p2:eq:mzero} and the coarse upper bound in
\eqref{p2:full:literature}.  They imply
\begin{equation}\label{p2:full:graph}
 2Pm_*-|B|>
 \left(1-\frac1P-\frac5{P^2}-\frac3{P^3}\right)q
 \ge\frac{256}{343}q>\frac q2.
\end{equation}
Join two ordinary points when they lie on a short secant. Each
vertex of this graph has at least $Pm_*$ neighbours; the inequality
$2Pm_*>|B|$ forces the graph to be connected.  One common
rescaling puts all ordinary representatives in $\F_q^3$.
Moreover
\[
 q<N_0\le |B|
 <\left(1+\frac1P+\frac1{P^2}+\frac3{P^3}\right)q
 \le\frac{402}{343}q<\frac65q.
\]
Let $r_U$ be the $\F_P$-dimension of their span and put
$\delta=(P-1)m_*+1<q$.
The dimension bound, Proposition~\Iref{I:code-bound} of~\cite{p2:PUB1}, gives
\[
 P^{r_U}\le
 \frac{PN_0((P-1)m_*+1)}{2+2Pm_*-N_0}
 <\frac{12}{5}Pq<P^2q=P^{n+2}.
\]
Thus $r_U\le n+1$.  Since $N_0>q$, whereas an $\F_P$-space of
dimension at most $n$ has at most $(q-1)/(P-1)$ projective directions,
we have $r_U=n+1$.

Let $U$ be this span and $B'=L(U)$.  The usual dimension argument
makes $B'$ blocking, all its line intersections are $1\pmod P$,
and
\[
 |B'|\le b_{\max}:=\frac{Pq-1}{P-1}.
\]
Every ordinary point of $B$ belongs to $B'$.  It is important not
to assume $b_{\max}-|B|\ge0$: that inequality is precisely one of the
conclusions being sought.  Directly from the intersection count,
\[
 |B\mathbin\triangle B'|
 =|B|+|B'|-2|B\cap B'|
 \le b_{\max}-|B|+2\sigma.
\]
Set $\rho_C=(PS-1)/(P-1)$.  Since $b_{\max}-|B|=P(\rho_C-\rho)$,
$\sigma\le d-1$, and $\rho\ge0$, it follows that
\begin{align*}
 |B\mathbin\triangle B'|
 &\le P\rho_C+2PS-(P-2)\rho-2\\
 &\le P\rho_C+2PS-2<4PS<\frac{2(P-1)q}{P}+2.
\end{align*}
Lemma~\Iref{I:signed} of~\cite{p2:PUB1} excludes a
nonempty symmetric difference between two sets whose line sums
agree modulo $P$.  Therefore $B=B'=L(U)$.

This proves Theorem~\ref{p2:II:main} when $P\ge7$.
If $B=L(U')$ for an $\F_{p^f}$-space $U'$, let $\widehat\ell$
be the vector plane of a short secant. The intersection
$U'\cap\widehat\ell$ represents $P+1$ distinct projective points,
so it contains two vectors independent over $\F_q$, hence also over
$\F_{p^f}$. Their $\F_{p^f}$-span gives an $\F_{p^f}$-subline
inside the $\F_P$-subline $B\cap\ell$. Normalize three common points to
$0,1,\infty$. The sublines become the standard sublines, so
$\F_{p^f}\subseteq\F_P$. Thus $\F_P$ is the maximum geometric field of linearity.
\section{The entire exponent-one interval in characteristic five}\label{p2:secmap:9}
\label{p2:sec:five-entire-interval}

The remaining prime power below seven is $P=5$.
The following argument replaces the extraction, recognition and
parameter-cancellation steps; the Frobenius foundation is cited from
the first paper throughout.
\begin{theorem}[Characteristic five]
\label{p2:thm:five-entire-interval}
Let $q=5^n$, $n\ge3$, and let $B$ be a small minimal blocking set in
$\mathrm{PG}(2,q)$ of exponent one. Then there is an
$\mathbb F_5$-subspace $U\le\mathbb F_q^3$ of dimension $n+1$ such that
$B=L(U)$. In particular, the conclusion covers the whole classical
exponent-one interval, not only sizes at most $(5q-1)/4$.
\end{theorem}

\subsection{The interval and the uniform inputs}

Put $P=5$ and $S=5^{n-3}$, so $q/P^2=5S$ and $q/P=25S$. Write
\[
 q=125S,\qquad |B|=q+5d+1,\qquad d=5S+\rho.
\]
The size theorem of Paper I gives
\begin{equation}\label{p2:five:lower}
 d\ge5S,\qquad\rho\ge0.
\end{equation}
After substituting $|B|=q+5d+1$, the quadratic defining the upper endpoint,
divided by $25$, is
\[
 f_S(\rho)=\rho^2-(90S+1)\rho+150S^2-30S.
\]
The allowed $\rho$ lies below its smaller root. Direct evaluation gives
\[
 f_S\left(\frac{7S}{4}-\frac15\right)
 =-\frac{71S^2}{16}-\frac{289S}{20}+\frac6{25}<0.
\]
Consequently
\begin{equation}\label{p2:five:upper}
 \rho<\frac{7S}{4}-\frac15,
 \qquad d<\frac{27S}{4}-\frac15,
 \qquad |B|<\frac{127}{100}q.
\end{equation}
These bounds suffice for the estimates below.

Theorem~\Iref{I:uniform-main} of~\cite{p2:PUB1} applies
and supplies the Frobenius representation. The local results in
Section~\Iref{I:local} of that paper apply in characteristic five.
In the present notation, its identities are
\begin{align}
 Hg&=M^{[5]}V,& Mg&=Ha_0V^{[q/P^2]},& \det M&=H^2a_0D,
 \label{p2:five:uniform}\\
 N&=\operatorname{adj}(M)/H,&Dg&=NV^{[q/P^2]},
 &MN&=Ha_0D I.\nonumber
\end{align}
Here $M,g$ are primitive, $M$ is nonsingular, $H$ is squarefree, and
\begin{align}
 \deg M&=S+t,&\deg H&=5t-\rho+1,\nonumber\\
 \deg a_0&=S+2\rho-4t-1,&\deg D&=2S-3t-1,
 &\deg N&=2S+\rho-3t-1.\label{p2:five:degrees}
\end{align}
Nonnegativity of these degrees gives, in particular,
\begin{equation}\label{p2:five:t-range}
 t\ge\frac{\rho-1}{5},\qquad 0\le t\le\frac{2S-1}{3}.
\end{equation}

The local leading-block, short-vertex regularity, special-component, and
weight-gap arguments of Lemmas~\Iref{ppg:leading}--\Iref{ppg:gap} apply
unchanged: their comparison excluding weights $2,3,4$ is
$5(9-2r)>0$ for $2\le r<5$. All original scalar factors in
\eqref{p2:five:uniform} are nonzero at short vertices. Every short has six
ordinary incident components and an $\mathbb F_5$-subline of directions.
Every ordinary component has at least
\begin{equation}\label{p2:five:mzero}
 m_0=25S+1-\left\lfloor\frac{d-1}{4}\right\rfloor
 >\frac{373}{16}S
\end{equation}
short vertices. If $\sigma$ is the number of special components and $N_0$
the number of ordinary components, then
\begin{equation}\label{p2:five:ordinary-count}
 \sigma\le d-1-2(S+t)=3S+\rho-2t-1,
 \qquad N_0\ge q+4d+2>q.
\end{equation}
We repeatedly use the following consequence: a polynomial of degree
less than $m_0$ vanishing at all short vertices is zero. Its restriction
to every ordinary component has too many zeros; all $N_0>q>m_0$
ordinary linear factors would therefore divide it.

\subsection{The first extraction: removing the error term}

Write $N=C_N\bar N$, $D=C_N\bar D$, with $\bar N$ primitive, and set
\[
 z=\deg C_N,\qquad \delta_0=\deg\bar D=2S-3t-1-z.
\]
Differentiate $\bar D M^{[5]}V=H\bar N V^{[q/P^2]}$ and eliminate the
undifferentiated term. This gives
\[
 \bar D^2M^{[5]}e_j\in(X^{q/P^2},Y^{q/P^2},Z^{q/P^2}).
\]
Write the expansion by exponent residues modulo five
$\bar D^2=\sum_{0\le\alpha_i<5}V^\alpha G_\alpha^5$.
Uniqueness of this residue expansion gives
$G_\alpha M_{ij}\in(X^S,Y^S,Z^S)$ and
$\deg G_\alpha\le\lfloor2\delta_0/5\rfloor$.
Since $\bar D$ is nonzero at shorts, these coefficients do not vanish
simultaneously there. Homogenize them to their largest degree using a
linear form avoiding the finite short set, and choose a linear
combination avoiding the forbidden hyperplanes over a finite extension. It gives a homogeneous
$G$, nonzero at every short, with
\[
 GM=\sum_iV_i^S\widetilde A_i,\qquad
 g_0:=\deg G\le\frac{2\delta_0}{5},\qquad
 r_0:=\deg\widetilde A_i=t+g_0.
\]
Since $GM$ is nonzero, $r_0\ge0$, while \eqref{p2:five:t-range} gives
\begin{equation}\label{p2:five:seed-rough}
 r_0\le\frac{4S-t-2-2z}{5}
 \le\frac{4S-2}{5}<\frac45S<S.
\end{equation}
Set
\[
 \widetilde B=[\widetilde A_1^{[5]}V\ \widetilde A_2^{[5]}V\
 \widetilde A_3^{[5]}V],\qquad U_{q/P^2}=V^{[q/P^2]}.
\]
The multiplied reciprocal identity says
$(\bar D\widetilde B-G^5H\bar N)U_{q/P^2}=0$.
Lemma~\ref{p2:II:koszul}, applied row by row, therefore gives
a polynomial matrix $E$ such that
\begin{equation}\label{p2:five:koszul-error}
 \bar D\widetilde B=G^5H\bar N+E[U_{q/P^2}]_\times,
 \qquad \deg E=\delta_0+5r_0+1-q/P^2
 \le S-4t-2-3z.
\end{equation}
Put $K=ME$. Its degree satisfies
\begin{equation}\label{p2:five:k-degree}
 \kappa:=\deg K=5r_0-2S-2t-z
 \le2S-3t-2-3z\le2S-2<q/P^2.
\end{equation}
A negative expected degree means that the corresponding matrix is zero.

Here is the local implication needed to remove this error. On an ordinary
component choose rational coordinates making it $Z=0$. The normal
restriction for $M$ says $M_{31}=M_{32}=0$ there. In $GM=\sum V_i^S
\widetilde A_i$, the coefficient degrees are $r_0<S$, so separation of
$X^S,Y^S$ on $Z=0$ gives $(\widetilde A_1)_{3j}=
(\widetilde A_2)_{3j}=0$ there for $j=1,2$.
It follows that $\widetilde B_{31}=\widetilde B_{32}=0$ there.
Thus $\widetilde B$ has the opposite-twist ordinary normality.
At a short vertex normalize the point to $e_3$ and its six directions
to $\mathrm{PG}(1,5)$. For every incident normal
$a=(a_1,a_2,0)$ over $\mathbb F_5$, both $a^tM$ and
$a^t\widetilde B$ are proportional to $a^t$ at that vertex.
Consequently both transverse $2\times2$ blocks are scalar and both
upper-right columns vanish. Thus
$M\widetilde B=cI+e_3v^t$ there, without requiring
$\widetilde B$ nonsingular.

Multiplying \eqref{p2:five:koszul-error} by $M$ gives
\[
 \bar D M\widetilde B=G^5H^2a_0\bar D I+K[U_{q/P^2}]_\times.
\]
At the normalized vertex $U_{q/P^2}=e_3$, write
\[
 K=\begin{pmatrix}a&b&c\\d&e&f\\h&i&j\end{pmatrix},\qquad
 K[e_3]_\times=\begin{pmatrix}b&-a&0\\e&-d&0\\i&-h&0\end{pmatrix}.
\]
Since the transverse block of this matrix is scalar, $a=e=0$ and
$d=-b$. Thus the restriction of $K+K^t$ to the transverse plane is
zero. Equivalently,
\begin{equation}\label{p2:five:symmetric-local}
 [U_{q/P^2}]_\times^t(K+K^t)[U_{q/P^2}]_\times=0
\end{equation}
there. This assertion is coordinate invariant: for $V_{\rm old}=CV$, put
$R_C=C^{[q/P^2]}$; the multiplied reciprocal matrix is conjugated by $R_C$,
whereas $K$ transforms by a nonzero constant times
$R_C^{-1}KR_C^{-t}$.
The degree of \eqref{p2:five:symmetric-local} is
\[
 2q/P^2+\kappa\le12S-2<m_0.
\]
Interpolation makes it a polynomial identity. Since $\kappa<q/P^2$, Lemma~\ref{p2:II:separation} makes its quadratic
$q/P^2$-power blocks independent. Replace $U_{q/P^2}$ by an independent vector
$u$. Substitution $u=e_1,e_2,e_3$ shows respectively that each of the
three complementary $2\times2$ principal restrictions of $K+K^t$
is zero. These restrictions contain all entries, so $K+K^t=0$.
Write $K=[w]_\times$. The cross-product identity yields
\[
 \bar D M\widetilde B=G^5H^2a_0\bar D I
       +U_{q/P^2}w^t-(w^tU_{q/P^2})I.
\]
For each $j$, the two off-diagonal entries give
$\bar D\mid (U_{q/P^2})_i w_j$ for both $i\ne j$.
Those two coordinate powers are coprime, so $\bar D\mid w_j$.
But
\[
 \deg(w/\bar D)=5r_0+t+1-4S
 \le6t+2\delta_0+1-4S=-1-2z<0.
\]
Hence $w=K=E=0$.

Primitivity of $\bar N$ now gives $\bar D\mid G^5H$.
Cancel the greatest common factor of the entries of all the $\widetilde A_i$ and the same
factor of $G$; it divides $G$ because $M$ is primitive.
The resulting identities are
\begin{align}
 A:=B_0M&=\sum_iV_i^S A_i,& A_i^{[5]}V&=H_2\bar N_i,
 &\bar D&=B_0^5H/H_2,\label{p2:five:seed}\\
 AB_*&=F I,& B_*&=[A_i^{[5]}V]=H_2\bar N,
 &F&=B_0^6H^2a_0.\nonumber
\end{align}
The coefficient family is primitive, and $H_2$ is squarefree:
a repeated prime factor of $H_2$ would, after differentiation, divide
every entry of every $A_i^{[5]}$, contradicting primitivity.
All these actual scalar factors are nonzero at short vertices. In particular,
each factor of $B_0$ occurs with multiplicity at least four in
$\bar D=B_0^5H/H_2$.
Writing $r=\deg A_i$, we retain $r<S$ and have
\[
 h_2:=\deg H_2=5r+3t-2S-\rho+z+2.
\]

\subsection{A cubic determinant and its Hasse coefficient test}

\begin{lemma}[Output cubic criterion]\label{p2:five:cubic-criterion}
Let $C(Y)=\sum_{i=1}^3Y_iC_i$ be a $3\times3$ matrix pencil with nonzero determinant over an infinite field of odd characteristic, with three
independent coefficient matrices. Define $S_C(a)$ by
$a^tC(Y)W=Y^tS_C(a)W$. If $\det S_C(a)=0$ identically, then after
multiplication by the inverse of one invertible member
\[
 C(Y)=\lambda(Y)I+u(Y)v^t
\]
with a fixed row $v^t$. The evaluation map $Y\mapsto C(Y)W$ is
nonsingular for an independent column $W$.
\end{lemma}
\begin{proof}
There is no fixed nonzero column $w$ with $S_C(a)w=0$ for all $a$:
that would give $C(Y)w=0$ for all $Y$, contradicting invertibility.
There is no fixed nonzero row $b^t$ with $b^tS_C(a)=0$ either:
that would give $\sum b_iC_i=0$. A linear space of matrices of rank
at most one has a common row or a common image, by the rank-one
argument in Lemma~\ref{p2:II:quadratic-space}, and would have one of
these fixed null vectors. Thus $S_C$ has rank two over the rational
function field in $a$.

Its cofactor matrix has rank one and homogeneous entries of degree
two. Factoring its entries gives
$\adj S_C=h\,u_0(a)v_0(a)^t$, with primitive polynomial columns
$u_0,v_0$. For example, take a nonzero column, remove its greatest
common factor, and use the vanishing $2\times2$ minors to factor
all other columns. Neither $u_0$ nor $v_0$ is constant, since they
satisfy $S_Cu_0=0$ and $v_0^tS_C=0$ and fixed null vectors were
excluded. Degree two therefore makes both columns linear and $h$
constant. Each linear map $a\mapsto u_0(a),v_0(a)$ has rank at least two.

If the rank of $u_0$ is three, change coordinates to make $u_0=a$.
The relation $S_C(a)a=0$ then gives
$S_C(a)=R[a]_\times$ with $R$ constant: each linear row annihilating
$a$ is a constant combination of $(a_2,-a_1,0)$,
$(a_3,0,-a_1)$, and $(0,a_3,-a_2)$.
The absence of a fixed left null vector makes $R$ invertible.
For fixed $Y$, the expression
$Y^tR[a]_\times W=(R^tY)^t(a\times W)$ has an alternating
coefficient matrix in $a,W$, hence a singular $3\times3$ matrix.
Thus every $C(Y)$ would be singular, a contradiction.

The remaining rank is two. Normalize $u_0=(a_1,a_2,0)^t$.
The first two columns of $S_C$ must then be $a_2c,-a_1c$ for one
nonzero constant column $c$. Whenever $c^tY=0$ and $W_3=0$,
we have $a^tC(Y)W=0$ for every $a$. Choose bases in these two
planes. The four corresponding entries of this bilinear expression
are zero, so
$C(Y)=\lambda(Y)M+u(Y)v^t$, with $\lambda=c^tY$ and a fixed row $v^t$.
At an invertible value $Y_0$, $\lambda(Y_0)\ne0$.
Subtract $(\lambda(Y)/\lambda(Y_0))u(Y_0)$ from $u(Y)$ and absorb
the removed term into $M$. Normalize the invertible matrix $M$ to $I$.
We may then take $u(Y_0)=0$ and $\lambda(Y_0)=1$.
Independence of the three coefficient matrices forces
$u|_{\ker\lambda}$ to have rank two. In a basis with first vector
$Y_0$, the evaluation determinant is
\[
 (v^tW)^2\det(W,u(Y_1),u(Y_2))\not\equiv0,
\]
where $Y_1,Y_2$ span $\ker\lambda$. This proves the assertion.
\end{proof}

We record the actual detector with its twist, since the new scale is
not the scale of its moving normal hyperplane. Let
\[
 \widehat A=\sum_i V_i^\tau\widehat A_i
\]
be an actual extraction. In each $\widehat A_i$, older Frobenius blocks
have weights at least $5\tau$, and residual coefficients in $V$ have
degree $v<\tau$. Suppose $\widehat A$ is a scalar multiple of $M$ with
a scalar nonzero at shorts. Put
\[
 \mathcal F(a;V)=\det\begin{pmatrix}
 a^t\widehat A_1\\a^t\widehat A_2\\a^t\widehat A_3
 \end{pmatrix}.
\]
At a normalized short point $e_3$, write
$B_i=\widehat A_i(e_3)^{[5]}$.
Normality at order one makes the transverse block of $B_3$
equal to $kI$, with $k\ne0$, and its upper-right column zero.
The pure Hasse coefficient of order $5\tau$ in
$\widehat A^{[5]}V$ gives
\[
 \operatorname{top}_2(B_1e_3,B_2e_3)=\mu I.
\]
The mixed coefficient of order $5\tau+1$ gives
\[
 (B_1)_{\rm top}=\begin{pmatrix}a&b\\0&d\end{pmatrix},
 \qquad
 (B_2)_{\rm top}=\begin{pmatrix}e&0\\a-d&b+e\end{pmatrix}.
\]
Indeed, the relevant bivariate quadratic and cubic (homogeneous in the two direction coordinates) vanish in the six
$\mathbb F_5$ directions and hence identically. There are no interfering
Hasse terms: the residual degrees are below $\tau$, while older blocks
start at $5\tau$ before taking fifth powers.
The three rows
\[
 (ua,ub+vd,u\mu),\quad
 (ue+v(a-d),v(b+e),v\mu),\quad
 (uk,vk,0)
\]
have zero determinant by expansion. Taking fifth roots proves that
$\mathcal F(a;e_3)$ vanishes for $a_3=0$.
Under $V_{\rm old}=CV$, the covariance is
\[
 M_{\rm new}=(C^{-1})^{[1/25]}M_{\rm old}(CV)C^{[1/5]},
\]
and the new coefficient index changes by $C^{[\tau]}$.
Thus, at the short vertex $\mathsf V$ in the original rational coordinates,
the vanishing hyperplane is
\begin{equation}\label{p2:five:moving-hyperplane}
 a^t\mathsf V^{[1/25]}=a^t\mathsf V^{[q/25]}=0.
\end{equation}

If $h=\deg\widehat A_i$, then substituting $a=W\times V^{[q/P^2]}$
in $\mathcal F$ has coefficient degree $3h+3q/P^2$.
If $3h+3q/P^2<m_0$, interpolation makes
$\mathcal F(W\times V^{[q/P^2]};V)=0$ identically.
Over the rational function field in $V$, these cross products fill
the plane $a^tV^{[q/P^2]}=0$, so its linear equation divides
$\mathcal F$. Its three coefficients have greatest common factor one;
clearing denominators therefore proves polynomial divisibility.
If also $3h<q/P^2$, the factor has greater $V$-degree than
$\mathcal F$, forcing $\mathcal F=0$.
For the first seed $h=r<S$, both inequalities hold:
$3r+15S<18S<m_0$ and $3r<5S$.
The exact $B_*=H_2\bar N$ is nonsingular, so the first pencil has
coefficient dimension three, and the cubic criterion applies.

Let $L(Y;V)$ be its primitive repeated linear determinant factor,
$\ell=\deg_V L$. Then $L\mid\operatorname{adj}C$.
Differentiating \eqref{p2:five:seed} modulo a prime of $H_2$ gives rank one
with a common row, so $H_2\mid\operatorname{adj}C$ as well, and
$(L,H_2)=1$. Write
\[
 \operatorname{adj}C=LH_2K,
 \qquad K(V^{[S]};V)=E\bar N.
\]
The scalar $E$ is polynomial by primitivity of $\bar N$, and degree
comparison gives
\begin{equation}\label{p2:five:strong-seed}
 \deg E=S-3r-\ell-1\ge0,
 \qquad 3r+\ell\le S-1.
\end{equation}
If $6r+1<S$, separation in $AB_*=FI$ makes every $A_iB_*$ scalar,
contrary to coefficient dimension three. Hence
\begin{equation}\label{p2:five:seed-interval}
 (S-1)/6\le r\le(S-1)/3.
\end{equation}

\subsection{The repeated step: cancellation before substitution}

We retain the properties in Table~\ref{p2:tab:retained-properties},
now with $P=5$, and establish the degree bounds needed in this case.
At scales $S,S/5,\ldots,s$, retain a matrix polynomial
$C(Z_1,\ldots,Z_j;V)$, linear in each $Z_i$ and of coefficient
degree $w$ in $V$. Its coefficient family is primitive, and the
coefficient dimension in every argument is three. Its last pencil
is common-row with nonsingular evaluation at an independent vector. Put
\begin{align*}
 U&=C(Z_1,\ldots,Z_{j-1},V^{[s]};V),\\
 \mathcal D(Z_1,\ldots,Z_{j-1};V)W
 &=C^{\langle5\rangle}(W,Z_1,\ldots,Z_{j-1};V)V.
\end{align*}
Here, as in Section~\ref{p2:sec:repeated-step}, if
$C=\sum Z^I C_I(V)$, then
$C^{\langle5\rangle}=\sum Z^I C_I(V)^{[5]}$.
Each entry of each whole coefficient polynomial matrix is raised
to its fifth power; the formal arguments remain unchanged.
The invariant includes
\begin{equation}\label{p2:five:invariant}
 U\mathcal D=\phi I,\qquad \phi_\Delta=\Xi^6F,\qquad C_\Delta=\Xi A.
\end{equation}
Also $U$ has coefficient dimension three in every old argument.
\begin{equation*}
 \Xi\ne0\text{ at every short vertex.}
\end{equation*}
On the full actual diagonal, $C$ is a scalar multiple of the first
matrix $A$, with scalar nonzero at every short vertex. The scalar on the right of the matrix product is nonzero there as well.
Primitivity in $V$ and coefficient dimension three imply primitivity
in the full polynomial ring: a parameter factor would consume the
degree of an argument and reduce that argument's coefficient dimension to one.

The residual adjugate argument, which does not require $P\ge7$, gives
\begin{align}
 \mathcal D&=H_*\mathcal N,&
 \operatorname{adj}C&=LH_*K,&
 K(\ldots,V^{[s]};V)&=E\mathcal N,\label{p2:five:residual}\\
 6w+1&\ge s,&e:=\deg_VE&=s-3w-\ell-1\ge0,\nonumber\\
 &&\deg_{Z_i}E&=1-\deg_{Z_i}L\in\{0,1\}.\nonumber
\end{align}
Here $\mathcal N$ is primitive, $H_*$ is a squarefree polynomial in
$V$, and $L$ is the primitive repeated factor for the last pencil, with
$\ell=\deg_V L$. To recall the common-factor arguments, a parameter factor
dividing every entry of $\mathcal D$ would make $U$ projectively constant in
that argument, contrary to \eqref{p2:five:invariant}. A repeated factor depending only on $V$
would divide all derivatives with respect to $V$ of $\mathcal D$ and thus all fifth
powers of the primitive coefficient family. Finally, reduction modulo
a factor of $H_*$ gives common-row rank one, yielding
$H_*\mid\operatorname{adj}C$. Specialization and primitivity give the
polynomial $E$. Here the last-argument specialization
$L(\mathbf Z,V^{[s]};V)$ is nonzero because its square divides
$\det U\ne0$, exactly as in Lemma~\ref{p2:lem:gap};
there is no use of $3w<s$ to establish $E$. Only after $E$ has been
obtained does its nonnegative degree imply $3w<s$.
The full diagonal identity gives $D=E_\Delta L_\Delta(\Xi B_0)^3$;
thus every residual factor needed at a short vertex is nonzero.

We now prove parameter cancellation for $P=5$ without using
$2(P-3)\ge P+1$, which is false here. Start with the last old argument.
The evaluation of its $\mathcal D$-pencil is the shifted fifth-power
evaluation of the last pencil of $C$. It is nonsingular after its test
vector is specialized to $V$: the determinant has test-vector degree
three, strictly below five, with coefficients in the field of fifth powers of rational functions in $V$. Fifth-power basis independence prevents this specialization from
killing a nonzero determinant.

Over the field of the remaining variables, choose $Y_0$ with
$U(Y_0)\mathcal D(Y_0)=\phi(Y_0)I\ne0$, and put
\[
 \widetilde U(Y)=U(Y_0)^{-1}U(Y),\qquad
 \widetilde{\mathcal D}(Y)=\mathcal D(Y)U(Y_0)/\phi(Y_0).
\]
Both have value $I$ at $Y_0$ and their product is scalar. The
coefficient of $t$ in the product at $Y_0+tY$ is
$\widetilde U(Y)+\widetilde{\mathcal D}(Y)=\alpha(Y)I$;
its coefficient of $t^2$ gives
$\widetilde U(Y)^2-\alpha(Y)\widetilde U(Y)+\beta(Y)I=0$.
The proof of Lemma~\ref{p2:II:quadratic-space} now gives a scalar linear
form plus a rank-at-most-one matrix space. It has a common row or a
common image. In the common-image alternative
$\widetilde U=\lambda I+u v(Y)^t$, with $u$ fixed, the reciprocal
has the form $(\lambda+v(Y)^tu)I-u v(Y)^t$. Its value at any test
vector $W$ belongs to $\Span(W,u)$, contradicting nonsingular
evaluation. This is the specific reason for choosing the common row.

Its output cubic vanishes. Before the last-argument substitution this cubic has coefficient degree $3w<s$ and its last-argument
degree is three. Disjoint $s$-power blocks promote the identity to the
formal matrix polynomial $C$. The cubic criterion gives the earlier-argument common-row
property and nonsingular evaluation. For clarity, the backwards induction starts with old argument $j-1$.
The evaluation property in argument $j$ of $C$ is the evaluation property
in argument $j-1$ of $\mathcal D$, by the displayed shift in its definition.
The reciprocal-pencil calculation therefore proves the property for
argument $j-1$ of $U$, and its output cubic lifts to $C$ by
Lemma~\ref{p2:II:separation}. With that argument proved, the same shift supplies
the hypothesis for argument $j-2$. Continue until argument $1$. All other
parameters remain in the coefficient field at each step; the bound
$3w<s$ is the same at every lift. Thus every argument of $C$ is common-row.

Let $f$ be an irreducible parameter factor of $E$ involving $Z_i$.
It occurs once, and $\deg_{Z_i}L=0$.
Let $\Pi_i$ be the primitive repeated linear determinant factor of the
formal $Z_i$-pencil. It is distinct from $L$, and
\[
 \Pi_i^2L^2\mid\det C.
\]
Since $\det C$ has degree three in the final argument and $L$ has degree
one there, $\Pi_i$ is independent of that argument. It therefore survives
specialization unchanged. Also $f\mid\operatorname{adj}U$ by
\eqref{p2:five:residual}. For a three-dimensional common-row pencil the greatest common divisor of the
adjugate entries over the field of the other variables is exactly its
unique repeated linear factor.

Indeed, for $C(Y)=\lambda(Y)I+u(Y)v^t$, direct multiplication gives
\[
 \adj C=\lambda\bigl((\lambda+v^tu)I-u v^t\bigr).
\]
After taking $v=e_3$, the entries in parentheses include $\lambda$
and the coordinates of $u$. A common nonconstant factor would divide
all these linear forms, making the coefficient space one-dimensional.
Thus the bracket has no nonconstant common factor and the common factor
of the adjugate entries is exactly
$\lambda$, up to nonzero constants. This argument is over the fraction field of
the other parameters; clearing reduced denominators transfers the divisibility to the
polynomial ring, because the chosen factors have coprime coefficients.
 Thus $f$ and $\Pi_i$ are associates over
that field. Clearing denominators in their primitive forms gives
$f\mid\operatorname{adj}C$ in the full polynomial ring.
Since $f$ is coprime to $LH_*$, it divides $K$ before specialization.
Cancel every parameter factor to obtain
\begin{equation}\label{p2:five:spatial-residual}
 K_0(\ldots,V^{[s]};V)=E_0(V)\mathcal N,
 \qquad e_0:=\deg E_0\le e.
\end{equation}
Each earlier-argument degree of $K_0$ is
$(2-\deg_{Z_i}L)-(1-\deg_{Z_i}L)=1$; the final argument is linear as well.

\subsection{The next extraction and its matrix identity}

For $s>1$, put $\tau=s/5$.
Differentiate $E_0\mathcal D\in(X^s,Y^s,Z^s)$ and eliminate the
undifferentiated term. Coefficient comparison gives
$E_0^2f^5\in(X^s,Y^s,Z^s)$ for every scalar coefficient entry $f$ of $C$.
The expansion by exponent residues modulo five, homogenization, and coefficient choice used
above produce a common multiplier $G$, nonzero at all short vertices,
such that
\begin{equation}\label{p2:five:recursive-extraction}
 GC_I=\sum_aV_a^\tau C'_{Ia},\qquad
 \gamma:=\deg G\le\frac{2e_0}{5},\qquad
 v:=\deg C'_{Ia}=w+\gamma-\tau.
\end{equation}
Cancel the greatest common factor of all the new coefficient entries
from both sides. By the old
full-ring primitivity it divides $G$; this only decreases $\gamma,v$
and preserves the nonvanishing at short vertices. From \eqref{p2:five:residual},
\begin{align}
 \gamma&\le(s-1-2\ell)/5<\tau,\nonumber\\
 v&\le(s-w-2\ell-2)/5
 \le(5s-11-12\ell)/30<5\tau/6.\label{p2:five:recursive-gaps}
\end{align}
Consequently $v<w$, $w+v<s$, and $3v<s$.

On the full actual diagonal the new coefficients $\widehat A_i$ have
degree
\[
 h=\frac54(S-\tau)+v<\frac54S.
\]
The detector already proved applies, since
\[
 3h+3q/P^2<75S/4<m_0,\qquad 3h<15S/4<q/P^2.
\]
It makes the actual output cubic vanish.
If the actual evaluation were singular, its local determinant at a
short point would be $\mu^2$ times a nonzero scalar. Thus $\mu=0$ in
the detector's notation, so every $\widehat A_i^{[5]}\mathsf V$ is radial.
Put $a=A^{[5]}V$ and $b_i=\widehat A_i^{[5]}V$. The three minors
$a_j(b_i)_k-a_k(b_i)_j$ vanish at every short vertex and have degree
\[
 5(S+r+h)+2<\frac{155}{12}S+\frac13<m_0,
\]
using $r\le(S-1)/3$. Interpolation makes them identically zero.
Degree-two independence over the fifth-power field of rational functions in $V$ replaces
$V$ by an independent test vector; since $A$ is invertible, each actual
$\widehat A_i$ is a rational scalar multiple of $A$.
For any two matrix positions $a,b$, the expressions
$(C'_i)_aC_b-(C'_i)_bC_a$ have coefficient degree $w+v<s$ and degree
at most two in each older argument. Lemma~\ref{p2:II:separation} makes
them zero before substitution. Primitivity then gives
$C'_i=\eta_i C$ with $\eta_i$ polynomial, but
$\deg_V\eta_i=v-w<0$, a contradiction.
Thus the new evaluation is nonsingular and its coefficient dimension
is three.

The vanishing output cubic promotes to the formal matrix polynomial: each earlier-argument
degree is three, below five, and its residual degree in $V$ is $3v<s$.
The cubic criterion proves the new last-argument common-row property.
Old argument dimensions are retained because specializing the new argument
recovers $GC$. In fact the new partial matrix is exactly $GC$, which
retains coefficient dimension three in every old argument, as required by
\eqref{p2:five:invariant}.

To restore the crossed identity, first cancel the multiplier $G$.
The product of the old $C$ with the new partner $\mathcal D'$ specializes
in the newest old argument to $G^5\phi I$. Before that specialization,
its off-diagonal entries and diagonal differences have degree in $V$
\[
 w+5v+1=6w+5\gamma-s+1
 \le6w+2e_0-s+1\le s-2\ell-1<s.
\]
The relevant argument degree is two. Separation makes these entries vanish
formally. Multiplying by $G$ gives $U'\mathcal D'=\phi'I$.
The scalar on the full diagonal is determined by
\[
 \mathcal D'_\Delta V^{[5S]}=(C'_\Delta)^{[5]}V,
 \qquad C'_\Delta=\Xi'A,
\]
and $AA^{[5]}V=FV^{[5S]}$. Hence
$\phi'_\Delta=\Xi'^6F$, including its nonvanishing at every short
vertex. This restores the scalar identity needed by the next residual
factorization, as well as the matrix identity.
Thus all parts of the recursive invariant, including full primitivity
and the specialized earlier-argument dimension condition, are restored.
Notice that the unnecessary larger bound $\tau+6v+1<s$ was never used.

For $S\ge5$, iterate to $s=5$. The residual integer interval gives
$4/6\le w\le4/3$, hence $w=1$. If $S=5$, the first strong interval
already gives $r=1$; if $S=1$, the initial degree constraints give
$t=0$, and $M$ is linear. We obtain in all cases a full constant matrix polynomial, linear in each vector argument, with
\begin{equation}\label{p2:five:endpoint}
 \Theta M=\mathsf C(V,V^{[5]},\ldots,V^{[S]}),
 \qquad\deg(\Theta M)=\frac{5S-1}{4}.
\end{equation}

\subsection{Constructing the representing space and recovering \texorpdfstring{$B$}{B}}

The common-function result of
\cite[Theorem~\Iref{I:endpoint-theorem}]{p2:PUB1} applies: its argument degrees
are at most two, below five.
They give multilinear factors $\mathcal F,\mathcal G$ satisfying
\[
 \det\mathsf C=\mathcal F^2\mathcal G,
 \qquad \operatorname{adj}\mathsf C=\mathcal F\mathsf B,
 \qquad D=\Theta^3\mathcal F_\Delta.
\]
The restrictions to two-dimensional normal planes then give a common rational function
\[
 W=\frac{\mathcal F_\Delta^5}{\Theta^5H},\qquad\deg W=d-1,
 \qquad h_a|_{L_a}=\beta_aW|_{L_a}
\]
on every ordinary component. It is regular and nonzero at all shorts.
The auxiliary choices may enlarge the coefficient field. The argument
in Paper I uses coefficients raised to $P^i$th powers for all integers
$i$; only the original rational normal planes need period $n$.

Choose compatible ordinary representatives $u_a=\lambda_a a$ with
$\lambda_a^4=\beta_a^5$. The preceding local semilinear fixed-space\TermNote{exp:p2:semilinear}{the solution space of $Aw^{[5]}=w$, which is a vector space over $\F_5$}
argument places the six lifts on every short line in a two-dimensional
$\mathbb F_5$-space. Two adjacent lifts have rational sum-direction,
so $\lambda_a/\lambda_b\in\mathbb F_q^*$.
The short-adjacency graph has minimum degree at least $5m_0$, and
\begin{equation}\label{p2:five:connectivity}
 10m_0-|B|>
 \left(\frac{373}{200}-\frac{127}{100}\right)q
 =\frac{119}{200}q>0.
\end{equation}
It is connected; one common rescaling therefore puts every ordinary
lift in $\mathbb F_q^3$.

Let $r_U$ be their $\mathbb F_5$-span dimension. Every nonzero linear functional is nonzero at at least
$\delta=4m_0+1$ representatives, and the short lines supply the
collinear triples needed for the dimension bound. Proposition~\Iref{I:code-bound} of~\cite{p2:PUB1} gives
\[
 5^{r_U}\le\frac{5N_0(4m_0+1)}{2+10m_0-N_0}.
\]
Here $\delta\le4(q/P+1)+1=100S+5<q$,
$N_0\le |B|<127q/100$, and \eqref{p2:five:connectivity} bounds the denominator
below by $119q/200$. Thus
\[
 5^{r_U}<\frac{1270}{119}q<25q=5^{n+2}.
\]
Hence $r_U\le n+1$. Since $N_0>q$, dimension at most $n$ is impossible
even after allowing projective fibers, so $r_U=n+1$.

Put $B'=L(U)$ for this span. Dimension counting makes $B'$ blocking;
its line intersections are $1\pmod5$, it contains all ordinary points
of $B$, and $|B'|\le(5q-1)/4$. The comparison must allow $|B|$ to exceed
this last bound. Nevertheless the always valid inequality is
\[
 |B\mathbin\triangle B'|\le\frac{5q-1}{4}-|B|+2\sigma.
\]
Using the exact special-component estimate in
\eqref{p2:five:ordinary-count}, then \eqref{p2:five:t-range} and
\eqref{p2:five:lower}, gives
\begin{align*}
 |B\mathbin\triangle B'|
 &\le\frac{49S}{4}-3\rho-4t-\frac{13}{4}\\
 &\le\frac{49S-13}{4}<200S+2=\frac{8q}{5}+2.
\end{align*}
Lemma~\Iref{I:signed} of~\cite{p2:PUB1} says that a nonzero such difference has size at least
$8q/5+2$ (if only one support is nonempty the bound is stronger).
Therefore $B=B'$.
The short six-secants give exponent exactly one and force
$\mathbb F_5$ to be the maximum geometric field of linearity, by the subline
containment argument at the end of Section~\ref{p2:sec:reconstruction}.
This proves Theorem~\ref{p2:thm:five-entire-interval} and completes the proof of Theorem~\ref{p2:II:main}.

\section{The upper endpoint is attained}\label{p2:secmap:10}\label{p2:sec:upper}

The size bound implied by Theorem~\ref{p2:II:main} is the elementary
cardinality bound for a rank-$(n+1)$ linear set. Partitioning the
$P^{n+1}-1$ nonzero vectors into projective-point fibers\TermNote{exp:p2:fiber}{all nonzero vectors of $U$ representing one projective point; weight one means that these are exactly $P-1$ vectors}, each of
size at least $P-1$, gives
\[
 |L(U)|\le\frac{P^{n+1}-1}{P-1}=b_{\max}.
\]
Equality holds exactly when every point has weight one, meaning
$\dim_{\F_P}(U\cap\F_q v)=1$ for every represented point $[v]$.
\label{exp:p2:scattered-first}Such a linear set is called scattered.

For every $P,n$ in the theorem, define
\[
 U_{\max}=\{(x,x^P,t):x\in\F_q,\ t\in\F_P\}\le\F_q^3.
\]
This is an $\F_P$-space of dimension $n+1$. Suppose that
two nonzero vectors in it differ by a scalar $\lambda\in\F_q^*$.
If their first coordinate is nonzero, comparison of the second
coordinates gives $\lambda^P=\lambda$. If the first coordinate
is zero, the third coordinate is nonzero and comparison there
gives $\lambda\in\F_P^*$. Thus every point has weight one and
\[
 \boxed{|L(U_{\max})|=\frac{Pq-1}{P-1}.}
\]
Equivalently, its affine part is the graph
$\{(x,x^P,1):x\in\F_q\}$ of size $q$, and at infinity it has
$(q-1)/(P-1)$ points. This is the classical R\'edei-type construction
from $x\mapsto x^P$. It is small for $P\ge3$, since
$(Pq-1)/(P-1)<3(q+1)/2$ in that range.

The rank argument shows that $L(U_{\max})$ blocks every line.
All its line intersections are $1\pmod P$. If a point had no
tangent, its pencil would contain at least $1+P(q+1)$ points,
contrary to the displayed size. Hence the set is minimal.
The line $Y=X$ contains the $P$ affine points with $x\in\F_P$
and the single point $(1,1,0)$ at infinity. It is a $(P+1)$-secant,
so the exponent is exactly $e$.

The theorem also shows that any blocking set attaining the upper
endpoint under its hypotheses admits a scattered rank-$(n+1)$
$\F_P$-linear representation\TermNote{exp:p2:scattered}{every represented point has weight one}. This does not assert that all such
sets are projectively equivalent to the example above.

\section{A known strictly intermediate example}\label{p2:secmap:11}\label{p2:sec:interior}
The following size and construction belong to the classical examples
of Polito--Polverino~\cite[Theorem~1.1 and Section~3]{p2:PolitoPolverino};
see also~\cite{p2:BonoliPolverino}. We give the vector-space count to
explain why a proof cannot restrict attention to two endpoint sizes.

There are minimal blocking sets whose size lies strictly inside
the interval established here. Let $n=4$, $q=P^4$, and choose
$\alpha\in\F_q\setminus\F_{P^2}$. Put
\[
 A=\F_{P^2},\qquad C=\Span_{\F_P}\{1,\alpha\},\qquad
 U=A\times C\times\F_P\le \F_q^3.
\]
Then $\dim_{\F_P}U=5$, so $L(U)$ blocks $\PG(2,q)$ by the
dimension argument used above. Its affine part $Z\ne0$ consists
of the $P^4$ points $(a,c,1)$ with $a\in A,c\in C$.

At infinity the represented vector space is $A\times C$.
The two coordinate-axis points each have weight two. Every other
represented point has weight one. Indeed, a non-axis point of
weight two would require $A$ and $C$ to be scalar-equivalent.
If $C=\lambda A$, then $1\in C$ implies $\lambda\in A^*$, hence
$C=A$, contrary to $\alpha\notin A$.
Counting nonzero vectors in $A\times C$ therefore gives
\[
 |L(U)\cap\{Z=0\}|
 =2+\frac{P^4-1-2(P^2-1)}{P-1}
 =P^3+P^2-P+1.
\]
Consequently
\[
 \boxed{|L(U)|=P^4+P^3+P^2-P+1,}
\]
strictly between $P^4+P^3+1$ and $P^4+P^3+P^2+1$, hence
also below the present upper bound $P^4+P^3+P^2+P+1$.
Every line meets $L(U)$ in $1\pmod P$ points. If a point of
$L(U)$ had no tangent, its pencil would force at least
$1+P(q+1)$ points, more than this displayed size. Thus $L(U)$
is minimal. The line $Y=X$ is a $(P+1)$-secant, since
$A\cap C=\F_P$; its exact exponent is $e$.

This example shows why an interval theorem cannot be proved by
asserting that only the two endpoint cardinalities occur. The
theorem establishes linearity; it does not assert that every
integer $\rho$ between $0$ and $\rho_{\max}$ is realizable.

\section*{Acknowledgements}
OpenAI ChatGPT assisted with mathematical exploration, computations and
drafting. The author retains responsibility for verifying the
proofs and the submitted text.

\label{part2:last}

\begin{thebibliography}{9}
\bibitem{p0:Part1} P. Sziklai, \emph{The Linearity Conjecture, Part~1: Frobenius
methods for small minimal blocking sets and binary linearity in
exponent one}, this collection, pp.~\pageref{part1:first}--\pageref{part1:last}.
\bibitem{p0:Part2} P. Sziklai, \emph{The Linearity Conjecture, Part~2: The planar
Linearity Conjecture for intersection exponents $P\ge5$}, this
collection, pp.~\pageref{part2:first}--\pageref{part2:last}.
\end{thebibliography}

\begin{thebibliography}{99}
\bibitem{p1:Szonyi} T. Sz\H{o}nyi, \emph{Blocking sets in Desarguesian affine and projective planes}, Finite Fields Appl. \textbf{3} (1997), 187--202. \href{https://doi.org/10.1006/ffta.1996.0176}{doi:10.1006/ffta.1996.0176}.
\bibitem{p1:SW} T. Sz\H{o}nyi and Zs. Weiner, \emph{Small blocking sets in higher dimensions}, J. Combin. Theory Ser. A \textbf{95} (2001), 88--101. \href{https://doi.org/10.1006/jcta.2000.3152}{doi:10.1006/jcta.2000.3152}.
\bibitem{p1:Sziklai} P. Sziklai, \emph{On small blocking sets and their linearity}, J. Combin. Theory Ser. A \textbf{115} (2008), 1167--1182. \url{https://www.sciencedirect.com/science/article/pii/S0097316508000149}.
\bibitem{p1:DSc} P. Sziklai, \emph{Applications of Polynomials over Finite Fields}, doctoral dissertation, Hungarian Academy of Sciences, Budapest, 2013; especially Section 11.5, Theorem 11.33 and Corollaries 11.35 and 11.37. \url{https://real-d.mtak.hu/617/7/dc_637_12_doktori_mu.pdf}.
\bibitem{p1:Blokhuis} A. Blokhuis, \emph{On the size of a blocking set in
$\PG(2,p)$}, Combinatorica \textbf{14} (1994), 111--114.
\href{https://doi.org/10.1007/BF01305953}{doi:10.1007/BF01305953}.
\bibitem{p1:Polverino} O. Polverino, \emph{Small blocking sets in
$\PG(2,p^3)$}, Des. Codes Cryptogr. \textbf{20} (2000), 319--324.
\bibitem{p1:PolverinoStorme} O. Polverino and L. Storme,
\emph{Small minimal blocking sets in $\PG(2,q^3)$}, European J. Combin.
\textbf{23} (2002), 83--92.
\bibitem{p1:BBBSS} A. Blokhuis, S. Ball, A. E. Brouwer, L. Storme and
T. Sz\H{o}nyi, \emph{On the number of slopes of the graph of a function
defined on a finite field}, J. Combin. Theory Ser. A \textbf{86}
(1999), 187--196.
\bibitem{p1:Ball} S. Ball, \emph{The number of directions determined by a function over a finite field}, J. Combin. Theory Ser. A \textbf{104} (2003), 341--350. \href{https://doi.org/10.1016/j.jcta.2003.09.006}{doi:10.1016/j.jcta.2003.09.006}.
\bibitem{p1:Missing} P. Sziklai, \emph{Functions determining few directions:
the final missing case}, arXiv:2610.28002, 2026.
\url{https://arxiv.org/abs/2610.28002}.
\bibitem{p1:DBVminimum} J. De Beule and G. Van de Voorde, \emph{The minimum size of a linear set}, J. Combin. Theory Ser. A \textbf{164} (2019), 109--124. \href{https://doi.org/10.1016/j.jcta.2018.12.008}{doi:10.1016/j.jcta.2018.12.008}.
\bibitem{p1:JenaVDV} D. Jena and G. Van de Voorde,
\emph{The geometric field of linearity of linear sets},
Des. Codes Cryptogr. \textbf{90} (2022), 779--799.
\href{https://doi.org/10.1007/s10623-022-01011-9}{doi:10.1007/s10623-022-01011-9}.
\bibitem{p1:PUB2} P. Sziklai, \emph{The Linearity Conjecture, Part~2:
The planar Linearity Conjecture for intersection exponents $P\ge5$},
this collection, pp.~\pageref{part2:first}--\pageref{part2:last}.
\bibitem{p1:VDV} G. Van de Voorde, \emph{On the linearity of higher-dimensional blocking sets}, Electron. J. Combin. \textbf{17} (2010), Research Paper 174. \url{https://www.combinatorics.org/ojs/index.php/eljc/article/download/v17i1r174/pdf/}.
\bibitem{p1:SziklaiSzonyi} P. Sziklai and T. Sz\H{o}nyi,
\emph{Blocking sets and algebraic curves}, Rend. Circ. Mat. Palermo
(2) Suppl. \textbf{51} (1998), 71--86.
\bibitem{p1:LSVcode} M. Lavrauw, L. Storme and G. Van de Voorde,
\emph{On the code generated by the incidence matrix of points and
$k$-spaces in $\PG(n,q)$ and its dual},
Finite Fields Appl. \textbf{14} (2008), 1020--1038.
\url{https://arxiv.org/abs/1201.3291}.
\bibitem{p1:LidlNiederreiter} R. Lidl and H. Niederreiter,
\emph{Finite Fields}, second edition, Encyclopedia of Mathematics
and its Applications 20, Cambridge University Press, 1997.
\bibitem{p1:LSV} M. Lavrauw, L. Storme and G. Van de Voorde,
\emph{A proof of the linearity conjecture for $k$-blocking sets in
$\PG(n,p^3)$, $p$ prime}, J. Combin. Theory Ser. A \textbf{118}
(2011), 808--818. \url{https://arxiv.org/abs/1201.3296}.
\bibitem{p1:Lunardon} G. Lunardon, \emph{Normal spreads}, Geom. Dedicata
\textbf{75} (1999), 245--261.
\href{https://doi.org/10.1023/A:1005052007006}{doi:10.1023/A:1005052007006}.
\bibitem{p1:LunardonK} G. Lunardon, \emph{Linear $k$-blocking sets},
Combinatorica \textbf{21} (2001), 571--581.
\href{https://doi.org/10.1007/s004930100013}{doi:10.1007/s004930100013}.
\end{thebibliography}

\begin{thebibliography}{99}
\bibitem{p2:DSc} P. Sziklai, \emph{Applications of Polynomials over Finite Fields}, doctoral dissertation, Hungarian Academy of Sciences, Budapest, 2013; especially Section 11.5, Theorem 11.33 and Corollaries 11.35 and 11.37. \url{https://real-d.mtak.hu/617/7/dc_637_12_doktori_mu.pdf}.
\bibitem{p2:Blokhuis} A. Blokhuis, \emph{On the size of a blocking set in
$\PG(2,p)$}, Combinatorica \textbf{14} (1994), 111--114.
\href{https://doi.org/10.1007/BF01305953}{doi:10.1007/BF01305953}.
\bibitem{p2:Szonyi} T. Sz\H{o}nyi, \emph{Blocking sets in Desarguesian affine and projective planes}, Finite Fields Appl. \textbf{3} (1997), 187--202. \href{https://doi.org/10.1006/ffta.1996.0176}{doi:10.1006/ffta.1996.0176}.
\bibitem{p2:Polverino} O. Polverino, \emph{Small blocking sets in
$\PG(2,p^3)$}, Des. Codes Cryptogr. \textbf{20} (2000), 319--324.
\bibitem{p2:PolverinoStorme} O. Polverino and L. Storme,
\emph{Small minimal blocking sets in $\PG(2,q^3)$}, European J. Combin.
\textbf{23} (2002), 83--92.
\bibitem{p2:LSV} M. Lavrauw, L. Storme and G. Van de Voorde,
\emph{A proof of the linearity conjecture for $k$-blocking sets in
$\PG(n,p^3)$, $p$ prime}, J. Combin. Theory Ser. A \textbf{118}
(2011), 808--818. \url{https://arxiv.org/abs/1201.3296}.
\bibitem{p2:BBBSS} A. Blokhuis, S. Ball, A. E. Brouwer, L. Storme and
T. Sz\H{o}nyi, \emph{On the number of slopes of the graph of a function
defined on a finite field}, J. Combin. Theory Ser. A \textbf{86}
(1999), 187--196.
\bibitem{p2:Ball} S. Ball, \emph{The number of directions determined by a function over a finite field}, J. Combin. Theory Ser. A \textbf{104} (2003), 341--350. \href{https://doi.org/10.1016/j.jcta.2003.09.006}{doi:10.1016/j.jcta.2003.09.006}.
\bibitem{p2:Missing} P. Sziklai, \emph{Functions determining few directions:
the final missing case}, arXiv:2610.28002, 2026.
\url{https://arxiv.org/abs/2610.28002}.
\bibitem{p2:JenaVDV} D. Jena and G. Van de Voorde,
\emph{The geometric field of linearity of linear sets},
Des. Codes Cryptogr. \textbf{90} (2022), 779--799.
\href{https://doi.org/10.1007/s10623-022-01011-9}{doi:10.1007/s10623-022-01011-9}.
\bibitem{p2:SW} T. Sz\H{o}nyi and Zs. Weiner, \emph{Small blocking sets in higher dimensions}, J. Combin. Theory Ser. A \textbf{95} (2001), 88--101. \href{https://doi.org/10.1006/jcta.2000.3152}{doi:10.1006/jcta.2000.3152}.
\bibitem{p2:Sziklai} P. Sziklai, \emph{On small blocking sets and their linearity}, J. Combin. Theory Ser. A \textbf{115} (2008), 1167--1182. \url{https://www.sciencedirect.com/science/article/pii/S0097316508000149}.
\bibitem{p2:PUB1} P. Sziklai, \emph{The Linearity Conjecture, Part~1:
Frobenius methods for small minimal blocking sets and binary linearity
in exponent one}, this collection, pp.~\pageref{part1:first}--\pageref{part1:last}.
\bibitem{p2:VDV} G. Van de Voorde, \emph{On the linearity of higher-dimensional blocking sets}, Electron. J. Combin. \textbf{17} (2010), Research Paper 174. \url{https://www.combinatorics.org/ojs/index.php/eljc/article/download/v17i1r174/pdf/}.
\bibitem{p2:Bruen} A. A. Bruen, \emph{Baer subplanes and blocking sets},
Bull. Amer. Math. Soc. \textbf{76} (1970), 342--344.
\href{https://doi.org/10.1090/S0002-9904-1970-12470-3}{doi:10.1090/S0002-9904-1970-12470-3}.
\bibitem{p2:PolitoPolverino} P. Polito and O. Polverino,
\emph{Linear blocking sets in $\PG(2,q^4)$}, Australas. J. Combin.
\textbf{26} (2002), 41--48.
\url{https://ajc.maths.uq.edu.au/pdf/26/ajc_v26_p041.pdf}.
\bibitem{p2:BonoliPolverino} G. Bonoli and O. Polverino,
\emph{$\F_q$-linear blocking sets in $\PG(2,q^4)$}, Innov. Incidence
Geom. \textbf{2} (2005), 35--56.
\url{https://msp.org/iig/2005/2-1/iig-v2-n1-p02-p.pdf}.
\end{thebibliography}
\end{document}